\documentclass[a4paper,12pt,reqno]{amsart}
\usepackage{amssymb,amsmath,amsthm}
\usepackage{esint,hyperref}
\usepackage[utf8]{inputenc}
\numberwithin{equation}{section}
\allowdisplaybreaks

\newcommand{\dd}{\mathop{}\!\mathrm{d}}

\newtheorem{thm}{Theorem}[section]
\theoremstyle{definition}
\newtheorem{example}[thm]{Example}
\theoremstyle{plain}
\newtheorem{backgroundthm}{Theorem}

\newtheorem{prop}[thm]{Proposition}
\newtheorem{lem}[thm]{Lemma}
\newtheorem{cor}[thm]{Corollary}
\newtheorem{defn}[thm]{Definition}
\newtheorem*{question}{Question}
\newcounter{step}
\newenvironment{step}[1][]
{\refstepcounter{step}\par\medskip\noindent
	\textbf{Step~\thestep.}\hspace*{0.5em}\textit{#1}\par\smallskip}
{\par\medskip}

\title{Ricci-pinched 3-Manifolds with second-fundamental-form-pinched Boundary}

\author{Chao Xia}
\address{(C.X.) School of Mathematical Sciences,
Xiamen University, 361005, Xiamen, P.~R.~China}
\email{chaoxia@xmu.edu.cn}
\thanks{C.X. is supported by NSFC (Grant No. 12271449, 12671069, 12526203, 12526102) and the Natural Science Foundation of Fujian Province of China (Grant No. 2024J011008).}

\author{Jincong Zheng}
\address{(J.Z.) School of Mathematical Sciences,
Xiamen University, 361005, Xiamen, P.~R.~China}
\email{zhengjincong@stu.xmu.edu.cn}

\begin{document}
	\begin{abstract}
		We prove a boundary analogue of the rigidity theorem for
		Ricci-pinched three-manifolds under a superquadratic volume-growth
		assumption.  More precisely, we show that a complete connected
		three-manifold with convex boundary, pinched Ricci tensor,
		and pinched second fundamental form is isometric to the Euclidean
		half-space.  The proof is potential-theoretic.  We construct and
		study mixed Dirichlet--Neumann $p$-capacitary potentials, establish
		monotonicity formulas for their free-boundary level sets, and use a
		weak Gauss--Bonnet formula adapted to this setting.  
	\end{abstract}

	\maketitle
	
	\section{Introduction}
	A Riemannian manifold $(M,g)$ is called \emph{Ricci-pinched} if
	$\operatorname{Ric}\ge0$ and, for some $\rho>0$,
	\begin{equation}\label{eq:intro-Ricci-pinching}
		\operatorname{Ric}\ge \rho\,\mathrm{R}g.
	\end{equation}
	Here $\operatorname{Ric}$ and $\mathrm R$ denote the Ricci and scalar
	curvatures.  Hamilton's Ricci-flow theorem for compact
	three-manifolds with positive Ricci curvature~\cite{Ham82} led to the
	following scale-invariant flat-or-compact conjecture; see
	\cite[Conjectures~1.1 and~1.2]{Lot24} for its formulation and
	attribution. Hamilton's pinching conjecture was resolved through a sequence of
	Ricci-flow developments.  Chen--Zhu~\cite{CZ00} proved the
	flat-or-compact conclusion under the additional assumptions of
	bounded curvature and nonnegative sectional curvature.
	Lott~\cite{Lot24} retained bounded curvature but relaxed
	nonnegativity of sectional curvature to an inverse-quadratic lower
	bound, allowing a negative part that decays quadratically at infinity.
    Building on Lott's work, Deruelle--Schulze--Simon~\cite{DSS25}
	removed the additional lower sectional-curvature bound and proved
	the conjecture assuming only bounded curvature. Finally,
	Lee--Topping~\cite{LT25} removed the remaining bounded-curvature assumption by constructing a complete Ricci flow from arbitrary complete noncompact Ricci-pinched initial data, thereby completely proving the following result.
	
	\begin{backgroundthm}\label{thm:intrinsic-background}
		Let $(M,g)$ be a complete, connected $3$-manifold without
		boundary.  Suppose that $M$ is \emph{Ricci-pinched}.  Then $M$
		is flat or compact.
	\end{backgroundthm}
	
	This is a scale-invariant counterpart of Bonnet--Myers: no uniform
	positive lower bound for $\operatorname{Ric}$ is assumed.  Instead,
	\eqref{eq:intro-Ricci-pinching} says that at every point either
	$\operatorname{Ric}=0$, or all Ricci eigenvalues are positive and
	quantitatively comparable.
	Under additional global assumptions, the same rigidity has also been
	obtained through several different approaches.  Under the additional
	assumption that the manifold
	has a pole, Xu~\cite{Xu24} proved that a complete noncompact
	Ricci-pinched $3$-manifold is flat.  Under maximal (Euclidean) volume
	growth,
	Huisken--Koerber~\cite{HK24} used weak inverse mean curvature flow,
	while Benatti--Mantegazza--Oronzio--Pluda~\cite{BMPO} used a harmonic
	potential; more recently, Chen--Xu--Zhang~\cite[Corollary~1.10]{CXZ26}
	obtained the same rigidity from an integral formula involving the
	Green's function.  Building on the harmonic-potential approach,
    Benatti--Le\'on Quir\'os--Oronzio--Pluda~\cite{BPO} developed a
    nonlinear $p$-capacitary argument applicable under superquadratic
    volume growth.  
	Higher-dimensional results are naturally formulated using PIC1
	pinching.  Under the additional assumption of positive asymptotic
	volume ratio, Deruelle--Schulze--Simon~\cite[Theorem~1.3]{DSS24}
	proved that every complete PIC1-pinched manifold is isometric to
	Euclidean space.  The full flat-or-compact rigidity was resolved through
	the work of
	Lee--Topping~\cite{LTPIC1} and its subsequent extension by
	Deruelle--Lee--Schulze--Simon--Topping~\cite{DLSST26}.
	
	The extrinsic analogue replaces the Ricci tensor by the second
	fundamental form.  Hamilton~\cite{Ham94} used the quasiconformal
	Gauss map to prove a compactness theorem for strictly convex
	hypersurfaces with pinched second fundamental form.
	Ni~\cite{Ni25} gave a detailed alternative proof of Hamilton's
	compactness theorem for strictly convex hypersurfaces.
	
	\begin{backgroundthm}[Hamilton]\label{thm:extrinsic-background}
		Let $\Sigma^n\subset\mathbb{R}^{n+1}$, $n\ge2$, be a smooth,
		complete, connected, strictly convex hypersurface without boundary,
		bounding a region.  Suppose that
		\[
		\operatorname{II}\ge\lambda H g
		\]
		for some constant $\lambda>0$, where $g$ is the induced metric
		and $H=\operatorname{tr}_g\operatorname{II}$.  Then $\Sigma$ is compact.
	\end{backgroundthm}

	An alternative approach uses mean curvature flow.
	Bourni--Langford--Lynch~\cite{BLL23} gave such a proof under
	the additional assumption of bounded curvature.
	For hypersurface dimensions $n\ge3$, Cheng--Lu~\cite{CL26}
	constructed the flow without assuming bounded initial curvature,
	thereby removing this assumption from the flow approach.\par
    The natural boundary
analogue of Bonnet--Myers should instead impose a lower bound on the
interior Ricci curvature together with a convexity condition on the
boundary, and should conclude compactness of $M$.  This is precisely
the content of a recent theorem of Yan--Zhu~\cite{YZ26}.  Confirming a
conjecture of M.~Li~\cite[Conjecture~1.3]{Li14}, they prove that if
$(M^n,g)$ is a complete Riemannian $n$-manifold with nonempty smooth
boundary satisfying
\[
\operatorname{Ric}_g\ge0,
\qquad \operatorname{II}\ge1\cdot g
\quad\text{on }\partial M,
\]
then $M$ is compact
\cite[Theorem~1.2]{YZ26}. \par
	Motivated simultaneously by these intrinsic and extrinsic theorems,
	we consider a complete Riemannian $3$-manifold with boundary, impose
	Ricci pinching in the interior, and impose second fundamental form pinching  on
	the boundary.  More precisely, when $\partial M\ne\emptyset$, we call
	$(M,\partial M,g)$ \emph{second-fundamental-form-pinched} if
	$\operatorname{II}\ge0$ and
	\begin{equation}\label{eq:intro-II-pinching}
		\operatorname{II}\ge\lambda H g
		\qquad\text{on }\partial M
	\end{equation}
	for some $\lambda>0$, where
	$\operatorname{II}(X,Y)=\langle\nabla_XY,\mathbf n\rangle$,
	$H=\operatorname{tr}\operatorname{II}$, and $\mathbf n$ is the inward
	unit normal.  This leads to the following boundary analogue of
	Hamilton's pinching problem.

	\begin{question}
		Let $(M,\partial M,g)$ be complete and connected, with
		$\dim M=3$ and nonempty convex boundary.  If $(M,g)$ is
		Ricci-pinched and $(M,\partial M,g)$ is
		second-fundamental-form-pinched, must $M$ be compact or flat?
	\end{question}

	   We prove that
	superquadratic volume growth not only yields flatness in the noncompact
	case, but also rules out such collapsed flat quotients and singles out
	$\mathbb R^3_+$.  We recall the precise growth condition, allowing an
	empty boundary because the same notion is used in the boundaryless
	results above.

	\begin{defn}
		Let $(M,\partial M,g)$ be a complete, connected, noncompact
		Riemannian $3$-manifold with convex boundary, possibly empty, and
		$\operatorname{Ric}\ge0$.  We say $(M,\partial M,g)$ has
		\emph{superquadratic volume growth} if there exist $q\in M$,
		$C_{\mathrm{vol}}>0$, $\tilde r>0$, and $\alpha\in(1,2]$ such
		that for all $r>\tilde r$,
		\[
		C_{\mathrm{vol}}^{-1}\,r^{1+\alpha}\le|B_r(q)|
		\le C_{\mathrm{vol}}\,r^{1+\alpha},
		\]
		where $|B_r(q)|$ is the volume of the geodesic ball of radius
		$r$ centered at $q$.
	\end{defn}

	Our main result gives the anticipated dichotomy in the noncompact
	case and identifies its unique noncollapsed flat model.
	
	\begin{thm}\label{thm:main}
		Let $(M,\partial M,g)$ be a complete, connected $3$-manifold
		with nonempty convex boundary, Ricci-pinched and
		second-fundamental-form-pinched, with superquadratic volume
		growth.  Then $(M,\partial M,g)$ is isometric to the Euclidean
		half-space
		\[
		\mathbb{R}^3_+
		=\{\,x=(x_1,x_2,x_3)\in\mathbb{R}^3:x_3\ge0\,\},
		\]
		equipped with the standard flat metric.
	\end{thm}

	\begin{cor}
		Under the assumptions of Theorem~\ref{thm:main}, the boundary
		$\partial M$ is connected, noncompact, and totally geodesic.  In
		particular, there is no complete connected $3$-manifold with compact
		nonempty convex boundary satisfying the two pinching conditions and
		superquadratic volume growth.
	\end{cor}
	\begin{proof}
		By Theorem~\ref{thm:main}, $\partial M$ is isometric to
		$\partial\mathbb{R}^3_+=\mathbb{R}^2$, and its second fundamental form
		vanishes identically.
	\end{proof}
    The boundary pinching assumption is essential even for the interior
	flatness conclusion.  One might try to separate the intrinsic and
	extrinsic parts of the problem: first deduce interior flatness from
	Ricci pinching and boundary convexity alone, and then, after establishing
	an isometric realisation as a convex domain in $\mathbb R^3$, use
	boundary pinching to appeal to Euclidean hypersurface rigidity.
	However, we find that the first step is false, even under cubic volume growth and strict
	boundary convexity, see the example in Appendix \ref{sec:eps-appendix-c}.

	Our proof of Theorem \ref{thm:main} builds on the nonlinear potential-theoretic framework
	developed in~\cite{BFM24,BPP} and its application to Ricci-pinched
	three-manifolds in~\cite{BPO}.  In the boundaryless setting,
	superquadratic volume growth ensures the relevant $p$-nonparabolicity
	and hence the existence of proper $p$-capacitary potentials.  The
	level sets of such a potential carry the quantities $F_p$ and $G_p$,
	whose first-variation identities yield a system of monotonicity
	formulas.  Since these level sets need not be globally smooth,
	the approximation theory of~\cite{BPP} supplies a weak Gauss--Bonnet
	theorem for almost every level.  Combining this theorem with the Gauss
	equation and Ricci pinching reduces the geometric monotonicity formulas
	to differential inequalities.  Finally, the Willmore expansion of
	small geodesic spheres provides a strict initial deficit at any point
	of positive scalar curvature.  The resulting ODE estimates are
	incompatible with superquadratic volume growth, forcing the ambient
	manifold to be Ricci-flat.

	Passing from this scheme to manifolds with boundary requires four new
	ingredients.  We first develop the mixed Dirichlet--Neumann potential
	theory and the corresponding monotonicity formulas, including the
	physical-boundary contribution.  We then establish a free-boundary
	weak Gauss--Bonnet formula for the level surfaces.  Third, we replace
	geodesic balls by Fermi coordinate half-balls and compute a Willmore
	expansion that detects the mean curvature of $\partial M$; this is the
	key step in proving that the boundary is totally geodesic.  Finally,
	we analyse the equality case of the monotonicity formula to pass from
	boundary rigidity to $\operatorname{Ric}\equiv0$ in the interior.  This
	last step is inspired by the monotonicity--rigidity proof
	of~\cite{BFM24} and has no counterpart in~\cite{BPO} in this
	free-boundary form.

	We next outline the proof of Theorem~\ref{thm:main}.  Fix an exponent
	$p\in(1,2)$.  Superquadratic volume growth implies
	$p$-nonparabolicity and yields a proper mixed capacity potential
	$w_p$, whose free-boundary level sets carry monotone quantities
	$F_p$ and $G_p$.  If the second fundamental form of $\partial M$ does
	not vanish identically, its pinching produces a point
	$q\in\partial M$ with $H_{\partial M}(q)>0$.  A sufficiently small
	Fermi coordinate half-ball about $q$ then has
	\[
		\int_{S_r^+(q)}H^2\,\dd\mathcal{H}^2<8\pi,
	\]
	and hence $F_p(0)<2\pi$.  Monotonicity gives exponential decay of
	$F_p$ and $G_p$.  Combining this decay with the capacity identity,
	coarea, and the pointwise decay of the potential forces
	$\alpha\le4/(5-p)$.  Since the argument applies to every
	$p\in(1,2)$, letting $p\downarrow1$ contradicts $\alpha>1$.
	Consequently $\partial M$ is totally geodesic.

	Once $\operatorname{II}\equiv0$, small geodesic hemispheres centred on
	$\partial M$ are free-boundary surfaces and have Willmore energy at
	most $8\pi$.  Choose $p>1$ sufficiently close to $1$ that
	$4/(5-p)<\alpha$.  The preceding volume-growth contradiction rules out
	any strict loss in the initial monotone quantity, so the relevant
	$F_p$ is identically equal to $2\pi$.  Its derivative formula then
	forces equality in every nonnegative dissipation term.  The resulting
	equality system excludes critical points of the potential and, using
	the Ricci pinching, gives $\operatorname{Ric}=0$ outside the initial
	half-ball.  Repeating the argument for every sufficiently small radius
	about one fixed boundary point and letting the radius tend to zero
	proves that $M$ is flat.  The smooth flat double of $M$ has the same
	superquadratic growth exponent and is therefore Euclidean; the
	involution exchanging the two copies becomes reflection in an affine
	plane, which identifies $M$ with $\mathbb{R}^3_+$.

	The paper is organised as follows.  Section~2 develops the mixed
	$p$-capacitary potentials and the monotone quantities.  Section~3
	establishes the free-boundary weak Gauss--Bonnet formula and the
	Willmore-type estimates.  Section~4 proves the main theorem, including
	the Fermi coordinate calculation.  Appendix~A contains the
	regularisation and edge estimates,  Appendix~B proves $p$-nonparabolicity and the decay of the capacity potential, while Appendix~C gives an non-flat example on Ricci-pinched $3$-manifolds with convex boundary, but not second-fundamental-form-pinched.

	\medskip\noindent\textbf{Acknowledgements.}
	The authors would like to thank Luca Benatti, Man-Chun Lee, Alessandra Pluda  for their interest and helpful comments.

    \noindent\textbf{Disclosure on AI assistance.}
The authors used AI-assisted tools to write the paper. The authors verified and completed all mathematical
arguments and take full responsibility for the content.

	\section{\texorpdfstring{$p$}{p}-harmonic functions and monotone quantities}
	
	\subsection{\texorpdfstring{$p$}{p}-harmonic functions}
	
	Let $(M,\partial M,g)$ be a complete, connected, noncompact
	Riemannian $3$-manifold with nonempty convex boundary,
	$\operatorname{Ric}\ge0$, and superquadratic volume growth.
	Let $\Omega\subset M$ be a closed bounded set with smooth boundary,
	$\Omega\cap\partial M\neq\emptyset$, and assume that
	$M\setminus\Omega$ is connected.  For every such domain $K\subset M$
	we use the convention
	\[
		\partial_{\operatorname{int}}K
		:=\overline{\partial K\cap\operatorname{int}(M)};
	\]
	where the closure is taken in $M$; thus the interior face includes its edge
	$\partial_{\operatorname{int}}K\cap\partial M$.  We assume that
	$\partial_{\operatorname{int}}\Omega$ meets $\partial M$ orthogonally
	along this edge.
	
	\begin{defn}[$p$-capacity and $p$-nonparabolicity]
		For a closed bounded set $D\subset M$, the normalised
		$p$-capacity of $\partial_{\operatorname{int}}D$ is
		\begin{equation}\label{eq:capacity-defn}
			C_p(\partial_{\operatorname{int}}D)
			= \inf\Bigl\{\,
			\frac{1}{4\pi}\Bigl(\frac{p-1}{3-p}\Bigr)^{p-1}
			\int_{M\setminus D}|\nabla\psi|^{p}\,\dd\mathcal{H}^{3}
			: \psi\in C_c^{\infty}(M),\;\psi\ge\chi_D
			\Bigr\},
		\end{equation}
		where $\chi_D$ is the characteristic function of $D$.
		We say that $(M,\partial M)$ is \emph{$p$-non\-para\-bolic}
		if there exists a closed bounded set $\Omega\subset M$ as
		above such that $C_p(\partial_{\operatorname{int}}\Omega)>0$.
	\end{defn}
	
	Under the assumption that $(M,\partial M)$ is $p$-nonparabolic,
	the condition $C_p(\partial_{\operatorname{int}}\Omega)>0$ is
	equivalent to the existence of a non-trivial solution to the
	following mixed boundary value problem (for $p\in(1,2)$):
	\begin{equation}
		\begin{cases}
			\Delta_p u_p = 0                     & \text{in } M\setminus\Omega,\\[4pt]
			u_p = 1                              & \text{on }
			\partial_{\operatorname{int}}\Omega,\\[4pt]
			\langle\nabla u_p,\mathbf{n}\rangle=0& \text{on }
			\partial M\setminus\Omega,\\[4pt]
			u_p(x)\to0                           & \text{as }
			d(x,o)\to+\infty,\label{ph}
		\end{cases}
	\end{equation}
	where $\Delta_p f=\operatorname{div}(|\nabla f|^{p-2}\nabla f)$,
	$\mathbf{n}$ is the inward unit normal along $\partial M$, and
	$o\in M$ is a fixed reference point.  Although $u_p$ need not belong to
	the smooth compactly supported class in~\eqref{eq:capacity-defn}, the
	variational identification proved in
	Appendix~\ref{sec:pdepart-appendix} shows that the infimum agrees with
	the energy of $u_p$:
	\begin{equation}\label{eq:capacity-def}
		C_p(\partial_{\operatorname{int}}\Omega)
		= \frac{1}{4\pi}\Bigl(\frac{p-1}{3-p}\Bigr)^{p-1}
		\int_{M\setminus\Omega}|\nabla u_p|^{p}\,\dd\mathcal{H}^{3}.
	\end{equation}
	
	For complete manifolds without boundary, Li--Tam~\cite{LiTam}
	developed the corresponding linear theory for Green functions and
	harmonic functions, while Holopainen~\cite{Holopainen} established
	nonlinear $p$-potential-theoretic volume-growth criteria for the
	parabolicity of ends.  We do not invoke those boundaryless criteria
	directly here: the formulation needed for a convex physical boundary,
	including the relevant global and end issues, is proved separately in
	Appendix~\ref{sec:pdepart-appendix}.
	
	For manifolds with convex boundary, the corresponding analytic input is
	most naturally formulated on the metric-measure space
	$(M,d,\mathcal{H}^3)$.  It is an $\operatorname{RCD}(0,3)$ space and hence
	satisfies volume doubling and a weak $(1,p)$-Poincar\'e inequality,
	including on balls meeting $\partial M$.
	Appendix~\ref{sec:pdepart-appendix} combines these
	estimates with the Neumann Harnack inequality and a direct capacity bound
	for balls centred at arbitrary distant points.  The finite mixed problems
	are solved on a nested Lipschitz exhaustion whose artificial boundary is
	transverse to $\partial M$; orthogonality at this moving outer edge is not
	required, because all local estimates are applied on balls disjoint from
	the artificial Dirichlet face.  This gives the following result.
	
	\begin{prop}[$p$-nonparabolicity]\label{prop:p-nonparabolic}
		Under the above assumptions ($\operatorname{Ric}\ge0$, $\partial M$ convex, and superquadratic volume growth $|B_r|\ge C^{-1}r^{1+\alpha}$ with $\alpha>1$), $(M,\partial M)$ is $p$-nonparabolic, i.e.\ $C_p(\partial_{\operatorname{int}}\Omega)>0$.  Consequently, the solution $u_p$ to~\eqref{ph} exists and satisfies the pointwise decay estimate
		\[
		u_p(x)\le C\,d(x,o)^{-\frac{\alpha+1-p}{p-1}}
		\]
		for all $d(x,o)$ sufficiently large, with $C=C(M,\Omega,p)>0$.
	\end{prop}
	
	\begin{proof}
		This is Corollary~\ref{ap:application}.  Its proof uses a lower
		bound for the capacity of balls centred at arbitrary distant
		points, together with the Neumann Harnack inequality.  Thus the
		decay estimate holds along every sequence tending to infinity,
		without requiring a point to lie on a selected unbounded tail.
	\end{proof}
	\begin{lem}\label{lem:regularity}
		Under $\operatorname{Ric}\ge0$ and superquadratic volume growth,
		the solution $u_p$ to~\eqref{ph} satisfies the lower bound
		\begin{equation}
			u_p(x)\ge C^{-1}d(x,o)^{-\frac{3-p}{p-1}}
			\label{eq:u-lower-bound}
		\end{equation}
		for all $d(x,o)$ sufficiently large, with
		$C=C(M,\Omega,p)>0$.  The corresponding upper bound is
		provided by Proposition~\ref{prop:p-nonparabolic}.
	\end{lem}
	
	\begin{proof}
		Choose $q\in\operatorname{int}(M)\cap\operatorname{int}(\Omega)$ and set
		$r(x)=d(x,q)$ and $a=(3-p)/(p-1)$.  Thus $q\notin
		M\setminus\Omega$ and $d(q,\partial_{\operatorname{int}}\Omega)>0$.
		Convexity of $\partial M$ ensures that minimizing geodesics from the
		interior point $q$ to interior points do not leave $M$.  Consequently
		the standard Laplacian comparison under $\operatorname{Ric}\ge0$ gives
		$\Delta r\le2/r$ in the barrier, hence distributional, sense away
		from $q$.  At every smooth point of $r$,
		\[
			\Delta_p(r^{-a})
			=a^{p-1}r^{-a(p-1)-p}
			\bigl((p-1)(a+1)-r\Delta r\bigr)\ge0,
		\]
		because $(p-1)(a+1)=2$; the usual barrier-to-distribution argument
		across the cut locus therefore gives
		\[
			\Delta_p(r^{-a})\ge0\qquad\text{in }M\setminus\Omega.
		\]
		On the physical boundary let $\nu_{\mathrm{out}}=-\mathbf{n}$.  The first
		variation formula for a minimizing segment from the interior point
		$q$ to $x\in\partial M$ gives
		$\partial_{\mathbf{n}}r(x)\le0$ in the barrier sense, and hence
		\[
			\partial_{\nu_{\mathrm{out}}}(r^{-a})
			=a r^{-a-1}\partial_{\mathbf{n}}r\le0.
		\]
		Equivalently, including the physical-boundary contribution, for every
		nonnegative test function
		$\phi\in C_c^\infty(\overline{M\setminus\Omega})$ whose support is
		disjoint from $\partial_{\operatorname{int}}\Omega$,
		\begin{equation}\label{eq:barrier-weak-supersolution}
			\int_{M\setminus\Omega}
			\bigl\langle |\nabla r^{-a}|^{p-2}\nabla r^{-a},
			\nabla\phi\bigr\rangle\,\dd\mu\le0.
		\end{equation}
		This is precisely the weak inequality needed to compare $r^{-a}$
		from below with the solution of the mixed problem; see also
		\cite[Proposition~2.10]{BFM24} for the distance-function computation.

		Since $r$ is bounded away from zero on
		$\partial_{\operatorname{int}}\Omega$, one can choose $c>0$ so small
		that $c r^{-a}\le1=u_p$ there.  Fix $\delta>0$.  Because
		$r^{-a}\to0$ and $u_p\ge0$, a sufficiently large smooth relatively
		bounded exhaustion domain $E_k\subset M\setminus\Omega$,
		whose relative boundary
		consists of the initial face, a physical-boundary part, and an
		artificial outer face, satisfies
		$c r^{-a}\le u_p+\delta$ on its artificial outer face.  Testing the
		difference of the weak inequalities for $c r^{-a}$ and $u_p+\delta$
		with $(c r^{-a}-u_p-\delta)_+$ (justified by approximation in the
		mixed Sobolev test space), the physical-boundary term having the
		favourable sign by \eqref{eq:barrier-weak-supersolution}, and using
		strict monotonicity of $\xi\mapsto|\xi|^{p-2}\xi$, yields
		$c r^{-a}\le u_p+\delta$ in $E_k$.  Letting first $k\to\infty$ and
		then $\delta\downarrow0$ gives
		\[
			c r^{-a}\le u_p\qquad\text{on }M\setminus\Omega.
		\]
		Finally $r(x)$ and $d(x,o)$ are comparable for $d(x,o)$ large by the
		triangle inequality, which proves \eqref{eq:u-lower-bound}.
	\end{proof}
	
	Setting $w_p=-(p-1)\log u_p$, a direct computation gives
	\begin{equation}\label{eq:wp}
		\begin{cases}
			\Delta_p w_p = |\nabla w_p|^{p}  & \text{in } M\setminus\Omega,\\
			w_p = 0                          & \text{on }
			\partial_{\operatorname{int}}\Omega,\\
			\langle\nabla w_p,\mathbf{n}\rangle = 0 & \text{on }
			\partial M\setminus\Omega,\\
			w_p(x)\to+\infty                 & \text{as }
			d(x,o)\to+\infty.
		\end{cases}
	\end{equation}
	Moreover, $\nabla w_p=-(p-1)\nabla u_p/u_p$, and $w_p$ inherits
	the regularity of $u_p$.

	\medskip
	\noindent\textit{Choice of the auxiliary domains.}
	The auxiliary domains used below need only have artificial boundary
	transverse to $\partial M$.  Choose a smooth nonnegative proper
	exhaustion $\rho:M\to[0,+\infty)$.  By Sard's theorem, there are values
	$c_j\to+\infty$ which are regular both for
	$\rho|_{\operatorname{int}(M)}$ and for $\rho|_{\partial M}$.  Fix a
	compact connected neighborhood
	$K\supset\Omega$, and, after increasing $c_1$ so that
	$K\subset\{\rho<c_1\}$, let $D_j$ be the component of
	$\{\rho<c_j\}$ containing $K$.  These components are nested and exhaust
	$M$: every point can be joined to $K$ by a compact path, which is
	contained in $\{\rho<c_j\}$ for all sufficiently large $j$.  Properness
	of $\rho$ makes every $D_j$ relatively compact.

	The artificial boundary
	$\Gamma_j=\partial_{\operatorname{int}}D_j$ is smooth and meets
	$\partial M$ transversely: at an edge point the regularity of
	$\rho|_{\partial M}$ says that the tangential component of
	$\nabla\rho$ is nonzero, so the normal to $\Gamma_j$ is not parallel to
	the normal to $\partial M$.
	Finally, for any prescribed $T,\tau>0$, properness of $w_p$ makes
	$\{w_p\le T+\tau\}$ compact.  Taking $j$ so large that this set is
	contained in $D_j$ gives, by compactness of $\Gamma_j$,
	\[
		\inf_{\Gamma_j}w_p>T+\tau.
	\]
	Thus the sublevel sets of $w_p$ determine how far out the auxiliary
	domain must be taken, but are not themselves used as its artificial
	boundary.  Indeed, at this stage a level of $w_p$ may contain critical
	points and hence need not be a smooth hypersurface.  This is the same
	type of transverse exhaustion used in
	Appendix~\ref{sec:pdepart-appendix} to construct the global potential.
	
	\begin{lem}\label{lem:eps-regularization}
		Let $D\subset M$ be a bounded open set containing $\Omega$, with
		smooth interior boundary and
		$\partial_{\operatorname{int}}D$ transverse to $\partial M$.
		For each $\varepsilon>0$ there exists a weak solution
		$w_p^\varepsilon$ on $D\setminus\Omega$ solving
		\begin{equation}\label{eq:eps-regularized}
			\begin{cases}
				\Delta_p^\varepsilon
				\bigl(e^{-w_p^\varepsilon/(p-1)}\bigr) = 0
				& \text{in } D\setminus\Omega,\\[4pt]
				w_p^\varepsilon = 0
				& \text{on } \partial_{\operatorname{int}}\Omega,\\[4pt]
				\langle\nabla w_p^\varepsilon,\mathbf{n}\rangle = 0
				& \text{on } \partial M\cap(D\setminus\Omega),\\[4pt]
				w_p^\varepsilon = w_p
				& \text{on } \partial D\cap\operatorname{int}(M),
			\end{cases}
		\end{equation}
		where $\Delta_p^\varepsilon f =
		\operatorname{div}\bigl((|\nabla f|^{2}+\varepsilon^{2})^{(p-2)/2}
		\,\nabla f\bigr)$.  The solution is smooth in the interior and up to
		the smooth portions of the physical Neumann face and of the initial
		Dirichlet face away from their intersection.  It is continuous up to
		the artificial outer Dirichlet face and assumes the prescribed trace
		there, but no smoothness along that face is asserted.  As
		$\varepsilon\to0^{+}$ the following holds:
		\begin{enumerate}
			\item\label{item:C1} For some $\beta>0$ the family
			$\{w_p^\varepsilon\}$ is uniformly bounded
			in $C^{1,\beta}(K)$ and converges to $w_p$ in $C^1(K)$ for every
			compact
			$K\subset\overline{D\setminus\Omega}
			\setminus\partial_{\operatorname{int}}D$.  In
			particular, there is a
			neighborhood $U_0$ of $\partial_{\operatorname{int}}\Omega$ such
			that the same bounds and convergence hold in
			$C^{1,\beta}(\overline{U_0\setminus\Omega})$ and
			$C^1(\overline{U_0\setminus\Omega})$, respectively;
			\item\label{item:W12} $\{|\nabla w_p^\varepsilon|\}$ is uniformly bounded in
			$W^{1,2}_{\mathrm{loc}}(D\setminus\Omega)$, and
			$|\nabla w_p^\varepsilon|\rightharpoonup|\nabla w_p|$
			weakly in $W^{1,2}_{\mathrm{loc}}(D\setminus\Omega)$;
			\item\label{item:levels} for almost every
			$t\in(0,\inf_{\partial D}w_p)$, the level set
			$\Sigma_t^\varepsilon=\{w_p^\varepsilon=t\}$ is a smooth
			hypersurface meeting $\partial M$ orthogonally;
		\item\label{item:reg} $w_p\in C^{1,\beta}_{\mathrm{loc}}(D\setminus\Omega)$,
			and $w_p$ is smooth on
			$D\setminus\bigl(\Omega\cup\operatorname{Crit}(w_p)\bigr)$.
		\item\label{item:initial-edge} Let
			$E=\partial_{\operatorname{int}}\Omega\cap\partial M$.
			There are a neighborhood $U$ of
			$\partial_{\operatorname{int}}\Omega$, constants $c>0$ and
			$\beta\in(0,1)$, such that, up to the Dirichlet--Neumann edge,
			\[
				w_p\in C^{1,\beta}\bigl(\overline{U\setminus\Omega}\bigr)
				\cap W^{2,2}(U\setminus\Omega),
				\qquad |\nabla w_p|\ge c.
			\]
			After shrinking $U$ if necessary, for all sufficiently small
			$\varepsilon$ one also has
			$|\nabla w_p^\varepsilon|\ge c/2$ on $U\setminus\Omega$ and
			$\{w_p^\varepsilon\}$ is uniformly bounded in
			$W^{2,2}(U\setminus\Omega)$.
	\end{enumerate}
	\end{lem}
	
	\begin{proof}[Proof of Lemma~\ref{lem:eps-regularization} (overview)]
		The boundaryless regularization and compactness statements are
		collected in \cite[Propositions~2.1 and~2.2]{BPP}; the only part that
		is not already contained there is the free-boundary (Neumann)
		treatment, which is
		carried out in Appendix~\ref{sec:eps-appendix}.  We outline the
		proof here.
		
		Set $u_p^\varepsilon:=e^{-w_p^\varepsilon/(p-1)}$, so that
		\eqref{eq:eps-regularized} is equivalent to
		$\Delta_p^\varepsilon u_p^\varepsilon=0$ with
		$\langle\nabla u_p^\varepsilon,\mathbf{n}\rangle=0$ on
		$\partial M\cap(D\setminus\Omega)$ and the corresponding
		Dirichlet data.  Since
		$A_\varepsilon(\xi)=(|\xi|^2+\varepsilon^2)^{(p-2)/2}\xi$ is
		strictly monotone, existence and uniqueness of
		$u_p^\varepsilon$ follow from the standard theory of monotone
		operators, and the comparison principle gives
		$0<c\le u_p^\varepsilon\le1$ with $c$ independent of
		$\varepsilon$.
		
		\noindent\textit{(i).}  The interior $C^{1,\alpha}$ estimate is
		due to DiBenedetto~\cite{DiBenedetto83} (Tolksdorf~\cite{Tolksdorf84}
		for the more general class), while the boundary $C^{1,\beta}$
		estimate is due to Lieberman~\cite[Theorems~1 and~2]{Lieberman88},
		applied respectively to the Dirichlet and conormal parts away from
		their intersection.  At the initial Dirichlet--Neumann edge, the
		uniform estimate follows from the reflection argument in
		Appendix~\ref{sec:eps-appendix}.  All constants are independent of
		$\varepsilon$.  No regularity at the artificial outer Dirichlet face
		$\partial_{\operatorname{int}}D$ is asserted or used.  The energy
		bound gives weak compactness in
		$W^{1,p}$, strict monotonicity identifies every weak limit with the
		unique solution $u_p:=e^{-w_p/(p-1)}$ of the limiting mixed problem,
		and the uniform $C^{1,\beta}$ estimate then gives, by
		Arzel\`a--Ascoli, convergence of the whole family in
		$C^1_{\mathrm{loc}}$.
		
		\noindent\textit{(ii).}  The $\varepsilon$-uniform local
		$W^{1,2}$ estimate needed here follows directly from
		Schulz--Violo~\cite[Theorem~1.4]{SV25}.  Indeed, with
		\[
			\Psi_\varepsilon(s)=(s^2+\varepsilon^2)^{(p-2)/2},
			\qquad
			q_\varepsilon
			=\Psi_\varepsilon(|\nabla u_p^\varepsilon|)
			 |\nabla u_p^\varepsilon|,
		\]
		the theorem gives an $\varepsilon$-uniform local $W^{1,2}$ bound for
		$q_\varepsilon$ and an $\varepsilon$-uniform local Lipschitz bound for
		$u_p^\varepsilon$.  The homogeneous Neumann condition is precisely
		what makes $u_p^\varepsilon$ a local weak solution on the metric-measure
		space $M$, including at its physical boundary.  Since the map
		$s\mapsto\Psi_\varepsilon(s)s$ has, on the resulting bounded gradient
		range, an inverse with Lipschitz constant independent of
		$\varepsilon$, it follows that
		$\{|\nabla u_p^\varepsilon|\}$ is locally bounded in $W^{1,2}$.
		The logarithmic change of variables then shows that
		$\{|\nabla w_p^\varepsilon|\}$ is bounded in
		$W^{1,2}_{\mathrm{loc}}$ and
		$|\nabla w_p^\varepsilon|\rightharpoonup|\nabla w_p|$ weakly in
		$W^{1,2}_{\mathrm{loc}}$.  The uniformity and the inverse-map step are
		verified in Appendix~\ref{sec:eps-appendix}.
		
		\noindent\textit{(iii).}  For each fixed $\varepsilon>0$, Sard's
		theorem gives, for a.e.\ $t\in(0,\inf_{\partial D}w_p)$, that
		$\Sigma_t^\varepsilon=\{w_p^\varepsilon=t\}$
		is smooth; the conditions $w_p^\varepsilon=0$ on
		$\partial_{\operatorname{int}}\Omega$ and
		$w_p^\varepsilon=w_p\ge\inf_{\partial D}w_p$ on
		$\partial D\cap\operatorname{int}(M)$ keep $\Sigma_t^\varepsilon$
		away from the corners, and the Neumann condition
		$\langle\nabla w_p^\varepsilon,\mathbf{n}\rangle=0$ forces
		$\Sigma_t^\varepsilon$ to meet $\partial M$ orthogonally.
		
		\noindent\textit{(iv).}  Local $C^{1,\beta}$ regularity follows by
		compactness from (i), while smoothness of $w_p$ on
		$D\setminus\bigl(\Omega\cup\operatorname{Crit}(w_p)\bigr)$ follows
		from ordinary uniformly elliptic regularity there.

		\noindent\textit{(v).}
		The assertion at the initial Dirichlet--Neumann edge follows by
		evenly reflecting the equation across the Neumann face.  The
		orthogonality of the two faces turns the reflected initial face into
		a regular Dirichlet boundary.  Boundary regularity, the Hopf lemma,
		and the resulting uniform ellipticity in a collar of the initial
		face then give the stated conclusion.  The details are included in
		Appendix~\ref{sec:eps-appendix}.
	\end{proof}

	\begin{lem}[Cheng--Yau-type gradient estimate]\label{lem:cheng-yau}
		On a complete Riemannian manifold with
		$\operatorname{Ric}\ge0$ and convex boundary $\partial M$,
		the function $w_p=-(p-1)\log u_p$ defined above satisfies the pointwise
		gradient bound
		\[
		|\nabla w_p|(x)\le C_0\,d(x,o)^{-1}
		\]
		for all $x$ with $d(x,o)$ sufficiently large, where
		$C_0=C_0(n,p)$ is allowed to depend on $p$.
	\end{lem}
	
	\begin{proof}
		Endow $M$ with its intrinsic distance and Riemannian volume measure.
		By~\cite[Corollaries~2.5 and~2.6]{Han20}, the assumptions
		$\operatorname{Ric}\ge0$ and $\operatorname{II}\ge0$ imply that
		$(M,d,\operatorname{vol}_g)$ is an $\operatorname{RCD}(0,n)$ space.
		For $x$ sufficiently far from $\Omega$, set
		\[
		R=\frac12d(x,\Omega).
		\]
		Then the intrinsic ball $B(x,R)$ is contained in $M\setminus\Omega$.
		The weak formulation of the mixed problem, including the homogeneous
		Neumann condition on the physical boundary, says precisely that $u_p$
		is $p$-harmonic on this ball in the metric-measure sense.  Therefore
		the Cheng--Yau-type estimate of
		Schulz--Violo~\cite[Corollary~1.2]{SV25} gives
		\[
		|\nabla\log u_p|(x)\le \frac{C(n,p)}{R}.
		\]
		Since $\Omega$ is compact,
		$d(x,\Omega)\ge d(x,o)-C_\Omega$ for some $C_\Omega<\infty$;
		hence $R\ge c\,d(x,o)$ when $d(x,o)$ is sufficiently large.
		Finally,
		$|\nabla w_p|=(p-1)|\nabla\log u_p|$, which proves the claim.
	\end{proof}
	
	\medskip\noindent
	\textit{Mean curvature of the level sets.}
	A direct computation from the equations satisfied by
	$w_p^\varepsilon$ and $w_p$
	(see~\cite[(2.2), (2.3), and~(2.5)]{BPP}) gives
	the following formulas for the mean curvature of their level sets
	at regular points.  For the $\varepsilon$-regularised solution,
	\begin{equation}\label{eq:H-epsilon}
		H^\varepsilon
		= \Bigl(|\nabla w_p^\varepsilon|
		- (p-1)\frac{\langle\nabla|\nabla w_p^\varepsilon|,
			\nabla w_p^\varepsilon\rangle}
		{|\nabla w_p^\varepsilon|^{2}}\Bigr)
		\Bigl(1+\frac{2-p}{p-1}\,\theta_\varepsilon\Bigr),
	\end{equation}
	where $\theta_\varepsilon\in[0,1]$ is defined by
	$\theta_\varepsilon
	=1-|\nabla e^{-w_p^\varepsilon/(p-1)}|^{2}
	/|\nabla e^{-w_p^\varepsilon/(p-1)}|_{\varepsilon}^{2}$.
	For the limit function $w_p$,
	\begin{equation*}
		H
		= |\nabla w_p|
		- (p-1)\frac{\langle\nabla|\nabla w_p|,\nabla w_p\rangle}
		{|\nabla w_p|^{2}}.
	\end{equation*}
	Both formulas hold at every regular point of the corresponding
	level sets, i.e.\ outside $\operatorname{Crit}(w_p^\varepsilon)$ (resp.\
	$\operatorname{Crit}(w_p)$).
	
	We use the normalized $p$-capacity introduced in
	\eqref{eq:capacity-defn}.  For the fixed obstacle $\Omega$, its
	variational identification with the energy of $u_p$ is
	\eqref{eq:capacity-def}.
	
	For $t\ge0$, set
	\[
	\Omega_t^{(p)} = \{x\in M : w_p(x)\le t\}\cup\Omega,
	\qquad
	\Sigma_t = \partial_{\operatorname{int}}\Omega_t^{(p)}
	= \overline{\partial\Omega_t^{(p)}\cap\operatorname{int}(M)}.
	\]
	By coarea, for a.e.\ $t>0$ the set $\Sigma_t$ is countably
	$\mathcal{H}^2$-rectifiable and
	$\mathcal{H}^2(\Sigma_t\cap\operatorname{Crit}(w_p))=0$; its regular
	part is smooth and meets $\partial M$ orthogonally.  No global
	$C^{1,\alpha}$ regularity of the entire level set is asserted here.
	
	\begin{lem}\label{lem:capacity}
		For a.e.\ $t\in[0,+\infty)$,
		\begin{equation}
			C_p(\Sigma_t)
			= \frac{1}{4\pi}\int_{\Sigma_t}
			\Bigl(\frac{|\nabla w_p|}{3-p}\Bigr)^{p-1}\dd\mathcal{H}^{2},
			\label{eq:capacity-levelset}
		\end{equation}
		and $C_p(\Sigma_t)=e^{t}\,C_p(\Sigma_0)$ where
		$\Sigma_0=\partial_{\operatorname{int}}\Omega$.
		
	\end{lem}
	\begin{proof}
		We first prove that the $p$-flux is constant without assuming
		that every level is regular.  For
		$\eta\in C_c^{\infty}((0,+\infty))$, use $\eta(w_p)$ as a test
		function in the weak equation for $u_p$.  The Neumann condition is
		incorporated in this weak formulation.  Since
		$\nabla w_p=-(p-1)\nabla u_p/u_p$, the coarea formula gives
		\begin{align*}
			0
			&=\int_{M\setminus\Omega}|\nabla u_p|^{p-2}
			\langle\nabla u_p,\nabla(\eta(w_p))\rangle
			\,\dd\mathcal{H}^3 \\
			&=-\int_0^{+\infty}\eta'(t)
			\left(\int_{\Sigma_t}|\nabla u_p|^{p-1}
			\,\dd\mathcal{H}^2\right)\dd t.
		\end{align*}
		Consequently, there is a constant $A$ such that
		\[
		\int_{\Sigma_t}|\nabla u_p|^{p-1}\,\dd\mathcal{H}^2=A
		\qquad\text{for a.e. }t>0.
		\]
		A second application of coarea, using
		$|\nabla w_p|=(p-1)|\nabla u_p|/u_p$ and
		$u_p=e^{-w_p/(p-1)}$, yields
		\[
		\int_{M\setminus\Omega}|\nabla u_p|^p\,\dd\mathcal{H}^3
		=\frac{A}{p-1}\int_0^{+\infty}e^{-t/(p-1)}\,\dd t=A.
		\]
		Thus
		$C_p(\Sigma_0)=\frac1{4\pi}
		(\frac{p-1}{3-p})^{p-1}A$.
		
		For a.e. $t>0$, the $p$-capacity potential of
		$\Omega_t^{(p)}$ is not $u_p$ itself, but
		$v_t:=e^{t/(p-1)}u_p$.  Indeed, $v_t=1$ on $\Sigma_t$, it satisfies
		the same homogeneous Neumann condition on $\partial M$, and it
		vanishes at infinity.  The truncation and variational argument used
		in Appendix~\ref{sec:pdepart-appendix} identifies $v_t$ as the
		minimizer in the relaxed energy class.  Therefore,
		\begin{align*}
		C_p(\Sigma_t)
		&=\frac1{4\pi}\Bigl(\frac{p-1}{3-p}\Bigr)^{p-1}
		e^{pt/(p-1)}\int_{\{w_p>t\}}|\nabla u_p|^p
		\,\dd\mathcal{H}^3 \\
		&=\frac1{4\pi}\Bigl(\frac{p-1}{3-p}\Bigr)^{p-1}
		e^{pt/(p-1)}\frac{A}{p-1}
		\int_t^{+\infty}e^{-s/(p-1)}\,\dd s \\
		&=e^t C_p(\Sigma_0).
		\end{align*}
		Finally, on $\Sigma_t$,
		$|\nabla w_p|^{p-1}=(p-1)^{p-1}
		e^t|\nabla u_p|^{p-1}$.  Integrating this identity and using the
		definition of $A$ gives~\eqref{eq:capacity-levelset}.
	\end{proof}
	
	\subsection{Monotone quantities}
	
	For a.e.\ $t\ge0$, let $H$ be the mean curvature of $\Sigma_t$
	with respect to $\nu=\nabla w_p/|\nabla w_p|$.  Define
	\begin{equation*}
		F_p(t)
		:= \int_{\Sigma_t}
		\Bigl(\frac{H\,|\nabla w_p|}{3-p}
		-\frac{|\nabla w_p|^{2}}{(3-p)^{2}}\Bigr)
		\dd\mathcal{H}^{2},
	\end{equation*}
	\begin{equation*}
		G_p(t)
		:= \int_{\Sigma_t}
		\frac{|\nabla w_p|^{2}}{(3-p)^{2}}\,\dd\mathcal{H}^{2}.
	\end{equation*}
	Both are well-defined $L^{1}_{\mathrm{loc}}$-functions
	(see~\cite{BPP}).
	Expanding $(H/2-|\nabla w_p|/(3-p))^{2}\ge0$ gives
	\[
	\frac{H|\nabla w_p|}{3-p}-\frac{|\nabla w_p|^{2}}{(3-p)^{2}}
	\le\frac{H^{2}}{4},
	\]
	hence
	\begin{equation}\label{eq:Fp-bound}
		F_p(t)\le\frac14\int_{\Sigma_t}H^{2}\,\dd\mathcal{H}^{2}.
	\end{equation}
	For $x\in\Sigma_t$, one has
	$u_p(x)=e^{-t/(p-1)}$.  Hence, for $t$ sufficiently large, the
	lower bound in Lemma~\ref{lem:regularity} gives
	\[
		e^{-t/(p-1)}
		\ge C^{-1}d(x,o)^{-\frac{3-p}{p-1}},
		\qquad\text{and therefore}\qquad
		d(x,o)\ge c\,e^{t/(3-p)}.
	\]
	The Cheng--Yau-type estimate in Lemma~\ref{lem:cheng-yau} consequently
	yields
	\[
		\sup_{\Sigma_t}|\nabla w_p|^{3-p}\le C e^{-t}.
	\]
	Combining this estimate with Lemma~\ref{lem:capacity}, we obtain
	\[
	G_p(t)
	\le\frac{4\pi\,C_p(\Sigma_t)}{(3-p)^{3-p}}
	\sup_{\Sigma_t}|\nabla w_p|^{3-p}
	\le C\,e^{-t}C_p(\Sigma_t)=C\,C_p(\Sigma_0),
	\]
	and $G_p(t)$ is uniformly bounded as $t\to+\infty$.
	
	\begin{lem}\label{lem:Fp-derivative}
		$F_p\in W^{1,1}_{\mathrm{loc}}(0,+\infty)$ and is nonincreasing
		on $(0,+\infty)$.  For a.e.\ $t>0$,
		\begin{align}
			F_p'(t) &= -\frac{1}{3-p}\int_{\Sigma_t}
			\biggl(
			\operatorname{Ric}(\nu,\nu)
			+|\mathring{h}|^{2}
			+\frac{|\nabla^{\top}|\nabla w_p||^{2}}{|\nabla w_p|^{2}}
			\nonumber\\
			&\qquad
			+\frac{3-p}{2(p-1)}
			\Bigl(H-\frac{2\,|\nabla w_p|}{3-p}\Bigr)^{\!2}
			\biggr)\dd\mathcal{H}^{2} \nonumber\\
			&\qquad
			-\int_{\Sigma_t\cap\partial M}
			\frac{\operatorname{II}(\nabla w_p,\nabla w_p)}
			{(3-p)\,|\nabla w_p|^{2}}\,\dd\mathcal{H}^{1}
			\;\le\;0,
			\label{eq:Fp-derivative}
		\end{align}
		where $\nu=\nabla w_p/|\nabla w_p|$, $\mathring{h}$ is the
		traceless second fundamental form of $\Sigma_t$, and
		$\nabla^{\top}$ is the tangential gradient.  Moreover,
		$G_p\in W^{2,1}_{\mathrm{loc}}(0,+\infty)$ with
		\begin{equation}
			G_p'(t)
			= \frac{1}{p-1}\int_{\Sigma_t}
			\Bigl(\frac{2\,|\nabla w_p|^{2}}{(3-p)^{2}}
			-\frac{H\,|\nabla w_p|}{3-p}\Bigr)
			\dd\mathcal{H}^{2}.
			\label{eq:Gp-derivative}
		\end{equation}
	\end{lem}
	\begin{proof}
		On the regular set
		$M\setminus(\Omega\cup\operatorname{Crit}(w_p))$, define
		\begin{align}
			X &:= \frac{|\nabla w_p|\,\nabla w_p}{(3-p)^{2}},
			 \\[4pt]
			Y &:= \frac{1}{3-p}
			\Bigl(\frac{\Delta w_p\,\nabla w_p}{|\nabla w_p|}
			-\nabla|\nabla w_p|\Bigr)-X.
			\label{eq:Y-def}
		\end{align}
		The pointwise computations on this set give
		\begin{align}
			\operatorname{div}X
			&= \frac{|\nabla w_p|}{p-1}
			\Bigl(\frac{2|\nabla w_p|^{2}}{(3-p)^{2}}
			-\frac{H|\nabla w_p|}{3-p}\Bigr),
			\label{eq:div-X} \\[4pt]
			\operatorname{div}Y
			&= -\frac{|\nabla w_p|}{3-p}\biggl(
			\operatorname{Ric}(\nu,\nu)
			+|\mathring{h}|^{2}
			+\frac{|\nabla^{\!\top}|\nabla w_p||^{2}}{|\nabla w_p|^{2}}
			+\frac{3-p}{2(p-1)}
			\Bigl(H-\frac{2|\nabla w_p|}{3-p}\Bigr)^{\!2}
			\biggr).
			\label{eq:div-Y}
		\end{align}
		We justify these identities across the critical set by the
		gradient-cutoff procedure of
		\cite[Step~1 in the proof of Theorem~3.1]{BPP}.

		Choose a nondecreasing function
		$\chi\in C^{\infty}([0,+\infty))$ with $\chi=0$ on $[0,1]$
		and $\chi=1$ on $[2,+\infty)$, and put
		$\chi_\delta(r)=\chi(r/\delta)$ for $r\ge0$.  Let
		$\varphi\in\operatorname{Lip}_c((0,+\infty))$ be nonnegative.
		The vector field $\varphi(w_p)\chi_\delta(|\nabla w_p|)Y$ vanishes near
		$\operatorname{Crit}(w_p)$, and its support stays away from the
		initial and outer Dirichlet faces.  Gauss--Green therefore gives
		\begin{align}
		-\int_{M\setminus\Omega}\chi_\delta(|\nabla w_p|)
		\langle Y,\nabla[\varphi(w_p)]\rangle\,\dd\mathcal{H}^3
		&=\int_{M\setminus\Omega}\varphi(w_p)
		\chi_\delta(|\nabla w_p|)
		\operatorname{div}Y\,\dd\mathcal{H}^3 \nonumber\\
		&\quad+E_\delta
		-\int_{\partial M\setminus\Omega}\varphi(w_p)
		\chi_\delta(|\nabla w_p|)
		\langle Y,-\mathbf{n}\rangle\,\dd\mathcal{H}^2,
		\label{eq:Fp-cutoff-identity}
		\end{align}
		where
		\[
		E_\delta
		:=\int_{M\setminus\Omega}\varphi(w_p)
		\chi_\delta'(|\nabla w_p|)
		\langle Y,\nabla|\nabla w_p|\rangle\,\dd\mathcal{H}^3.
		\]
		To identify the cutoff term precisely, rescale $Y$ by the constant
		$-(3-p)$.  The level-set identity
		\[
		H=|\nabla w_p|-(p-1)
		\frac{\partial_\nu|\nabla w_p|}{|\nabla w_p|}
		\]
		gives
		\begin{equation*}
			-(3-p)Y
			=\nabla^{\!\top}|\nabla w_p|
			+(p-1)(\partial_\nu|\nabla w_p|)\nu
			-\frac{2-p}{3-p}|\nabla w_p|^2\nu.
		\end{equation*}
		The rescaled field on the left-hand side is exactly the cutoff field used in
		\cite[Step~1 in the proof of Theorem~3.1]{BPP} when
		$n=3$ and $\alpha=3-p$.  The hypotheses needed there are satisfied:
		by item~\ref{item:W12} of
		Lemma~\ref{lem:eps-regularization},
		$|\nabla w_p|\in W^{1,2}_{\mathrm{loc}}$, while
		\[
			\alpha=3-p>\frac{3-p}{2}=\frac{n-p}{n-1},
			\qquad \alpha+p-2=1>0.
		\]
		In particular, the coarea argument there
		proves
		\begin{equation}\label{eq:Fp-cutoff-error}
		E_\delta\longrightarrow0\qquad\text{as }\delta\downarrow0.
		\end{equation}
		Indeed,
		\[
		E_\delta=-\frac1{3-p}\int_{M\setminus\Omega}\varphi(w_p)
		\chi_\delta'(|\nabla w_p|)
		\langle-(3-p)Y,\nabla|\nabla w_p|\rangle\,\dd\mathcal{H}^3.
		\]
		The argument for the integral on the right-hand side uses only coarea on the
		hypersurfaces $\{|\nabla w_p|=q\}$ and the local second-order estimates for
		$w_p$.  The only new contribution in the free-boundary setting
		appears when that argument applies Gauss--Green to
		$\{r<|\nabla w_p|<q\}$.  On the physical boundary, the Neumann condition and
		Hessian symmetry give
		\begin{equation*}
		\langle Y,-\mathbf{n}\rangle
		=\frac{\operatorname{II}(\nabla w_p,\nabla w_p)}
		{(3-p)|\nabla w_p|}
		\ge0
		\qquad\text{on }\partial M\cap\{|\nabla w_p|>0\}.
		\end{equation*}
		Correspondingly,
		\[
		\langle-(3-p)Y,-\mathbf{n}\rangle
		=-\frac{\operatorname{II}(\nabla w_p,\nabla w_p)}
		{|\nabla w_p|}\le0.
		\]
		In the identity for the flux through the gradient annulus this term
		is subtracted from the physical-boundary side, and therefore only
		strengthens the one-sided estimate used in the cited proof.  Thus
		the conclusion that the gradient-level flux tends to zero, and hence
		\eqref{eq:Fp-cutoff-error}, remains valid.  No initial-edge term occurs because
		$\operatorname{supp}\varphi\Subset(0,+\infty)$.

		Denote the nonnegative expression in parentheses in
		\eqref{eq:div-Y} by $\mathcal{Q}$, and set it equal to zero on the
		critical set.  The same cutoff estimate gives its local
		coarea integrability.  Letting $\delta\downarrow0$ in
		\eqref{eq:Fp-cutoff-identity}, using
		\eqref{eq:Fp-cutoff-error} and monotone convergence for the
		nonnegative volume and boundary dissipations, yields
		\begin{align}
		-\int_0^{+\infty}\varphi'(t)F_p(t)\,\dd t
		&=-\frac1{3-p}\int_0^{+\infty}\varphi(t)
		\int_{\Sigma_t}\mathcal{Q}\,\dd\mathcal{H}^2\,\dd t \nonumber\\
		&\quad-\int_0^{+\infty}\varphi(t)
		\int_{\Sigma_t\cap\partial M}
		\frac{\operatorname{II}(\nabla w_p,\nabla w_p)}
		{(3-p)|\nabla w_p|^2}\,\dd\mathcal{H}^1\,\dd t.
		\label{eq:Fp-weak-derivative}
		\end{align}
		Here we used coarea in $M$ for the volume term and coarea on
		$\partial M$ for the boundary term; the Neumann condition implies
		$|\nabla^{\partial M}w_p|=|\nabla w_p|$ there.  Splitting a general test
		function into its positive and negative parts shows that
		\eqref{eq:Fp-weak-derivative} holds for every
		$\varphi\in\operatorname{Lip}_c((0,+\infty))$.  Hence
		$F_p\in W^{1,1}_{\mathrm{loc}}(0,+\infty)$ and its weak
		derivative is~\eqref{eq:Fp-derivative}.  In particular, $F_p$ is
		nonincreasing.

		For $G_p$, the same argument is simpler.  The field $X$ extends
		by zero across the critical set and
		$\langle X,\mathbf{n}\rangle=0$ on $\partial M$, so no physical
		boundary term occurs.  The weak form of~\eqref{eq:div-X} and
		coarea give~\eqref{eq:Gp-derivative}.  Comparing this formula with
		the definitions of $F_p$ and $G_p$ gives
		\[
		G_p'=\frac{G_p-F_p}{p-1}.
		\]
		Since $F_p,G_p\in W^{1,1}_{\mathrm{loc}}$, it follows that
		$G_p\in W^{2,1}_{\mathrm{loc}}(0,+\infty)$.
	\end{proof}

	\begin{lem}[Endpoint continuity]\label{lem:Fp-endpoint}
		The locally absolutely continuous representative of $F_p$ on
		$(0,+\infty)$ extends continuously to the geometric initial value:
		\begin{equation}\label{eq:Fp-endpoint}
			F_p(0+):=\lim_{t\downarrow0}F_p(t)=F_p(0).
		\end{equation}
		With this extension, $F_p$ is nonincreasing on $[0,+\infty)$.
	\end{lem}
	\begin{proof}
		By item~\ref{item:initial-edge} of
		Lemma~\ref{lem:eps-regularization}, there is a collar of the initial
		face on which $w_p\in C^{1,\beta}\cap W^{2,2}$ and
		$|\nabla w_p|$ is bounded away from zero.  Hence the vector field
		$Y$ in~\eqref{eq:Y-def} belongs to $L^2$ in this collar.  Away from
		the edge, the standard approximation giving~\eqref{eq:div-Y} shows
		that this identity holds weakly and that its right-hand side belongs
		to $L^1$.

		Let $E=\partial_{\operatorname{int}}\Omega\cap\partial M$.
		To apply Gauss--Green to $\{0<w_p<t\}$, choose a logarithmic cutoff
		$\chi_\delta$ which vanishes in a $\delta^2$-neighborhood of $E$,
		equals one outside a $\delta$-neighborhood, and satisfies
		\[
			\int|\nabla\chi_\delta|^2\,\dd\mathcal{H}^3
			\le \frac{C}{|\log\delta|}\longrightarrow0.
		\]
		For a test function $\varphi$, apply the weak identity away from $E$
		to $\chi_\delta\varphi$.  The additional term satisfies
		\[
			\left|\int\varphi\langle Y,\nabla\chi_\delta\rangle
			\,\dd\mathcal{H}^3\right|
			\le \|\varphi Y\|_{L^2}\|\nabla\chi_\delta\|_{L^2}
			\longrightarrow0,
		\]
		while the $L^1$ right-hand side converges by absolute continuity.
		Thus~\eqref{eq:div-Y} holds weakly across $E$ as well.  Applying the
		same cutoff in Gauss--Green and letting $\delta\downarrow0$ gives,
		for almost every sufficiently small $t>0$,
		\begin{equation}\label{eq:Fp-endpoint-GG}
			F_p(t)-F_p(0)
			=\int_{\{0<w_p<t\}}\operatorname{div}Y\,\dd\mathcal{H}^3
			-\int_{\partial M\cap\{0<w_p<t\}}
			\langle Y,-\mathbf{n}\rangle\,\dd\mathcal{H}^2.
		\end{equation}
		Here the definition~\eqref{eq:Y-def} identifies the initial-face
		flux in the direction $\nu=\nabla w_p/|\nabla w_p|$ with $F_p(0)$,
		and the logarithmic cutoff excludes any additional flux concentrated
		on $E$.  The physical-boundary flux is nonnegative by the computation
		in the proof of Lemma~\ref{lem:Fp-derivative}; together
		with~\eqref{eq:div-Y}, this gives $F_p(t)\le F_p(0)$.
		Moreover, the volume integral in~\eqref{eq:Fp-endpoint-GG} tends to
		zero by the absolute continuity of the $L^1$ integral.  The boundary
		integral also tends to zero: its integrand is bounded in the collar,
		and the boundary coarea formula, using that
		$|\nabla^{\partial M}w_p|=|\nabla w_p|\ge c$, shows that the area of
		$\partial M\cap\{0<w_p<t\}$ tends to zero.  Thus
		$F_p(t)\to F_p(0)$ through the full-measure set of values for which
		\eqref{eq:Fp-endpoint-GG} holds.  Since $F_p$ is locally
		absolutely continuous and nonincreasing on $(0,+\infty)$, the same
		limit holds without restricting to regular values, proving
		\eqref{eq:Fp-endpoint}.
	\end{proof}
	
	\begin{lem}\label{lem:Fp-Gp-relation}
		Under $\operatorname{Ric}\ge0$ and $\operatorname{II}\ge0$,
		$F_p$ and $G_p$ are nonincreasing.  For a.e.\ $t$,
		\begin{equation}
			0\le G_p(t)\le F_p(t),
			\label{eq:G-F-inequality}
		\end{equation}
		\begin{equation}
			\bigl(e^{-\frac{t}{p-1}}G_p'(t)\bigr)'
			=-\frac{1}{p-1}\,e^{-\frac{t}{p-1}}F_p'(t)\ge0,
			\label{eq:ODE-relation}
		\end{equation}
		For a.e.\ $\delta>0$,
		\begin{equation}
			e^{-\frac{\delta}{p-1}}G_p'(\delta)
			=\frac{1}{p-1}\int_\delta^{\infty}
			e^{-\frac{t}{p-1}}F_p'(t)\,\dd t.
			\label{eq:integral-relation}
		\end{equation}
	\end{lem}
	\begin{proof}
		Comparing the expressions of $F_p$, $G_p$, and $G_p'$ one
		checks that
		\begin{equation}
			G_p'(t)=\frac{1}{p-1}\bigl(G_p(t)-F_p(t)\bigr).
			\label{eq:G-F-relation}
		\end{equation}
		Differentiating and substituting $G_p''=(G_p'-F_p')/(p-1)$,
		\[
		\bigl(e^{-\frac{t}{p-1}}G_p'(t)\bigr)'
		=-\tfrac{1}{p-1}e^{-\frac{t}{p-1}}F_p'(t),
		\]
		proving~\eqref{eq:ODE-relation}; $F_p'\le0$ gives $\ge0$.
		
		Thus $e^{-t/(p-1)}G_p'(t)$ is nondecreasing.  If $G_p'(s)>0$
		for some $s$, then
		\[
		G_p'(t)\ge e^{(t-s)/(p-1)}G_p'(s)>0
		\qquad\text{for }t\ge s,
		\]
		and integration forces $G_p(t)\to+\infty$, contradicting its
		boundedness.  Hence $G_p'\le0$.  From~\eqref{eq:G-F-relation},
		$G_p\le F_p$, yielding~\eqref{eq:G-F-inequality}.
		
		Fix Lebesgue points $0<\delta<R$.  Integrating
		\eqref{eq:ODE-relation} on $[\delta,R]$ gives
		\[
			e^{-\frac{R}{p-1}}G_p'(R)
			-e^{-\frac{\delta}{p-1}}G_p'(\delta)
			=-\frac1{p-1}\int_\delta^R
			e^{-\frac{t}{p-1}}F_p'(t)\,\dd t.
		\]
		The nondecreasing function
		$e^{-t/(p-1)}G_p'(t)\le0$ has limit zero at infinity.  Indeed,
		if its limit were a negative number, then $G_p'(t)$ would be
		bounded above by a negative multiple of $e^{t/(p-1)}$ for all
		large $t$, contradicting $G_p\ge0$.  Letting $R\to\infty$ in the
		last identity proves~\eqref{eq:integral-relation}.
	\end{proof}
	
	\section{Weak Gauss--Bonnet theorem and consequences}
	
	\subsection{Weak Gauss--Bonnet with free boundary}
	
	For a smooth surface $\Sigma$ in $(M,g)$, the Gauss equation
	expresses the induced scalar curvature $\mathrm{R}^{\top}$ as
	\begin{equation}
		\mathrm{R}^{\top}
		= \mathrm{R}-2\operatorname{Ric}(\nu,\nu)
		+H^{2}-|h|^{2},
		\label{eq:Gauss-equation}
	\end{equation}
	where $\nu$, $H$, $h$ are the unit normal, mean
	curvature, and second fundamental form of $\Sigma$.
	
	In our setting the level sets $\Sigma_t$ of $w_p$ lack the
	regularity for classical Gauss--Bonnet: $w_p$ is merely
	$C^{1,\alpha}$ for $p\neq2$, and although $\Sigma_t$ is
	smooth away from the critical set, the second fundamental form
	is not globally defined.  This nonlinear low-regularity difficulty
	is the one encountered in
	Benatti--Le\'{o}n Quir\'{o}s--Oronzio--Pluda~\cite{BPO}.  By contrast,
	the linear harmonic potential used in
	Benatti--Mantegazza--Oronzio--Pluda~\cite{BMPO} is smooth, so regular
	levels there can be handled directly by Sard's theorem.
	
	We overcome it via the weak Gauss--Bonnet theorem of
	Benatti--Pluda--Pozzetta~\cite{BPP}, which approximates $w_p$
	by smooth functions, passes to the limit in the curvature
	varifold sense, and uses $L^{2}$-semicontinuity estimates
	combined with $p$-harmonic energy identities to obtain a
	Gauss--Bonnet formula for the level sets.  In our free-boundary
	setting, the boundary $\Sigma_t\cap\partial M$ introduces an
	additional geodesic curvature term, whose convergence relies
	on the $C^{1}$-convergence of the approximating functions and
	on varifold convergence results for the boundary curves within
	$\partial M$; see~\cite[Appendix~A]{BPP} and
	\cite[Theorem~6.1]{Mantegazza} for the required compactness and
	identification results.
	
	To carry out this programme we first adapt the regularity
	estimate of~\cite[Proposition~3.5]{BPP} to the free boundary
	setting.  The weak Gauss--Bonnet theorem will be proved at the
	end of the section (Theorem~\ref{thm:weak-GB}).
	
	Let $D\supset\Omega$ and $w_p^\varepsilon$ be as in
	Lemma~\ref{lem:eps-regularization}.  Define the
	vector field
	\begin{equation}\label{eq:Yeps-def}
		Y_\varepsilon
		:= \nabla|\nabla w_p^\varepsilon|
		+(p-2)(1-\theta_\varepsilon)
		\nabla^\perp|\nabla w_p^\varepsilon|,
	\end{equation}
	evaluated on $D\setminus\operatorname{Crit}(w_p^\varepsilon)$, where
	$\nabla^\perp|\nabla w_p^\varepsilon|
	=\langle\nabla|\nabla w_p^\varepsilon|,
	\frac{\nabla w_p^\varepsilon}{|\nabla w_p^\varepsilon|}\rangle
	\frac{\nabla w_p^\varepsilon}{|\nabla w_p^\varepsilon|}$.
	By~\cite[Lemma~3.4]{BPP},
	\begin{equation}\label{eq:divY}
		\operatorname{div}(Y_\varepsilon)
		= |\nabla w_p^\varepsilon|\,
		\bigl(\mathcal{D}_+^\varepsilon
		+\mathcal{D}_\sigma^\varepsilon\bigr),
	\end{equation}
	where $\mathcal{D}_\sigma^\varepsilon\in
	L^2_{\mathrm{loc}}(D\setminus\Omega)$ and
	\begin{align}
		\mathcal{D}_+^\varepsilon
		&:= (p-1)^2\Bigl(1+\frac{2-p}{p-1}\theta_\varepsilon\Bigr)
		\nonumber\\
		&\qquad\times
		\Bigl[\frac{1}{p-1}
		-\frac12\Bigl(1+\frac{2-p}{p-1}\theta_\varepsilon\Bigr)
		\Bigr]
		\frac{|\nabla^\perp|\nabla w_p^\varepsilon||^{2}}
		{|\nabla w_p^\varepsilon|^{2}} \nonumber\\
		&\qquad
		+\frac{|\nabla^\top|\nabla w_p^\varepsilon||^{2}}
		{|\nabla w_p^\varepsilon|^{2}}
		+|\mathring{h}^\varepsilon|^{2},\\[4pt]
		\mathcal{D}_\sigma^\varepsilon
		&:= \operatorname{Ric}
		\Bigl(\frac{\nabla w_p^\varepsilon}{|\nabla w_p^\varepsilon|},
		\frac{\nabla w_p^\varepsilon}{|\nabla w_p^\varepsilon|}\Bigr)
		+\frac12\Bigl(1+\frac{2-p}{p-1}\theta_\varepsilon\Bigr)^{2}
		|\nabla w_p^\varepsilon|^{2} \nonumber\\
		&\qquad
		+2\frac{2-p}{(p-1)^{2}}(1-\theta_\varepsilon)\theta_\varepsilon
		|\nabla w_p^\varepsilon|^{2} \nonumber\\
		&\qquad
		+(p-1)\Bigl(1+\frac{2-p}{p-1}\theta_\varepsilon\Bigr)^{2}
		\Bigl\langle\nabla|\nabla w_p^\varepsilon|\,
		\Bigl|
		\frac{\nabla w_p^\varepsilon}{|\nabla w_p^\varepsilon|}
		\Bigr\rangle \nonumber\\
		&\qquad
		+(2-p)(1-\theta_\varepsilon)
		\Bigl(1-\frac{p}{p-1}\theta_\varepsilon\Bigr)
		\Bigl\langle\nabla|\nabla w_p^\varepsilon|\,
		\Bigl|
		\frac{\nabla w_p^\varepsilon}{|\nabla w_p^\varepsilon|}
		\Bigr\rangle.
		\label{eq:calD}
	\end{align}
	For $n=3$ and $p\in(1,2)$, the coefficient of
	$|\nabla^\perp|\nabla w_p^\varepsilon||^{2}/|\nabla w_p^\varepsilon|^{2}$
	in $\mathcal{D}_+^\varepsilon$ simplifies to
	\begin{equation*}
		\Psi_p(\theta)
		=\frac12[p-1+(2-p)\theta][3-p-(2-p)\theta]
		\ge c_p:=\frac{(p-1)(3-p)}2>0,
		\qquad 0\le\theta\le1.
	\end{equation*}
	Indeed, writing $b=p-1+(2-p)\theta\in[p-1,1]$, this is
	$b(2-b)/2$, an increasing function of $b$ on that interval.
	Consequently $\mathcal{D}_+^\varepsilon\ge0$, with a positive
	coefficient independent of $\varepsilon$ in each gradient-square term.
	
	\begin{prop}[Regularity estimate with free boundary]\label{prop:free-boundary-regularity}
		For every $0<\tau<T<\inf_{\partial D}w_p$ and
		$\tau<\inf_{\partial D}w_p-T$, the $\varepsilon$-level sets
		$\Sigma_s^\varepsilon=\{w_p^\varepsilon=s\}$ satisfy
		\begin{equation}\label{eq:interior-regularity-eps}
			\int_\tau^{T}\!\int_{\Sigma_s^\varepsilon}
			|\mathring{h}^\varepsilon|^{2}
			+\frac{|\nabla|\nabla w_p^\varepsilon||^{2}}
			{|\nabla w_p^\varepsilon|^{2}}
			\,\dd\mathcal{H}^{2}\,\dd s
			\;\le\;
			C\Bigl(
			1+\int_{\tau/2}^{T+\tau}\!\int_{\Sigma_s^\varepsilon}
			|\nabla w_p^\varepsilon|^{2}\,\dd\mathcal{H}^{2}\,\dd s
			\Bigr),
		\end{equation}
		for some $C>0$ depending on $\tau,T,D,p,n=3$ and a lower bound on
		$\operatorname{Ric}$.  Consequently,
		\begin{equation}\label{eq:2ff-estimate}
			\int_\tau^{T}\!\int_{\Sigma_s^\varepsilon}
			|h^\varepsilon|^{2}\,\dd\mathcal{H}^{2}\,\dd s
			\;\le\;
			C\Bigl(
			1+\int_{\tau/2}^{T+\tau}\!\int_{\Sigma_s^\varepsilon}
			|\nabla w_p^\varepsilon|^{2}\,\dd\mathcal{H}^{2}\,\dd s
			\Bigr).
		\end{equation}
	\end{prop}
	
	\begin{proof}
		Let $\varphi\in\operatorname{Lip}_c(0,T)$ be nonnegative.
		Choose $\chi_\delta(r)=\chi(r/\delta)$, with $\chi$ smooth and
		nondecreasing, so that $\chi_\delta\in C^\infty([0,+\infty))$ satisfies
		$\chi_\delta(t)=0$ for $t\le\delta$,
		$0\le\chi_\delta'(t)\le2/\delta$ for $t\in[\delta,2\delta]$,
		$\chi_\delta(t)=1$ for $t\ge2\delta$, and set
		$Y_\varepsilon^\delta
		:=\chi_\delta(|\nabla w_p^\varepsilon|)\,Y_\varepsilon$;
		then $Y_\varepsilon^\delta$ is smooth and vanishes identically on
		$\operatorname{Crit}(w_p^\varepsilon)$.
		
		\begin{step}[Divergence theorem with free boundary]
			Apply the divergence theorem to
			$\varphi(w_p^\varepsilon)\,Y_\varepsilon^\delta$ on
			$D\setminus\Omega$.  The boundary consists of three parts:
			$\partial D$ and $\partial_{\operatorname{int}}\Omega$ (where
			the terms vanish because $\varphi$ has compact support in
			$(0,T)$), and $\partial M\cap(D\setminus\Omega)$.
			
			On $\partial M$, the outward normal to $D\setminus\Omega$ is
			$\nu_{\mathrm{out}}=-\mathbf{n}$, where $\mathbf{n}$ is the
			inward unit normal to $M$.  Since $\nabla^\perp|\nabla w_p^\varepsilon|$
			is parallel to $\nabla w_p^\varepsilon$ and
			$\langle\nabla w_p^\varepsilon,\mathbf{n}\rangle=0$, we have
			\begin{equation*}
				\langle Y_\varepsilon,\mathbf{n}\rangle
				= \langle\nabla|\nabla w_p^\varepsilon|,\mathbf{n}\rangle
				= \frac{\operatorname{II}(\nabla w_p^\varepsilon,
					\nabla w_p^\varepsilon)}
				{|\nabla w_p^\varepsilon|},
			\end{equation*}
			where the second equality follows from the symmetry of the
			Hessian and the Neumann condition, as in~\eqref{eq:kg-II}.
			
			Expanding the divergence theorem,
			\begin{align}
				-\int_{D\setminus\Omega}
				\bigl\langle Y_\varepsilon^\delta,\,
				\nabla[\varphi(w_p^\varepsilon)]
				\bigr\rangle\,\dd\mathcal{H}^{3}
				&=\;
			\int_{D\setminus\Omega}
			\chi_\delta'(|\nabla w_p^\varepsilon|)\,
			\bigl\langle Y_\varepsilon,
			\nabla|\nabla w_p^\varepsilon|\bigr\rangle
			\varphi(w_p^\varepsilon)\,\dd\mathcal{H}^{3} \nonumber\\
			&\quad
			+\int_{D\setminus\Omega}
			\chi_\delta(|\nabla w_p^\varepsilon|)\,
			\operatorname{div}(Y_\varepsilon)\,
			\varphi(w_p^\varepsilon)\,\dd\mathcal{H}^{3} \nonumber\\
				&\quad
				+ \int_{\partial M\cap(D\setminus\Omega)}
				\varphi(w_p^\varepsilon)\,
				\chi_\delta(|\nabla w_p^\varepsilon|)\,
				\frac{\operatorname{II}(\nabla w_p^\varepsilon,
					\nabla w_p^\varepsilon)}
				{|\nabla w_p^\varepsilon|}
				\,\dd\mathcal{H}^{2}.
				\label{eq:div-thm-expanded}
			\end{align}
		\end{step}
		
		\begin{step}[Limit $\delta\to0$]
			The first and third terms on the right-hand side
			of~\eqref{eq:div-thm-expanded} are nonnegative: indeed,
			$\chi_\delta'\ge0$, $\varphi\ge0$, and
			\begin{equation*}
				\bigl\langle Y_\varepsilon,\,
				\nabla|\nabla w_p^\varepsilon|
				\bigr\rangle
				= |\nabla^{\top}|\nabla w_p^\varepsilon||^{2}
				+ (p-1)\Bigl(1+\frac{2-p}{p-1}\,\theta_\varepsilon\Bigr)
				|\nabla^{\perp}|\nabla w_p^\varepsilon||^{2}
				\;\ge\;0
			\end{equation*}
			by~\eqref{eq:Yeps-def}, while the boundary integral over
			$\partial M$ is nonnegative because
			$\operatorname{II}\ge0$ and
			$|\nabla w_p^\varepsilon|>0$ on the support of
			$\chi_\delta$.  The second term contains
			$\operatorname{div}(Y_\varepsilon)
			=|\nabla w_p^\varepsilon|(\mathcal{D}_+^\varepsilon
			+\mathcal{D}_\sigma^\varepsilon)$
			by~\eqref{eq:divY}, where
			$\mathcal{D}_+^\varepsilon\ge0$ by~\eqref{eq:calD} but
			$\mathcal{D}_\sigma^\varepsilon$ is not sign-definite.
			
			We now let $\delta\to0$.  On the left-hand side of
			\eqref{eq:div-thm-expanded},
			$Y_\varepsilon^\delta\to Y_\varepsilon$ pointwise.  Moreover, the
			integrand
			\[
			\bigl\langle Y_\varepsilon^\delta,
			\nabla[\varphi(w_p^\varepsilon)]\bigr\rangle
			\]
			is dominated by an $L^{1}$-function, so dominated convergence gives
			\[
			\int_{D\setminus\Omega}\langle Y_\varepsilon^\delta,
			\nabla[\varphi(w_p^\varepsilon)]\rangle\,\dd\mathcal{H}^{3}
			\xrightarrow{\;\delta\to0\;}
			\int_{D\setminus\Omega}\langle Y_\varepsilon,
			\nabla[\varphi(w_p^\varepsilon)]\rangle\,\dd\mathcal{H}^{3}.
			\]
			On the right-hand side, the first term in the volume integral
			is nonnegative and is dropped from the inequality.  For the
			second term, dominated convergence applies to
			$\mathcal{D}_\sigma^\varepsilon\in L^{2}_{\mathrm{loc}}$, while
			monotone convergence applies to $\mathcal{D}_+^\varepsilon\ge0$
			and to the $\partial M$ boundary term.  Thus
			\begin{equation}\label{eq:key-inequality}
				\begin{aligned}
				&-\int_{D\setminus\Omega}
				\bigl\langle Y_\varepsilon,\,
				\nabla[\varphi(w_p^\varepsilon)]
				\bigr\rangle\,\dd\mathcal{H}^{3}\\
				&\quad\ge
				\int_{D\setminus\Omega}
				\varphi(w_p^\varepsilon)\,
				\bigl(\mathcal{D}_+^\varepsilon
				+\mathcal{D}_\sigma^\varepsilon\bigr)\,
				|\nabla w_p^\varepsilon|\,\dd\mathcal{H}^{3}\\
				&\qquad+\int_{\partial M\cap(D\setminus\Omega)}
				\varphi(w_p^\varepsilon)
				\frac{\operatorname{II}(\nabla w_p^\varepsilon,
				\nabla w_p^\varepsilon)}{|\nabla w_p^\varepsilon|}
				\,\dd\mathcal{H}^2.
				\end{aligned}
			\end{equation}
			Here the boundary density is defined to be zero at critical
			points.  We retain this nonnegative term for the convergence
			argument below.  No vanishing of the gradient-cutoff error
			for fixed $\varepsilon$ is asserted or needed.
		\end{step}
		
		Discarding the nonnegative boundary term in
        \eqref{eq:key-inequality} gives the analogue of the inequality
        established in Step~1 of the proof of
        \cite[Proposition~3.5]{BPP}, equivalently its equation~(3.10)
        after applying the coarea formula.
		
		\begin{step}[Young's inequality and the $L^{2}$ estimate]
			Take $\varphi\in\operatorname{Lip}_c(0,T+\tau)$ with
			$\varphi(s)=1$ for $s\in[\tau,T]$,
			$\varphi(s)=0$ outside $[\tau/2,T+\tau]$, and
			$|\varphi'(s)|\le8/\tau$.  Insert $\varphi^{2}$ into
			\eqref{eq:key-inequality}, applied with $T+\tau$ in place of
			$T$, and discard its nonnegative boundary term.
			By Young's inequality, for every
			$\eta>0$,
			\begin{equation*}
				\bigl\langle Y_\varepsilon,\,
				\nabla[\varphi^{2}(w_p^\varepsilon)]
				\bigr\rangle
				\ge -\eta^{2}[\varphi(w_p^\varepsilon)]^{2}\,
				\frac{|\nabla|\nabla w_p^\varepsilon||^{2}}
				{|\nabla w_p^\varepsilon|}
				-\frac{C}{\eta^{2}}\,[\varphi'(w_p^\varepsilon)]^{2}\,
				|\nabla w_p^\varepsilon|^{3},
			\end{equation*}
			and, using the explicit expressions~\eqref{eq:calD},
			\begin{equation*}
			\begin{aligned}
			&[\varphi(w_p^\varepsilon)]^{2}
			(\mathcal{D}_+^\varepsilon+\mathcal{D}_\sigma^\varepsilon)
			|\nabla w_p^\varepsilon|\\
			&\quad\ge [\varphi(w_p^\varepsilon)]^{2}
			\Biggl[
			(c_{0}-\eta^{2})\frac{|\nabla|\nabla w_p^\varepsilon||^{2}}
			{|\nabla w_p^\varepsilon|}
			+|\mathring{h}^\varepsilon|^{2}|\nabla w_p^\varepsilon|\\
			&\hspace{42mm}
			-\frac{C}{\eta^{2}}|\nabla w_p^\varepsilon|^{3}
			-\kappa|\nabla w_p^\varepsilon|
			\Biggr].
			\end{aligned}
			\end{equation*}
			where $c_{0}>0$ depends only on $p$ and $n=3$, and $\kappa$ is
			a nonnegative constant such that $\operatorname{Ric}\ge-\kappa g$.
			
			Choosing $\eta\ll c_{0}/4$, absorbing the
			$|\nabla|\nabla w_p^\varepsilon||^{2}$ term from the left-hand
			side, and applying the coarea formula
			\[
			\int_{D\setminus\Omega}
			f(w_p^\varepsilon)\,|\nabla w_p^\varepsilon|\,\dd\mathcal{H}^{3}
			=\int_{0}^{+\infty}
			f(s)\int_{\Sigma_s^\varepsilon}\dd\mathcal{H}^{2}\,\dd s,
			\]
			together with the $L^{2}_{\mathrm{loc}}$ summability of the terms
			involved, yields~\eqref{eq:interior-regularity-eps}.
			
			The estimate~\eqref{eq:2ff-estimate} follows from
			\eqref{eq:interior-regularity-eps} and the pointwise bound
			$|h^\varepsilon|^{2}
			\le C(|\mathring{h}^\varepsilon|^{2}
			+|\nabla w_p^\varepsilon|^{2}
			+|\nabla|\nabla w_p^\varepsilon||^{2}
			/|\nabla w_p^\varepsilon|^{2})$,
			which is a consequence of~\eqref{eq:H-epsilon}.
		\end{step}
	\end{proof}
	
	The argument of~\cite[Corollary~3.6]{BPP}, adapted to the free
	boundary setting, upgrades Proposition~\ref{prop:free-boundary-regularity}
	to a global estimate including the initial surface
	$\partial_{\operatorname{int}}\Omega$.  The key additional ingredient
	is Hopf's lemma, which guarantees that
	$\operatorname{Crit}(w_p^\varepsilon)$ stays away from
	$\partial_{\operatorname{int}}\Omega$, so that the divergence theorem
	can be applied without a cutoff on the region between
	$\partial_{\operatorname{int}}\Omega$ and the first level sets.
	The extra boundary term on $\partial M$ produced by this procedure is
	again nonnegative, leaving the estimate unchanged.
	
	\begin{cor}[Global regularity estimate with free boundary]
		
		For every $0<T<\inf_{\partial D}w_p$ and
		$0<\tau<\inf_{\partial D}w_p-T$,
		\begin{equation}\label{eq:global-regularity-eps}
			\int_0^{T}\!\int_{\Sigma_s^\varepsilon}
			|h^\varepsilon|^{2}\,\dd\mathcal{H}^{2}\,\dd s
			\;\le\;
			C\Bigl(
			1+\int_{0}^{T+\tau}\!\int_{\Sigma_s^\varepsilon}
			|\nabla w_p^\varepsilon|^{2}\,\dd\mathcal{H}^{2}\,\dd s
			+\int_{\partial_{\operatorname{int}}\Omega}
			\bigl(H^{2}+|\nabla w_p^\varepsilon|^{2}\bigr)
			\,\dd\mathcal{H}^{2}
			\Bigr),
		\end{equation}
		for some $C>0$ depending on $T,\tau,D,p,n=3$ and a lower bound on
		$\operatorname{Ric}$.
	\end{cor}
	
	The following theorem extends \cite[Theorem~3.7]{BPP} to the free
	boundary setting.  Mantegazza's compactness theorem for curvature
	varifolds with boundary \cite{Mantegazza} is used for the interior
	level surfaces.  The boundary curves are treated independently as
	$1$-varifolds in the two-dimensional manifold $\partial M$, using
	the perimeter convergence supplied by \cite[Corollary~A.11]{BPP}
	and the identification criterion in
	\cite[Corollary~A.12(1)]{BPP}.
	
	\begin{thm}[Varifold convergence with free boundary]
		\label{thm:varifold-convergence-free-boundary}
		Let $p\in(1,2)$, let $\Omega\subset M$ be a closed bounded set
		with connected exterior and smooth boundary meeting $\partial M$
		orthogonally, and let
		$w_p$, $w_p^\varepsilon$ be as in
		Lemma~\ref{lem:eps-regularization}.
		Fix $0<T<\inf_{\partial D}w_p$ and
		$0<\tau<\inf_{\partial D}w_p-T$, and let
		$(\varepsilon_k)_{k\in\mathbb{N}}$ be a vanishing sequence.
		Then, after passing to a subsequence,
		\begin{enumerate}
			\item[(i)] $\Sigma_t^{\varepsilon_k}\to\Sigma_t$ in the sense
			of varifolds, for a.e.\ $t\in[0,T]$, and $\Sigma_t$ is a
			curvature varifold;
			\item[(ii)] $\partial\Sigma_t^{\varepsilon_k}
			\to\partial\Sigma_t$ in the sense of varifolds (as
			$1$-varifolds in $\partial M$), for a.e.\ $t\in[0,T]$;
			\item[(iii)] the $L^{2}$-estimate
			\begin{equation}\label{eq:interior-esitmate-2ff-p}
				\int_0^{T}\!\int_{\Sigma_t}
				|h|^{2}\,\dd\mathcal{H}^{2}\,\dd t
				\le C\Bigl(
				1+\int_0^{T+\tau}\!\int_{\Sigma_t}
				|\nabla w_p|^{2}\,\dd\mathcal{H}^{2}\,\dd t
				+\int_{\partial_{\operatorname{int}}\Omega}
				\bigl(H^{2}+|\nabla w_p|^{2}\bigr)
				\,\dd\mathcal{H}^{2}
				\Bigr)
			\end{equation}
			holds, where $C>0$ depends on $T,\tau,D,p,n=3$ and a lower
			bound on $\operatorname{Ric}$.
		\end{enumerate}
	\end{thm}
	
	\begin{proof}
		Choose $c$ with
		\[
			T+\tau<c<\inf_{\partial D}w_p.
		\]
		We first justify uniform separation from the artificial outer face.
		Put $m=\inf_{\partial D}w_p$ and choose $\eta>0$ so that
		$c+2\eta<m$.  Write
		$\Gamma=\partial_{\operatorname{int}}D$ and let $\nu_\Gamma$ denote
		its inward unit normal.  Since $\Gamma$ meets $\partial M$
		transversely, the tangential projection
		$P_{T\partial M}\nu_\Gamma$ is nonzero along $\partial\Gamma$ and
		\[
			\bigl\langle P_{T\partial M}\nu_\Gamma,
			\nu_\Gamma\bigr\rangle
			=\bigl|P_{T\partial M}\nu_\Gamma\bigr|^2>0.
		\]
		Consequently, a partition-of-unity argument gives a smooth vector
		field $X$ near $\Gamma$ which is tangent to $\partial M$ along
		$\partial M$ and satisfies
		$\langle X,\nu_\Gamma\rangle>0$ on $\Gamma$.  Its short-time flow
		preserves $\partial M$ and supplies a one-sided collar
		$\mathcal{A}$ of $\Gamma$ inside $D$, preserving the
		physical boundary along the edge.  By compactness of $\Gamma$ and
		continuity of $w_p$, this collar may be chosen with smooth inner face
		$S\subset D$ so that
		\[
			w_p>c+2\eta\qquad\text{on }\overline{\mathcal{A}}.
		\]
		The compact face $S$ stays a positive distance from $\Gamma$.
		Lemma~\ref{lem:eps-regularization}
		\ref{item:C1} therefore gives
		$w_p^{\varepsilon_k}>c+\eta$ on $S$ for all sufficiently large $k$.
		On $\Gamma$ one has $w_p^{\varepsilon_k}=w_p\ge m$.  Applying the
		mixed maximum principle in $\mathcal{A}$ to
		$u_p^{\varepsilon_k}=e^{-w_p^{\varepsilon_k}/(p-1)}$ and the
		constant $e^{-(c+\eta)/(p-1)}$ gives
		\[
			w_p^{\varepsilon_k}\ge c+\eta
			\qquad\text{throughout }\mathcal{A}.
		\]
		After discarding finitely many indices, removing a slightly smaller
		outer collar produces a compact set $K_c\Subset D$ containing
		$\{w_p^{\varepsilon_k}\le c\}$ and $\{w_p\le c\}$ for every $k$.
		In particular, all level surfaces with $t\in[0,T+\tau]$ have this
		common compact support, away from the artificial outer face.

		We record explicitly the convergence of the level-set term on the
		right-hand side of~\eqref{eq:global-regularity-eps}.  By coarea and
		the preceding common-support statement,
		\[
		\int_0^{T+\tau}\!\int_{\Sigma_s^{\varepsilon_k}}
		|\nabla w_p^{\varepsilon_k}|^2\,\dd\mathcal{H}^2\,\dd s
		=\int_{K_c}\mathbf 1_{\{0<w_p^{\varepsilon_k}<T+\tau\}}
		|\nabla w_p^{\varepsilon_k}|^3\,\dd\mathcal{H}^3.
		\]
		The $C^1$-convergence on $K_c$ gives a uniform bound for the
		integrands.  Away from the two level sets $\{w_p=0\}$ and
		$\{w_p=T+\tau\}$, the indicator functions converge pointwise.  On
		either exceptional level set, $|\nabla w_p|=0$ almost everywhere with
		respect to $\mathcal{H}^3$, as holds for every Sobolev function on each
		of its level sets.  Hence the integrands converge almost everywhere,
		and dominated convergence yields
		\begin{align*}
		&\lim_{k\to\infty}
		\int_0^{T+\tau}\!\int_{\Sigma_s^{\varepsilon_k}}
		|\nabla w_p^{\varepsilon_k}|^2\,\dd\mathcal{H}^2\,\dd s\\
		&\qquad=\int_0^{T+\tau}\!\int_{\Sigma_s}
		|\nabla w_p|^2\,\dd\mathcal{H}^2\,\dd s.
		\end{align*}
		The same $C^1$-convergence in the initial collar gives convergence of
		the last boundary integral in
		\eqref{eq:global-regularity-eps}.  Thus its entire right-hand side
		converges to the right-hand side of
		\eqref{eq:interior-esitmate-2ff-p}.  Fix a vanishing sequence
		$(\varepsilon_k)_{k\in\mathbb{N}}$.  By Sard's theorem and
		coarea, we may restrict to a full measure set of $t\in(0,T)$
		for which every $\Sigma_t^{\varepsilon_k}$ is smooth and
		\[
		\mathcal{H}^2(\Sigma_t\cap\operatorname{Crit}(w_p))=0,
		\qquad
		\mathcal{H}^1(\partial\Sigma_t\cap\operatorname{Crit}(w_p))=0.
		\]
		The second assertion uses coarea on $\partial M$ and the
		Neumann condition.
		Here $w_p^{\varepsilon_k}$ is $C^{1,\beta}$ up to $\partial M$ and
		$w_p^{\varepsilon_k}\to w_p$ in $C^{1}$ on the part of
		$\partial M$ contained in $K_c$, by the
		conormal boundary estimate of Lieberman used in
		Lemma~\ref{lem:eps-regularization}(i); hence the restrictions to
		$\partial M$ are Lipschitz and converge in $W^{1,1}$.  Therefore,
		by~\cite[Corollary~A.11]{BPP} applied on $M$ and on $\partial M$,
		after passing to a further common subsequence the area and lengths
		converge:
		\begin{equation}\label{eq:area-length-conv}
			\mathcal{H}^{2}(\Sigma_t^{\varepsilon_k})
			\to\mathcal{H}^{2}(\Sigma_t),\qquad
			\mathcal{H}^{1}(\partial\Sigma_t^{\varepsilon_k})
			\to\mathcal{H}^{1}(\partial\Sigma_t).
		\end{equation}
		
		Fix such a $t$.  All level surfaces lie in a common compact
		subset of $D$, including its physical boundary, and their masses
		converge by~\eqref{eq:area-length-conv}.  On every compact subset
		of the interior regular set $\{|\nabla w_p|>0\}$, elliptic
		bootstrapping gives $w_p^{\varepsilon_k}\to w_p$ smoothly:
		the $C^1$ convergence supplies a uniform positive lower bound
		for the gradient there, so the equations are uniformly elliptic
		with uniformly controlled smooth coefficients.  The implicit
		function theorem then gives smooth convergence of the local
		level surfaces.  Every varifold subsequential limit therefore
		agrees with the canonical varifold of $\Sigma_t$ on its interior
		regular part.  This part carries the entire mass of $\Sigma_t$;
		the critical part and its boundary have zero $\mathcal{H}^2$
		measure.  Equality of total masses excludes any remaining mass
		in the limit.  Compactness of finite measures on the Grassmann
		bundle over the common compact support now implies convergence
		of the whole selected sequence of canonical varifolds.

		We next establish curvature regularity, separately from this
		mass-and-tangent-plane convergence.  Fatou's lemma gives only
		\begin{equation*}
			\liminf_{k\to\infty}\int_{\Sigma_t^{\varepsilon_k}}
			|h^{\varepsilon_k}|^2\,\dd\mathcal{H}^2<\infty
			\quad\text{for a.e. }t\in(0,T).
		\end{equation*}
		For each such $t$, choose a subsequence attaining this finite
		lower limit; this subsequence may depend on $t$.  Along it the
		masses, boundary masses, and $L^2$ curvature norms are bounded.
		The curvature-varifold compactness theorem
		\cite[Theorem~6.1]{Mantegazza} applies, and the already identified
		canonical limit is a curvature varifold with $L^2$ second
		fundamental form.  This proves \textit{(i)} without asserting a
		common pointwise curvature bound for the original sequence.
		
		We prove \textit{(ii)} independently on $\partial M$.  Put
		$D_\partial:=D\cap\partial M$, extend the restrictions of
		$w_p^{\varepsilon_k}$ and $w_p$ by zero across
		$\Omega\cap D_\partial$, and set
		\[
			E_{k,t}:=\{w_p^{\varepsilon_k}<t\}\cap D_\partial,
			\qquad E_t:=\{w_p<t\}\cap D_\partial.
		\]
		Set $K=K_c\cap D_\partial$.  The preceding uniform separation gives
		\begin{equation*}
			\{w_p^{\varepsilon_k}\le c\}\cap D_\partial
			\subset K,
			\qquad
			\{w_p\le c\}\cap D_\partial\subset K
		\end{equation*}
		for every $k$.  Thus the zero extensions converge in
		$W^{1,1}_{\mathrm{loc}}(D_\partial)$ and their relevant sublevel
		sets have a common compact support.  Consequently, for a.e.\
		$t\in[0,T]$,
		\[
			\boldsymbol{1}_{E_{k,t}}\longrightarrow
			\boldsymbol{1}_{E_t}
			\quad\text{in }L^1(D_\partial),
		\]
		and \cite[Corollary~A.11]{BPP}, now applied intrinsically to the
		two-dimensional manifold $\partial M$, gives
		\begin{equation}\label{eq:boundary-perimeter-conv}
			\mathcal{H}^1(\partial^*E_{k,t}\cap D_\partial)
			\longrightarrow
			\mathcal{H}^1(\partial^*E_t\cap D_\partial).
		\end{equation}
		For such a $t$, let $W_{k,t}$ denote the canonical integral
		$1$-varifold of $\partial^*E_{k,t}\cap D_\partial$ in
		$\partial M$.  Its mass is
		uniformly bounded by~\eqref{eq:boundary-perimeter-conv}; since all
		the curves under consideration lie in a fixed compact subset of
		$\partial M$, every subsequence of $(W_{k,t})_k$ has a further
		varifold-convergent subsequence.  If $W_t$ is any such limit, the
		$L^1$-convergence of $E_{k,t}$ and
		\eqref{eq:boundary-perimeter-conv}, together with
		\cite[Corollary~A.12(1)]{BPP} applied on $\partial M$, identify it
		uniquely as
		\[
			W_t=\mathbf{v}(\partial^*E_t\cap D_\partial,1).
		\]
		Consequently the whole sequence $W_{k,t}$ converges to $W_t$.
		For a.e.\ $t$, the reduced boundaries inside $D_\partial$ agree
		up to $\mathcal{H}^1$-null sets with
		$\partial\Sigma_t^{\varepsilon_k}$ and $\partial\Sigma_t$,
		respectively.  This proves \textit{(ii)}.
		
		The $t$-dependent subsequence attaining the lower limit also gives
		\[
		\int_{\Sigma_t}|h|^2\,\dd\mathcal{H}^2
		\le\liminf_k\int_{\Sigma_t^{\varepsilon_k}}
		|h^{\varepsilon_k}|^2\,\dd\mathcal{H}^2.
		\]
		Integrating this inequality and using Fatou's lemma in
		\eqref{eq:global-regularity-eps} proves \textit{(iii)}.
	\end{proof}
	
	\begin{lem}\label{lem:kg-convergence}
		For a.e.\ $t\in[0,T]$, with the sequence
		$(\varepsilon_k)_{k\in\mathbb{N}}$ as in
		Theorem~\ref{thm:varifold-convergence-free-boundary},
		the following convergences hold:
		\begin{enumerate}
			\item[(i)]
			$\displaystyle\lim_{k\to+\infty}
			\int_{\partial\Sigma_t^{\varepsilon_k}}
			\operatorname{II}_{\partial M}(\nu^{\varepsilon_k},
			\nu^{\varepsilon_k})\,\dd\mathcal{H}^{1}
			=\int_{\partial\Sigma_t}
			\operatorname{II}_{\partial M}(\nu,\nu)\,\dd\mathcal{H}^{1}$;
			\item[(ii)]
			$\displaystyle\lim_{k\to+\infty}
			\int_{\partial\Sigma_t^{\varepsilon_k}}
			k_g^{\varepsilon_k}\,\dd\mathcal{H}^{1}
			=\int_{\partial\Sigma_t}
			k_g\,\dd\mathcal{H}^{1}$.
		\end{enumerate}
	\end{lem}
	
	\begin{proof}
		By Theorem~\ref{thm:varifold-convergence-free-boundary}(ii), the
		canonical $1$-varifolds of
		$\partial\Sigma_t^{\varepsilon_k}$ converge to that of
		$\partial\Sigma_t$ on $G_{1}(\partial M)$.  For
		$(x,L)\in G_1(\partial M)$, let $L^\perp$ be the orthogonal
		complement of $L$ in $T_x\partial M$, and define
		\begin{equation*}
			\Phi_\perp(x,L)
			:=\operatorname{II}_{\partial M,x}(\xi,\xi),
			\qquad \xi\in L^\perp,\quad |\xi|=1.
		\end{equation*}
		This is well defined because it is unchanged when $\xi$ is replaced
		by $-\xi$, and it is continuous on $G_1(\partial M)$.  Along a
		free-boundary level curve, $L$ is spanned by its unit tangent
		$\dot\gamma$, whereas the surface normal $\nu$ is tangent to
		$\partial M$ and orthogonal to $\dot\gamma$.  Hence
		$L^\perp=\operatorname{span}\{\nu\}$.  Notice that coarea on
		$\partial M$ gives
		$\mathcal{H}^1(\partial\Sigma_t\cap\{|\nabla w_p|=0\})=0$
		for a.e.\ $t$, so
		this identification is valid at $\mathcal{H}^1$-a.e.\ point of the
		limit curve.  Choose $\eta\in C_c(\partial M)$ equal to one on the
		fixed compact set containing all these curves.  Testing the
		varifold convergence with $\eta(x)\Phi_\perp(x,L)$ proves
		\textit{(i)}.
		
		For \textit{(ii)}, at $\mathcal{H}^1$-a.e.\ point of
		$\partial\Sigma_t$, choose a unit tangent $\dot\gamma$; then
		$\{\dot\gamma,\nu\}$ is an orthonormal basis of $T(\partial M)$,
		where $\nu=\nabla w_p/|\nabla w_p|$.  The analogous pointwise
		statement holds on each $\partial\Sigma_t^{\varepsilon_k}$.
		Taking the trace of
		$\operatorname{II}_{\partial M}$ gives
		$H_{\partial M}
		=\operatorname{II}_{\partial M}(\dot\gamma,\dot\gamma)
		+\operatorname{II}_{\partial M}(\nu,\nu)$
		at these points, where $H_{\partial M}$ is the mean curvature of
		$\partial M$ in $M$, a smooth function on $\partial M$
		depending only on the ambient geometry.  By
		\eqref{eq:kg-II},
		$k_g=\operatorname{II}_{\partial M}(\dot\gamma,
		\dot\gamma)$.  Hence
		$k_g
		=H_{\partial M}-\operatorname{II}_{\partial M}(\nu,\nu)$.
		The integral of $H_{\partial M}$ converges by testing the
		varifolds with the compactly supported continuous function
		$(x,L)\mapsto\eta(x)H_{\partial M}(x)$, while the normal term
		converges by
		\textit{(i)}.  Subtracting the two limits proves \textit{(ii)}.
	\end{proof}
	
	The next proposition is the free-boundary counterpart of
	\cite[Proposition~5.1]{BPP}.  Two points require attention here: the
	physical-boundary term must be retained and passed to the limit, and a
	common a.e. subsequence is extracted only after the total density is
	shown to converge in $L^1_{\mathrm{loc}}(0,T)$.  The latter order avoids
	inferring a common pointwise curvature bound from a space--time bound.

	\begin{prop}[Convergence of the second fundamental form
		with free boundary]
		\label{prop:convergence-second-fundamental-form}
		Under the hypotheses of Theorem~\ref{thm:varifold-convergence-free-boundary},
		let $(\varepsilon_k)_{k\in\mathbb{N}}$ be a vanishing sequence
		along which its varifold convergences hold.  Then, after passing to a
		further common subsequence, for a.e.\ $t\in(0,T)$,
		\begin{align}
			\lim_{k\to\infty}\int_{\Sigma_t^{\varepsilon_k}}
			|H^{\varepsilon_k}|^2\,\dd\mathcal{H}^2
			&=\int_{\Sigma_t}|H|^2\,\dd\mathcal{H}^2,
			\label{eq:ConvergenceMeanCurvature}\\
			\lim_{k\to\infty}\int_{\Sigma_t^{\varepsilon_k}}
			|h^{\varepsilon_k}|^2\,\dd\mathcal{H}^2
			&=\int_{\Sigma_t}|h|^2\,\dd\mathcal{H}^2.
			\label{eq:ConvergenceSecondFundamentalForm}
		\end{align}
	\end{prop}

	\begin{proof}
		Put $\mathcal{U}=D\setminus\Omega$ and
		$\Gamma=\partial M\cap\mathcal{U}$.  We work on the common
		full-measure set of levels for which the varifold convergences in
		Theorem~\ref{thm:varifold-convergence-free-boundary} hold, every
		$\Sigma_t^{\varepsilon_k}$ is smooth, and
		$\mathcal{H}^2(\Sigma_t\cap\operatorname{Crit}(w_p))=0$.
		All densities below are extended by zero on critical sets.

		\medskip\noindent
		\textit{1. Fixed-level lower semicontinuity.}
		Set
		\begin{align*}
			I_k^0(t)&=\int_{\Sigma_t^{\varepsilon_k}}
			|\mathring{h}^{\varepsilon_k}|^2
			\,\dd\mathcal{H}^2,\\
			I_k^\perp(t)&=\int_{\Sigma_t^{\varepsilon_k}}
			\Psi_p(\theta_{\varepsilon_k})
			\frac{\bigl|\nabla^\perp|\nabla w_p^{\varepsilon_k}|\bigr|^2}
			{|\nabla w_p^{\varepsilon_k}|^2}
			\,\dd\mathcal{H}^2,\\
			I_k^\top(t)&=\int_{\Sigma_t^{\varepsilon_k}}
			\frac{\bigl|\nabla^\top|\nabla w_p^{\varepsilon_k}|\bigr|^2}
			{|\nabla w_p^{\varepsilon_k}|^2}
			\,\dd\mathcal{H}^2,
			\qquad E_k=I_k^0+I_k^\perp+I_k^\top.
		\end{align*}
		Let $I^0,I^\perp,I^\top$ and $E$ denote the corresponding
		quantities for $w_p$, with
		$\theta=0$ and hence coefficient $\Psi_p(0)$ in $I^\perp$.
		We claim that, for almost every $t\in(0,T)$,
		\begin{equation}\label{eq:short-LSC-total}
			E(t)\le\liminf_{k\to\infty}E_k(t).
		\end{equation}

		Fix such a level and take a subsequence realizing the lower limit.
		There is nothing to prove if this lower limit is infinite.
		Otherwise $\sup_kE_k(t)<\infty$.  Since
		$\Psi_p\ge c_p>0$, formula~\eqref{eq:H-epsilon}, the uniform
		$C^1$ bound, and the area convergence give
		\[
			\sup_k\int_{\Sigma_t^{\varepsilon_k}}
			|H^{\varepsilon_k}|^2\,\dd\mathcal{H}^2<\infty.
		\]
		Thus the integrals of
		$|h^{\varepsilon_k}|^2
		=|\mathring{h}^{\varepsilon_k}|^2
		+|H^{\varepsilon_k}|^2/2$ are uniformly bounded on the level.
		Together with the area and boundary-length convergence, this permits
		the application of curvature-varifold compactness with boundary
		\cite[Theorem~6.1]{Mantegazza}.  Its ordinary varifold limit is the
		canonical varifold of $\Sigma_t$ by
		Theorem~\ref{thm:varifold-convergence-free-boundary}.
		To apply the Euclidean compactness theorem, smoothly extend a
		neighborhood of the common compact support across the physical
		boundary and isometrically embed it in Euclidean space.  The bounded
		second fundamental form of this embedding and the area bound
		convert the intrinsic curvature bound into the required extrinsic
		bound; subtracting the smooth ambient contribution identifies the
		intrinsic curvature of the limit.

		The lower-semicontinuity arguments in
		\cite[Proof of Proposition~5.1, equations~(5.4), (5.8),
		and~(5.9)]{BPP} now give, along this subsequence,
		\begin{equation}\label{eq:short-component-LSC}
			I^0(t)\le\liminf_k I_k^0(t),\qquad
			I^\perp(t)\le\liminf_k I_k^\perp(t),\qquad
			I^\top(t)\le\liminf_k I_k^\top(t).
		\end{equation}
		For later use, the same argument for the normal component gives,
		for every nonnegative $F\in C([0,1])$,
		\begin{equation}\label{eq:short-weighted-normal-LSC}
			F(0)\int_{\Sigma_t}
			\frac{|\nabla^\perp|\nabla w_p||^2}
			{|\nabla w_p|^2}\,\dd\mathcal{H}^2
			\le\liminf_k\int_{\Sigma_t^{\varepsilon_k}}
			F(\theta_{\varepsilon_k})
			\frac{\bigl|\nabla^\perp|\nabla w_p^{\varepsilon_k}|\bigr|^2}
			{|\nabla w_p^{\varepsilon_k}|^2}\,\dd\mathcal{H}^2.
		\end{equation}
		We briefly record why the boundary causes no change in these local
		arguments.  In the mean-curvature tests, the fields
		$F(\theta_{\varepsilon_k})\nabla w_p^{\varepsilon_k}$ are tangent
		to $\partial M$ by the Neumann condition, and hence annihilate the
		free-boundary conormal term.  For the tangential-gradient term one
		first uses test fields compactly supported in
		$\mathcal{U}\setminus\partial M$ and then exhausts $\mathcal{U}$;
		the omitted boundary curve is $\mathcal{H}^2$-negligible.
		Therefore the proof in BPP applies after the preceding
		curvature-varifold compactness step.  Adding
		\eqref{eq:short-component-LSC} proves
		\eqref{eq:short-LSC-total}.  Although the subsequence used in this
		verification may depend on $t$, the conclusion concerns the lower
		limit of the original sequence and therefore requires no measurable
		choice.
		More generally, the same argument proves the component and weighted
		normal lower bounds along any subsequence with bounded total energy
		on the fixed level: one applies compactness to a further subsequence
		realizing the lower limit of the component under consideration.

		\medskip\noindent
		\textit{2. No loss of total density.}
		First, \eqref{eq:interior-regularity-eps}, the uniform $C^1$ bound,
		and the common compact support imply
		$\sup_k\int_I E_k(t)\,\dd t<\infty$ for every
		$I\Subset(0,T)$.  Here the gradient-square integral on the
		right-hand side of the regularity estimate is uniformly bounded
		by coarea.  Since $\Psi_p$ is bounded on $[0,1]$, that estimate
		controls all three components of $E_k$.  Thus
		\eqref{eq:short-LSC-total} and Fatou's lemma give
		\[
			\int_I E(t)\,\dd t
			\le\liminf_k\int_I E_k(t)\,\dd t<\infty.
		\]
		Fix $0\le\psi\in C_c^\infty(0,T)$ and define
		\begin{align*}
			L_\varepsilon&=-\int_{\mathcal{U}}
			\langle Y_\varepsilon,\nabla[\psi(w_p^\varepsilon)]\rangle
			\,\dd\mathcal{H}^3,\\
			A_\varepsilon&=\int_{\mathcal{U}}\psi(w_p^\varepsilon)
			|\nabla w_p^\varepsilon|\mathcal{D}_+^\varepsilon
			\,\dd\mathcal{H}^3,\\
			R_\varepsilon&=\int_{\mathcal{U}}\psi(w_p^\varepsilon)
			|\nabla w_p^\varepsilon|\mathcal{D}_\sigma^\varepsilon
			\,\dd\mathcal{H}^3,\\
			B_\varepsilon&=\int_\Gamma\psi(w_p^\varepsilon)
			\frac{\operatorname{II}(\nabla w_p^\varepsilon,
			\nabla w_p^\varepsilon)}
			{|\nabla w_p^\varepsilon|}\,\dd\mathcal{H}^2.
		\end{align*}
		Use $L,A,R,B$ for the corresponding quantities at
		$\varepsilon=0$.  To relate the limiting field to
		Lemma~\ref{lem:Fp-derivative}, denote it by $Z$; then
		\[
			\begin{aligned}
			Z&=\nabla|\nabla w_p|
			+(p-2)\nabla^\perp|\nabla w_p|\\
			&=-(3-p)Y+\frac{2-p}{3-p}|\nabla w_p|\nabla w_p,
			\end{aligned}
		\]
		where $Y$ is the field in \eqref{eq:Y-def}.  The additional
		field is tangent to $\partial M$.  With the gradient cutoff
		$\chi_\delta$ used in that lemma, its cutoff error satisfies
		\[
			\begin{aligned}
			&\left|\int_{\mathcal U}\psi(w_p)
			\chi_\delta'(|\nabla w_p|)|\nabla w_p|^2
			\partial_\nu|\nabla w_p|\,\dd\mathcal H^3\right|\\
			&\qquad\le C\delta\int_K
			|\nabla|\nabla w_p||\,\dd\mathcal H^3\longrightarrow0,
			\end{aligned}
		\]
		where $K$ is a fixed compact set containing the support of
		$\psi(w_p)$.  The last integral is finite by
		Lemma~\ref{lem:eps-regularization}(ii).  Combining this with
		\eqref{eq:Fp-cutoff-error} gives the exact weak identity for $Z$,
		whose divergence on the regular set is
		$|\nabla w_p|(\mathcal D_++\mathcal D_\sigma)$ and whose outward
		physical-boundary flux is
		$-\operatorname{II}(\nabla w_p,\nabla w_p)/|\nabla w_p|$.
		The retained-boundary form of
		\eqref{eq:key-inequality} and the exact weak identity for the limit
		field therefore give, respectively,
		\begin{equation}\label{eq:short-energy-comparison}
			A_\varepsilon+R_\varepsilon+B_\varepsilon\le L_\varepsilon,
			\qquad A+R+B=L.
		\end{equation}
		The first relation is only an inequality because the nonnegative
		critical-set cutoff error has been discarded.

		The weak convergence of
		$\nabla|\nabla w_p^{\varepsilon_k}|$ in $L^2_{\mathrm{loc}}$
		and the uniform convergence of the remaining gradient factors give
		\[
			L_{\varepsilon_k}\longrightarrow L,\qquad
			R_{\varepsilon_k}\longrightarrow R;
		\]
		the second convergence is exactly the argument of
		\cite[(5.13)]{BPP}.  For the additional boundary term, define on
		$T\partial M$
		\[
			\Phi(x,z)=
			\begin{cases}
				\operatorname{II}_x(z,z)/|z|,&z\ne0,\\
				0,&z=0.
			\end{cases}
		\]
		The estimate $|\Phi(x,z)|\le C|z|$ shows that $\Phi$ is continuous
		at $z=0$.  Since the Neumann condition makes the gradients tangent
		to $\partial M$, the boundary $C^1$ convergence in
		Lemma~\ref{lem:eps-regularization}(i) yields
		\begin{equation}\label{eq:short-boundary-energy}
			B_{\varepsilon_k}
			=\int_\Gamma\psi(w_p^{\varepsilon_k})
			\Phi(x,\nabla w_p^{\varepsilon_k})\,\dd\mathcal{H}^2
			\longrightarrow B.
		\end{equation}
		This is the only new convergence term relative to
		\cite[Proposition~5.1]{BPP}.

		By coarea, $A_{\varepsilon_k}=\int_0^T\psi(t)E_k(t)\,\dd t$
		and $A=\int_0^T\psi(t)E(t)\,\dd t$.  Hence
		\eqref{eq:short-LSC-total} and Fatou's lemma give
		$A\le\liminf_kA_{\varepsilon_k}$.  Combining this inequality with
		\eqref{eq:short-energy-comparison} and
		\eqref{eq:short-boundary-energy} gives
		\[
			A\le\liminf_kA_{\varepsilon_k}
			\le\limsup_kA_{\varepsilon_k}
			\le L-R-B=A.
		\]
		Consequently,
		\begin{equation}\label{eq:short-total-no-loss}
			\int_0^T\psi(t)E_k(t)\,\dd t
			\longrightarrow\int_0^T\psi(t)E(t)\,\dd t
		\end{equation}
		for every nonnegative $\psi\in C_c^\infty(0,T)$.

		\medskip\noindent
		\textit{3. A common subsequence and the normal remainder.}
		The combination of \eqref{eq:short-LSC-total} and
		\eqref{eq:short-total-no-loss} implies
		$E_k\to E$ in $L^1_{\mathrm{loc}}(0,T)$.  Indeed,
		$(E-E_k)_+\to0$ pointwise and is bounded by $E$, which is locally
		integrable by the regularity estimate and Fatou argument in Step~2.
		If $I\Subset(0,T)$ and $\psi$ equals one on $I$, then
		\[
			|E_k-E|=(E_k-E)+2(E-E_k)_+
		\]
		together with \eqref{eq:short-total-no-loss} and dominated
		convergence proves convergence in $L^1(I)$.  A diagonal extraction
		therefore gives one common subsequence such that
		\begin{equation}\label{eq:short-pointwise-total}
			E_k(t)\longrightarrow E(t)
			\qquad\text{for a.e. }t\in(0,T).
		\end{equation}

		Fix a level satisfying \eqref{eq:short-pointwise-total}.  The
		fixed-level compactness argument of Step~1 applies to this common
		subsequence, whose total energy is now bounded on the level.
		Applying that argument to a subsequence realizing each component's
		lower limit proves \eqref{eq:short-component-LSC} for the common
		subsequence itself.  Since the three nonnegative components satisfy
		\eqref{eq:short-component-LSC} and their sum converges, each
		component converges.  In particular,
		\begin{align}
			I_k^0(t)&\longrightarrow I^0(t),
			\label{eq:short-traceless-no-loss}\\
			I_k^\perp(t)&\longrightarrow I^\perp(t).
			\label{eq:short-normal-no-loss}
		\end{align}

		Choose $0<\kappa<\min_{[0,1]}\Psi_p$.  Both
		$\Psi_p(\theta)-\kappa\theta$ and $\kappa\theta$ are nonnegative.
		Applying \eqref{eq:short-weighted-normal-LSC} to the first weight
		and using \eqref{eq:short-normal-no-loss} for their sum yields
		\begin{equation}\label{eq:short-theta-vanishing}
			\int_{\Sigma_t^{\varepsilon_k}}
			\theta_{\varepsilon_k}
			\frac{\bigl|\nabla^\perp|\nabla w_p^{\varepsilon_k}|\bigr|^2}
			{|\nabla w_p^{\varepsilon_k}|^2}\,\dd\mathcal{H}^2
			\longrightarrow0.
		\end{equation}
		Because $\Psi_p$ is Lipschitz and $\Psi_p(0)>0$,
		\eqref{eq:short-normal-no-loss} and
		\eqref{eq:short-theta-vanishing} also give
		\begin{equation}\label{eq:short-unweighted-normal}
			\int_{\Sigma_t^{\varepsilon_k}}
			\frac{\bigl|\nabla^\perp|\nabla w_p^{\varepsilon_k}|\bigr|^2}
			{|\nabla w_p^{\varepsilon_k}|^2}\,\dd\mathcal{H}^2
			\longrightarrow
			\int_{\Sigma_t}
			\frac{|\nabla^\perp|\nabla w_p||^2}
			{|\nabla w_p|^2}\,\dd\mathcal{H}^2.
		\end{equation}

		\medskip\noindent
		\textit{4. Mean curvature and the full second fundamental form.}
		Set
		\[
			q(\theta)=1+\frac{2-p}{p-1}\theta,
			\qquad q_k=q(\theta_{\varepsilon_k}).
		\]
		With $\vec{H}^{\varepsilon_k}=-H^{\varepsilon_k}
		\nu^{\varepsilon_k}$, formula~\eqref{eq:H-epsilon} is equivalent to
		\begin{equation}\label{eq:short-H-square}
			|H^{\varepsilon_k}|^2
			=q_k^2\left[(p-1)^2
			\frac{\bigl|\nabla^\perp|\nabla w_p^{\varepsilon_k}|\bigr|^2}
			{|\nabla w_p^{\varepsilon_k}|^2}
			-|\nabla w_p^{\varepsilon_k}|^2\right]
			-2q_k\langle\vec{H}^{\varepsilon_k},
			\nabla w_p^{\varepsilon_k}\rangle.
		\end{equation}
		The estimate $|q_k^2-1|\le C\theta_{\varepsilon_k}$ and
		\eqref{eq:short-theta-vanishing}--\eqref{eq:short-unweighted-normal}
		give convergence of the first normal term in
		\eqref{eq:short-H-square}.  The fixed-level compactness from Step~1
		and the convergence arguments
		\cite[(5.5) and~(5.6)]{BPP}, applied with
		$F(\theta)=q(\theta)^2$ and $F(\theta)=q(\theta)$, give convergence
		of the gradient-square and mean-curvature-pairing terms.
		Integrating \eqref{eq:short-H-square} therefore proves
		\eqref{eq:ConvergenceMeanCurvature}.
		Finally, in dimension three,
		\[
			|h^\varepsilon|^2
			=|\mathring{h}^\varepsilon|^2
			+\frac12|H^\varepsilon|^2.
		\]
		Combining this identity with
		\eqref{eq:short-traceless-no-loss} proves
		\eqref{eq:ConvergenceSecondFundamentalForm}.
		The diagonal extraction in Step~3 is common to both conclusions.
	\end{proof}
	
	We can now complete the proof of the weak Gauss--Bonnet theorem
	with free boundary.
	
	\begin{thm}[Weak Gauss--Bonnet with free boundary]\label{thm:weak-GB}
		Let $(M,\partial M,g)$ be a Ricci-pinched $3$-manifold with
		convex boundary and superquadratic volume growth.  For
		$p\in(1,2)$, let $w_p$ solve~\eqref{eq:wp}.  Then for a.e.\ $t$,
		\[
		\int_{\Sigma_t}\mathrm{R}^{\top}\,\dd\mathcal{H}^{2}
		+2\int_{\Sigma_t\cap\partial M}k_g\,\dd\mathcal{H}^{1}
		\in 4\pi\mathbb{Z},
		\]
		where $\mathrm{R}^{\top}$ is given by~\eqref{eq:Gauss-equation}
		and $k_g$ is the geodesic curvature of
		$\Sigma_t\cap\partial M$ within $\Sigma_t$.
	\end{thm}
	
	\begin{proof}
		Let $(\varepsilon_k)_{k\in\mathbb{N}}$ be a vanishing sequence
		as in Proposition~\ref{prop:convergence-second-fundamental-form},
		so that the conclusions of
		Theorem~\ref{thm:varifold-convergence-free-boundary},
		Lemma~\ref{lem:kg-convergence}, and
		Proposition~\ref{prop:convergence-second-fundamental-form} all
		hold for a common set of full measure in $[0,T]$.  Fix $t$ in
		this set.
		
		\medskip\noindent
		\textbf{1.  Classical Gauss--Bonnet for the smooth level sets.}
		For each $k$, $\Sigma_t^{\varepsilon_k}$ is a smooth compact
		surface with smooth boundary
		$\partial\Sigma_t^{\varepsilon_k}\subset\partial M$, meeting
		$\partial M$ orthogonally.  By the Gauss equation~\eqref{eq:Gauss-equation},
		the induced scalar curvature of $\Sigma_t^{\varepsilon_k}$ is
		\[
		\mathrm{R}^{\top}_{\varepsilon_k}
		=\mathrm{R}
		-2\operatorname{Ric}(\nu^{\varepsilon_k},\nu^{\varepsilon_k})
		+(H^{\varepsilon_k})^{2}
		-|h^{\varepsilon_k}|^{2}.
		\]
		The classical Gauss--Bonnet theorem, applied separately to every
		connected component of $\Sigma_t^{\varepsilon_k}$, yields
		\begin{equation}\label{eq:classical-GB}
			\int_{\Sigma_t^{\varepsilon_k}}
			\mathrm{R}^{\top}_{\varepsilon_k}\,\dd\mathcal{H}^{2}
			+2\int_{\partial\Sigma_t^{\varepsilon_k}}
			k_g^{\varepsilon_k}\,\dd\mathcal{H}^{1}
			=4\pi\,\chi(\Sigma_t^{\varepsilon_k})
			\in4\pi\mathbb{Z},
		\end{equation}
		Here both closed components and components with boundary are allowed.
		In either case their Euler characteristics are integers, so the sum
		on the right-hand side belongs to $4\pi\mathbb{Z}$.  The Neumann
		condition is used only to identify the boundary geodesic-curvature
		term when a component meets $\partial M$; it does not assert that
		every level set has nonempty boundary.
		
		\medskip\noindent
		\textbf{2.  Convergence of the Gauss--Bonnet functional.}
		We pass each term in~\eqref{eq:classical-GB} to the limit
		$k\to+\infty$.  By the varifold convergence
		$\Sigma_t^{\varepsilon_k}\to\Sigma_t$
		(Theorem~\ref{thm:varifold-convergence-free-boundary}(i))
		and the continuity of the ambient scalar curvature $\mathrm{R}$
		on $M$,
		\[
		\int_{\Sigma_t^{\varepsilon_k}}
		\mathrm{R}\,\dd\mathcal{H}^{2}
		\;\xrightarrow{\;k\to+\infty\;}
		\int_{\Sigma_t}\mathrm{R}\,\dd\mathcal{H}^{2}.
		\]
		For the Ricci term, the varifold convergence implies
		convergence of the Grassmann directions (the unit normals
		$\nu^{\varepsilon_k}$), so the same argument gives
		\[
		\int_{\Sigma_t^{\varepsilon_k}}
		\operatorname{Ric}(\nu^{\varepsilon_k},\nu^{\varepsilon_k})
		\,\dd\mathcal{H}^{2}
		\;\xrightarrow{\;k\to+\infty\;}
		\int_{\Sigma_t}
		\operatorname{Ric}(\nu,\nu)\,\dd\mathcal{H}^{2},
		\]
		since $\operatorname{Ric}$ is a smooth $(0,2)$-tensor on $M$.
		The mean curvature and second fundamental form converge by
		Proposition~\ref{prop:convergence-second-fundamental-form}:
		\[
		\lim_{k\to+\infty}
		\int_{\Sigma_t^{\varepsilon_k}}|H^{\varepsilon_k}|^{2}
		=\int_{\Sigma_t}|H|^{2},\qquad
		\lim_{k\to+\infty}
		\int_{\Sigma_t^{\varepsilon_k}}|h^{\varepsilon_k}|^{2}
		=\int_{\Sigma_t}|h|^{2}.
		\]
		The geodesic curvature term converges by
		Lemma~\ref{lem:kg-convergence}:
		\[
		\lim_{k\to+\infty}
		\int_{\partial\Sigma_t^{\varepsilon_k}}
		k_g^{\varepsilon_k}\,\dd\mathcal{H}^{1}
		=\int_{\partial\Sigma_t}k_g\,\dd\mathcal{H}^{1}.
		\]
		Summing the convergences via the Gauss equation, we obtain
		\begin{equation}\label{eq:GB-convergence}
			\lim_{k\to+\infty}
			\Bigl(
			\int_{\Sigma_t^{\varepsilon_k}}
			\mathrm{R}^{\top}_{\varepsilon_k}\,\dd\mathcal{H}^{2}
			+2\int_{\partial\Sigma_t^{\varepsilon_k}}
			k_g^{\varepsilon_k}\,\dd\mathcal{H}^{1}
			\Bigr)
			=\int_{\Sigma_t}\mathrm{R}^{\top}\,\dd\mathcal{H}^{2}
			+2\int_{\partial\Sigma_t}k_g\,\dd\mathcal{H}^{1}.
		\end{equation}
		
		\medskip\noindent
		\textbf{3.  Quantisation.}
		The entire left-hand side of~\eqref{eq:classical-GB}
		belongs to $4\pi\mathbb{Z}$.  Since $4\pi\mathbb{Z}$ is a
		discrete subset of $\mathbb{R}$, the limit in
		\eqref{eq:GB-convergence} must also belong to
		$4\pi\mathbb{Z}$, which is precisely the statement of the
		theorem.
	\end{proof}
	
	\subsection{Willmore-type estimates}
	
	Combining the weak Gauss--Bonnet theorem with the Gauss equation
	and the pinching conditions yields inequalities relating the
	Willmore energy $\int H^{2}$ to the Ricci curvature and traceless
	second fundamental form (cf.\ Huisken--Koerber~\cite{HK24} and
	Benatti--Le\'{o}n Quir\'{o}s--Oronzio--Pluda~\cite{BPO}).
	
	\begin{lem}
		For a.e.\ $t$:
		\begin{enumerate}
			\item[(i)] If $\int_{\Sigma_t}\mathrm{R}^{\top}
			+2\int_{\Sigma_t\cap\partial M}k_g\le0$, then
			\begin{equation}
				\int_{\Sigma_t}H^{2}
				\le 2\int_{\Sigma_t}
				\bigl(\operatorname{Ric}(\nu,\nu)
				+|\mathring{h}|^{2}\bigr).
				\label{eq:Willmore-case1}
			\end{equation}
			\item[(ii)] If $\int_{\Sigma_t}\mathrm{R}^{\top}
			+2\int_{\Sigma_t\cap\partial M}k_g\ge4\pi$, then
			\begin{equation}
				\int_{\Sigma_t}H^{2}
				\ge 8\pi
				-4\int_{\Sigma_t\cap\partial M}k_g
				-\frac{2}{\rho}
				\int_{\Sigma_t}\operatorname{Ric}(\nu,\nu).
				\label{eq:Willmore-case2}
			\end{equation}
		\end{enumerate}
	\end{lem}
	\begin{proof}
		From $|h|^{2}=|\mathring{h}|^{2}+\frac12 H^{2}$ and~\eqref{eq:Gauss-equation},
		\[
		H^{2}=2\mathrm{R}^{\top}-2\mathrm{R}
		+4\operatorname{Ric}(\nu,\nu)+2|\mathring{h}|^{2}.
		\]
		Integrating,
		\[
		\int_{\Sigma_t}H^{2}
		= \int_{\Sigma_t}\bigl(
		2\mathrm{R}^{\top}-2\mathrm{R}
		+4\operatorname{Ric}(\nu,\nu)+2|\mathring{h}|^{2}\bigr).
		\]
		
		On $\partial\Sigma_t=\Sigma_t\cap\partial M$, let $\bar{\nu}$ be
		the inward normal to $\partial\Sigma_t$ within $\Sigma_t$.
		Orthogonality of $\Sigma_t$ and $\partial M$ forces
		$\bar{\nu}=\mathbf{n}$.  With $\dot\gamma$ the unit tangent,
		\begin{equation}
			\int_{\partial\Sigma_t}k_g
			= \int_{\partial\Sigma_t}
			\langle\nabla_{\dot\gamma}\dot\gamma,\bar{\nu}\rangle
			= \int_{\partial\Sigma_t}
			\operatorname{II}(\dot\gamma,\dot\gamma)\ge0,
			\label{eq:kg-II}
		\end{equation}
		by convexity $\operatorname{II}\ge0$.
		
		\medskip\noindent\textit{Case~1.}
		$\int_{\Sigma_t}\mathrm{R}^{\top}+2\int_{\Sigma_t\cap\partial M}k_g\le0$.
		Then $\int_{\Sigma_t}\mathrm{R}^{\top}\le0$ and
		using the inequality $\operatorname{Ric}(\nu,\nu)-\mathrm{R}\le0$,
		which follows from $\operatorname{Ric}\ge0$ and pinching,
		\begin{align*}
			\int_{\Sigma_t}H^{2}
			&\le \int_{\Sigma_t}
			\bigl(2(\operatorname{Ric}(\nu,\nu)-\mathrm{R})
			+2\operatorname{Ric}(\nu,\nu)
			+2|\mathring{h}|^{2}\bigr) \\
			&\le 2\int_{\Sigma_t}
			\bigl(\operatorname{Ric}(\nu,\nu)+|\mathring{h}|^{2}\bigr),
		\end{align*}
		which is~\eqref{eq:Willmore-case1}.
		
		\medskip\noindent\textit{Case~2.}
		$\int_{\Sigma_t}\mathrm{R}^{\top}+2\int_{\Sigma_t\cap\partial M}k_g\ge4\pi$.
		Then $\int_{\Sigma_t}2\mathrm{R}^{\top}
		\ge8\pi-4\int_{\partial\Sigma_t}k_g$.
		Using $\operatorname{Ric}(\nu,\nu)\ge0$ and
		$|\mathring{h}|^{2}\ge0$,
		\[
		\int_{\Sigma_t}H^{2}
		\ge \int_{\Sigma_t}(2\mathrm{R}^{\top}-2\mathrm{R})
		\ge 8\pi-4\int_{\partial\Sigma_t}k_g
		-2\int_{\Sigma_t}\mathrm{R}.
		\]
		Pinching $\operatorname{Ric}\ge\rho\mathrm{R}g$ gives
		$\mathrm{R}\le\rho^{-1}\operatorname{Ric}(\nu,\nu)$, so
		\[
		\int_{\Sigma_t}H^{2}
		\ge 8\pi-4\int_{\partial\Sigma_t}k_g
		-\frac{2}{\rho}
		\int_{\Sigma_t}\operatorname{Ric}(\nu,\nu),
		\]
		which is~\eqref{eq:Willmore-case2}.
	\end{proof}
	
	\begin{lem}\label{lem:asymptotic}
		Under the assumptions of Theorem~\ref{thm:main}, suppose that
		$F_p(t_*)<2\pi$ at some $t_*\ge0$, where for $t_*=0$ we use the
		geometric initial value and for $t_*>0$ the locally absolutely
		continuous representative.  Then there exist $T_0\ge t_*$ and
		$C=C(T_0)>0$
		such that for a.e.\ $t\ge T_0$,
		\[
		F_p(t)\le C\,e^{-\frac{2}{3-p}t}.
		\]
	\end{lem}
	\begin{proof}
		The argument follows the strategy of
		Benatti--Le\'{o}n Quir\'{o}s--Oronzio--Pluda~\cite[Lemma~3.3]{BPO};
		the main difference is the additional control of the boundary
		terms arising from $\Sigma_t\cap\partial M$, which we handle
		via the second-fundamental-form pinching condition.
		By Lemmas~\ref{lem:Fp-derivative} and~\ref{lem:Fp-endpoint},
		$F_p(t)\le F_p(t_*)<2\pi$ for every $t\ge t_*$, up to the choice
		of its absolutely continuous representative.
		
		Theorem~\ref{thm:weak-GB} gives
		$\int_{\Sigma_t}\mathrm{R}^{\top}
		+2\int_{\Sigma_t\cap\partial M}k_g\in4\pi\mathbb{Z}$.
		
		\medskip\noindent\textit{Case~1.}
		$\displaystyle\int_{\Sigma_t}\mathrm{R}^{\top}
		+2\int_{\Sigma_t\cap\partial M}k_g\le0$.
		From~\eqref{eq:Fp-derivative} and~\eqref{eq:Willmore-case1},
		\begin{align*}
			-2(3-p)F_p'(t)
			&\ge 2\int_{\Sigma_t}
			\bigl(\operatorname{Ric}(\nu,\nu)+|\mathring{h}|^{2}\bigr) \\
			&\ge \int_{\Sigma_t}H^{2}
			\ge 4F_p(t).
		\end{align*}
		
		\medskip\noindent\textit{Case~2.}
		$\displaystyle\int_{\Sigma_t}\mathrm{R}^{\top}
		+2\int_{\Sigma_t\cap\partial M}k_g\ge4\pi$.
		From~\eqref{eq:Fp-derivative}, keeping the boundary term,
		\begin{align*}
			-2(3-p)F_p'(t)
			&\ge \int_{\Sigma_t}
			\Bigl(2\operatorname{Ric}(\nu,\nu)
			+\frac{3-p}{p-1}
			\bigl(H-\tfrac{2|\nabla w_p|}{3-p}\bigr)^{2}
			\Bigr)\dd\mathcal{H}^{2} \nonumber\\
			&\qquad
			+2\int_{\Sigma_t\cap\partial M}
			\operatorname{II}\!\Bigl(
			\frac{\nabla w_p}{|\nabla w_p|},
			\frac{\nabla w_p}{|\nabla w_p|}
			\Bigr)\dd\mathcal{H}^{1}.
		\end{align*}
		Denote by $\rho$ the Ricci-pinching constant
		($\operatorname{Ric}\ge\rho\mathrm{R}g$) and by
		$\lambda$ the second-fundamental-form pinching constant
		($\operatorname{II}\ge\lambda H g$).  Both
		pinching conditions become weaker when the corresponding
		constant is decreased, since $\mathrm{R}\ge0$ and
		$H\ge0$.  Hence, replacing $\rho$ and
		$\lambda$ by smaller positive constants if necessary, we may
		assume $\lambda\le\frac12$ and $\rho\le\frac12\lambda$.  Then
		\[
		\rho
		\le \frac12\,\lambda
		\le \frac14
		< \frac13
		< \frac{3-p}{p-1},
		\]
		where the last inequality follows from $p\in(1,2)$.
		
		Replacing $\frac{3-p}{p-1}$ by $\rho$ (using
		$\rho\le\frac{3-p}{p-1}$) and expanding the square,
		\begin{align*}
			-2(3-p)F_p'(t)
			&\ge \int_{\Sigma_t}
			\Bigl(2\operatorname{Ric}(\nu,\nu)
			+\rho H^{2}
			-\frac{4\rho}{3-p}H|\nabla w_p|
			+\frac{4\rho}{(3-p)^{2}}|\nabla w_p|^{2}
			\Bigr)\dd\mathcal{H}^{2} \\
			&\qquad
			+2\int_{\Sigma_t\cap\partial M}
			\operatorname{II}\!\Bigl(
			\frac{\nabla w_p}{|\nabla w_p|},
			\frac{\nabla w_p}{|\nabla w_p|}
			\Bigr)\dd\mathcal{H}^{1}.
		\end{align*}
		By~\eqref{eq:Willmore-case2},
		\begin{align*}
			-2(3-p)F_p'(t)
			&\ge \rho\Bigl(
			8\pi
			-\frac{4}{3-p}\int_{\Sigma_t}
			\bigl(H|\nabla w_p|
			-\tfrac{|\nabla w_p|^{2}}{3-p}\bigr)
			\dd\mathcal{H}^{2}
			-4\int_{\Sigma_t\cap\partial M}k_g\,\dd\mathcal{H}^{1}
			\Bigr) \\
			&\qquad
			+2\int_{\Sigma_t\cap\partial M}
			\operatorname{II}\!\Bigl(
			\frac{\nabla w_p}{|\nabla w_p|},
			\frac{\nabla w_p}{|\nabla w_p|}
			\Bigr)\dd\mathcal{H}^{1} \\
			&= \rho\bigl(8\pi-4F_p(t)\bigr)
			+2\int_{\Sigma_t\cap\partial M}
			\Bigl[
			\operatorname{II}\!\Bigl(
			\frac{\nabla w_p}{|\nabla w_p|},
			\frac{\nabla w_p}{|\nabla w_p|}
			\Bigr)
			-2\rho\,k_g
			\Bigr]\dd\mathcal{H}^{1}.
		\end{align*}
		On $\partial\Sigma_t$, $\nabla w_p/|\nabla w_p|$ is tangent to
		$\partial M$ and orthogonal to the unit tangent $\dot\gamma$
		of $\partial\Sigma_t$.  Using~\eqref{eq:kg-II},
		$k_g=\operatorname{II}(\dot\gamma,\dot\gamma)$.
		From the second-fundamental-form pinching
		$\operatorname{II}\ge\lambda H g$,
		the mean curvature decomposes as
		$H
		=\operatorname{II}(\frac{\nabla w_p}{|\nabla w_p|},
		\frac{\nabla w_p}{|\nabla w_p|})
		+\operatorname{II}(\dot\gamma,\dot\gamma)$.
		Hence
		\[
		\operatorname{II}\!\Bigl(
		\frac{\nabla w_p}{|\nabla w_p|},
		\frac{\nabla w_p}{|\nabla w_p|}
		\Bigr)
		\ge \lambda H
		\ge \lambda\,\operatorname{II}(\dot\gamma,\dot\gamma).
		\]
		Since $\rho\le\frac12\lambda$, we have
		$\lambda\ge2\rho$, and therefore
		\[
		\operatorname{II}\!\Bigl(
		\frac{\nabla w_p}{|\nabla w_p|},
		\frac{\nabla w_p}{|\nabla w_p|}
		\Bigr)
		-2\rho\,\operatorname{II}(\dot\gamma,\dot\gamma)
		\ge 0.
		\]
		Consequently
		\[
		-2(3-p)F_p'(t)
		\ge \rho\bigl(8\pi-4F_p(t)\bigr).
		\]
		
		Combining Case~1 and Case~2, for almost every
		$t\in[0,+\infty)$,
		\begin{equation}
			F_p'(t)
			\le \max\Bigl\{
			-\frac{2}{3-p}F_p(t),\;
			-\frac{2\rho}{3-p}
			\bigl(2\pi-F_p(t)\bigr)
			\Bigr\}.
			\label{eq:ODE-max}
		\end{equation}
		
		\noindent
		If the inequality
		$F_p(t)\ge\rho(2\pi-F_p(t))$
		held for every $t\ge t_*$, we would have
		\[
			F_p(t)\ge\frac{2\pi\rho}{1+\rho}>0
			\qquad\text{and}\qquad
			F_p'(t)
			\le -\frac{2\rho}{3-p}
			\bigl(2\pi-F_p(t_*)\bigr)
			< 0
		\]
		for almost every $t\ge t_*$, which forces
		$F_p(t)\to-\infty$ as $t\to+\infty$, contradicting
		$F_p\ge0$.  Therefore there must exist
		$T_0\in[t_*,+\infty)$ such that
		\[
		F_p(T_0) \le \rho\bigl(2\pi-F_p(T_0)\bigr).
		\]
		Since $F_p$ is nonincreasing (Lemma~\ref{lem:Fp-derivative}),
		this inequality persists for all $t\ge T_0$, and the second
		branch of~\eqref{eq:ODE-max} never applies beyond $T_0$.
		Hence
		\[
		F_p'(t) \le -\frac{2}{3-p}F_p(t)
		\qquad\text{for a.e.\ }t\ge T_0,
		\]
		which after integration yields
		\[
		F_p(t) \le C(T_0)\,e^{-\frac{2}{3-p}t}
		\qquad\text{for all }t\ge T_0,
		\]
		with $C(T_0)=F_p(T_0)\,e^{\frac{2}{3-p}T_0}$.
	\end{proof}
	
	\section{Proof of the Main Theorem}
	
	\subsection{Willmore energy of small Fermi coordinate balls}
	
	Let $p\in\partial M$.  In a neighbourhood of $p$ we introduce
	Fermi coordinates $(x^{1},x^{2},y)$ (see e.g.\
	\cite[Chapter~2]{Gray}), constructed in two steps.
	First, at a chosen point $p\in\partial M$, pick an orthonormal
	frame of $T_p\partial M$ and take geodesic normal coordinates on
	$\partial M$ centred at $p$---this yields coordinates
	$(x^{1},x^{2})$ on $\partial M$ near $p$.  Second, from each
	point $(\bar{x},0)\in\partial M$, follow the unit-speed geodesic
	emanating in the direction of the inward unit normal
	$\mathbf{n}$ for distance $y\ge0$.  The resulting coordinates
	satisfy $\partial M=\{y=0\}$ with $y>0$ the interior of $M$.
	By Gauss' lemma, the normal geodesics remain orthogonal to the
	level sets $y=\mathrm{const}$; hence the metric takes the form
	\[
	g = g_{ij}(x,y)\,\dd x^{i}\dd x^{j} + \dd y^{2},
	\qquad i,j=1,2,
	\]
	with $g_{iy}=0$ (orthogonality) and $g_{yy}=1$ (unit speed).
	The inward unit normal to $\partial M$ is
	$\mathbf{n}=\partial_{y}$.
	
	\begin{lem}\label{lem:Fermi-free-boundary}
		For $r>0$ sufficiently small, the relative Fermi coordinate ball
		\[
		B_r(p) = \{\, (x,y)\in M : |x|^{2}+y^{2}<r^{2} \,\}
		\]
		is well defined.  Its interior boundary
		\[
		S_r^{+}:=\partial_{\operatorname{int}}\overline{B_r(p)}
		=\{\,(x,y)\in M:|x|^2+y^2=r^2\,\}
		\]
		is a smooth
		surface, and it meets $\partial M$ orthogonally along
		$S_r^{+}\cap\partial M$; thus $B_r(p)$ satisfies the free
		boundary condition.
	\end{lem}
	\begin{proof}
		The interior boundary has interior part
		$S_r^{+}\cap\operatorname{int}(M)=\{F=r^{2},\,y>0\}$, where
		$F(x,y)=|x|^{2}+y^{2}$.  Its gradient is
		\[
		\nabla F = 2\sum_{i,j=1}^{2} g^{ij}x_j\,\partial_i + 2y\,\partial_y.
		\]
		At a point on the intersection $S_r^{+}\cap\partial M$, we
		have $y=0$ and $|x|=r$.  The normal to $S_r^{+}$ is $\nabla F$,
		and the normal to $\partial M$ is $\partial_y$.  Their inner
		product is
		\[
		g(\nabla F,\partial_y)
		= 2\sum_{i,j=1}^{2} g^{ij}x_j\,g_{iy} + 2y\,g_{yy}
		= 0,
		\]
		since $g_{iy}=0$ holds globally in Fermi coordinates and
		$y=0$ on $\partial M$.  Thus $S_r^{+}\perp\partial M$.
	\end{proof}
	
	In the closed-manifold setting, Benatti--Le\'{o}n
	Quir\'{o}s--Oronzio--Pluda~\cite{BPO} and
	Benatti--Mantegazza--Oronzio--Pluda~\cite{BMPO}
	choose as initial domain a small geodesic ball $B_r(o)$ centred
	at an interior point $o$ where the scalar curvature is positive.
	For such geodesic spheres, the Willmore energy admits the
	asymptotic expansion (see~\cite[Proposition~3.1]{Mondino})
	\[
	\int_{\partial B_r(o)} H^{2}\,\dd\mu
	= 16\pi - \frac{8\pi}{3}\mathrm{R}(o)\,r^{2} + O(r^{3}),
	\qquad r\to0^{+},
	\]
	yielding $\int_{\partial B_r(o)}H^{2}<16\pi$ for $r$ small.
	
	In our boundary setting, however, a geodesic ball centred at a
	point $p\in\partial M$ would \emph{not} satisfy the free
	boundary condition (its boundary would intersect $\partial M$
	at a non-orthogonal angle in general).  The free boundary
	condition is essential for the $C^{1,\beta}$ regularity of the
	$p$-harmonic function $u_p$ up to the corner
	$\partial\Omega\cap\partial M$, which in turn is needed for the
	weak Gauss--Bonnet theorem and the monotonicity formulas.
	We therefore replace the geodesic ball by the Fermi coordinate
	ball $B_r(p)$, which by Lemma~\ref{lem:Fermi-free-boundary} is a
	genuine free boundary domain.
	
	\begin{lem}\label{lem:Willmore-Fermi}
		Let $p\in\partial M$ and let $B_r(p)$ be the Fermi coordinate
		ball as above.  Then, as $r\to0^{+}$,
		\begin{equation}
			\int_{S_r^{+}} H^{2}\,\dd\sigma
			= 8\pi - 4\pi H_{0}\,r + O(r^{2}),
			\label{eq:Willmore-Fermi}
		\end{equation}
		where $S_r^{+}=\partial_{\operatorname{int}}\overline{B_r(p)}$ and
		$H_{0}=H(p)$.  In particular, if
		$H_{0}>0$, then for $r$ sufficiently small
		$\int_{S_r^{+}}H^{2}<8\pi$.
	\end{lem}
	\begin{proof}
		Fix Fermi coordinates $(x^{1},x^{2},y)$ centred at the given point
		$q\in\partial M$.  To avoid confusing the centre with the exponent
		used elsewhere, we write $q$ instead of $p$ throughout this proof.
		The metric satisfies
		\[
		g_{ij}=\delta_{ij}-2\operatorname{II}_{ij}y+O(r^{2}),
		\qquad g_{iy}=0,\qquad g_{yy}=1,
		\]
		with inverse
		\[
		g^{ij}=\delta^{ij}+2\operatorname{II}^{ij}y+O(r^{2}),
		\qquad g^{yy}=1,\qquad g^{iy}=0,
		\]
		where $r^{2}=|x|^{2}+y^{2}$.  Set $F(x,y)=|x|^{2}+y^{2}$ and
		parametrise $S_r^{+}=\{F=r^{2},\ y\ge0\}$ by
		\[
		(x^{1},x^{2},y)
		=r(\sin\theta\cos\varphi,
		\sin\theta\sin\varphi,\cos\theta),
		\]
		where $\theta\in[0,\pi/2]$ and $\varphi\in[0,2\pi]$.  Write
		$\omega=(\cos\varphi,\sin\varphi)$,
		$\omega^{\perp}=(-\sin\varphi,\cos\varphi)$, and use the shorthand
		\[
		\operatorname{II}=\sum_{i,j}\operatorname{II}_{ij}\omega^{i}\omega^{j},
		\qquad
		\operatorname{II}^{\perp}=\sum_{i,j}\operatorname{II}_{ij}
		\omega^{\perp i}\omega^{\perp j},
		\]
		together with $S=\sin^{2}\theta\cos\theta$ and
		$H_{0}=\operatorname{tr}\operatorname{II}$.

		\medskip\noindent\textit{Gradient and Laplacian of $F$.}
		Since
		$\nabla F=\sum_{\alpha,\beta}g^{\alpha\beta}
		(\partial_{\beta}F)\partial_{\alpha}$, we have
		\[
		\nabla F
		=\sum_{i=1}^{2}\Bigl(2x^{i}
		+4\sum_{j,k}\operatorname{II}^{ij}\delta_{jk}x^{k}y
		+O(r^{3})\Bigr)\partial_{i}+2y\,\partial_{y}.
		\]
		Consequently,
		\begin{align*}
		|\nabla F|_g^{2}
		&=\sum_{i,j}g_{ij}(\nabla^{i}F)(\nabla^{j}F)+(\nabla^{y}F)^{2}\\
		&=\sum_{i,j}(\delta_{ij}-2\operatorname{II}_{ij}y)
		\bigl(2x^{i}+4\sum_{k}\operatorname{II}^{ik}x^{k}y\bigr)
		\bigl(2x^{j}+4\sum_{\ell}\operatorname{II}^{j\ell}x^{\ell}y\bigr)
		+4y^{2}+O(r^{4}).
		\end{align*}
		The leading term is $4|x|^{2}$, the metric correction contributes
		$-8\sum_{i,j}\operatorname{II}_{ij}x^{i}x^{j}y$, and the two
		inverse-metric cross-terms contribute
		$16\sum_{i,j}\operatorname{II}_{ij}x^{i}x^{j}y$.  Hence
		\begin{equation*}
		|\nabla F|_g^{2}=4r^{2}+8\operatorname{II}\,r^{3}S+O(r^{4}),
		\end{equation*}
		and therefore
		\begin{equation}\label{eq:ap-grad}
		|\nabla F|_g=2r+2r^{2}S\operatorname{II}+O(r^{3}).
		\end{equation}

		For the Laplacian, the second-derivative term is
		\[
		g^{\alpha\beta}\partial_{\alpha\beta}F
		=2\sum_{i=1}^{2}g^{ii}+2
		=6+4H_0r\cos\theta+O(r^2).
		\]
		Using $\Gamma^{y}_{ij}=\operatorname{II}_{ij}+O(r)$, the only
		first-order Christoffel contribution is
		\[
		-g^{ij}\Gamma^{y}_{ij}\partial_yF
		=-2H_0r\cos\theta+O(r^2).
		\]
		The terms containing tangential derivatives of $F$ are $O(r^2)$
		because $(x^1,x^2)$ are geodesic normal coordinates on
		$\partial M$ at $q$.  Thus
		\begin{equation*}
		\Delta_gF=6+2H_0r\cos\theta+O(r^2).
		\end{equation*}

		\medskip\noindent\textit{Hessian and mean curvature.}
		Equation~\eqref{eq:ap-grad} gives
		$|\nabla F|_g^{-1}=(2r)^{-1}-\frac12S\operatorname{II}+O(r)$.
		The unit normal $\nu=\nabla F/|\nabla F|_g$ therefore satisfies
		\begin{align*}
		\nu^{i}&=\sin\theta\,\omega^{i}
		+2r\sin\theta\cos\theta\sum_j\operatorname{II}^{ij}\omega^j
		-r\sin\theta\,S\operatorname{II}\,\omega^i+O(r^2),\\[2pt]
		\nu^{y}&=\cos\theta-rS\operatorname{II}\cos\theta+O(r^2).
		\end{align*}
		For $\operatorname{Hess}_gF(\nu,\nu)$, the ordinary
		second-derivative part is
		\[
		2\sum_i(\nu^i)^2+2(\nu^y)^2
		=2+4rS\operatorname{II}+O(r^2).
		\]
		Indeed,
		\begin{align*}
		\sum_i(\nu^i)^2
		&=\sin^2\theta+4rS\operatorname{II}
		-2rS\operatorname{II}\sin^2\theta+O(r^2),\\
		(\nu^y)^2
		&=\cos^2\theta-2rS\operatorname{II}\cos^2\theta+O(r^2).
		\end{align*}
		The Christoffel part
		$-\nu^\alpha\nu^\beta\Gamma^\gamma_{\alpha\beta}
		\partial_\gamma F$ receives first-order contributions from the
		$(i,j,y)$, $(i,y,k)$, and $(y,i,k)$ terms.  Their sum is
		$(-2+2+2)rS\operatorname{II}=2rS\operatorname{II}$.  It follows that
		\begin{equation*}
		\operatorname{Hess}_gF(\nu,\nu)
		=2+6rS\operatorname{II}+O(r^2).
		\end{equation*}
		Since
		$H=(\Delta_gF-\operatorname{Hess}_gF(\nu,\nu))/|\nabla F|_g$,
		we obtain
		\begin{align*}
		H&=\bigl(4+2r\cos\theta
		(H_0-3\sin^2\theta\,\operatorname{II})\bigr)
		\bigl(\tfrac1{2r}-\tfrac12S\operatorname{II}\bigr)+O(r)\\
		&=\frac2r+H_0\cos\theta-5S\operatorname{II}+O(r).
		\end{align*}

		\medskip\noindent\textit{Area element and integration.}
		The induced metric on $S_r^+$ has area element
		\begin{equation*}
		\sqrt{\det h}
		=r^2\sin\theta\bigl[1-r\cos^3\theta\,\operatorname{II}
		-r\cos\theta\,\operatorname{II}^{\perp}\bigr]+O(r^4).
		\end{equation*}
		Multiplying the preceding expansions yields
		\begin{align*}
		H^2\sqrt{\det h}
		&=4\sin\theta
		-4r\sin\theta\cos^3\theta\,\operatorname{II}
		-4r\sin\theta\cos\theta\,\operatorname{II}^{\perp}\\
		&\quad+4r\sin\theta H_0\cos\theta
		-20r\sin^3\theta\cos\theta\,\operatorname{II}+O(r^2).
		\end{align*}
		Since
		\[
		\int_0^{2\pi}\operatorname{II}\,\dd\varphi
		=\int_0^{2\pi}\operatorname{II}^{\perp}\,\dd\varphi
		=\pi H_0,
		\]
		angular integration gives
		\[
		\int_0^{2\pi}H^2\sqrt{\det h}\,\dd\varphi
		=8\pi\sin\theta
		-16\pi rH_0\sin^3\theta\cos\theta+O(r^2).
		\]
		Finally,
		\begin{align*}
		\int_{S_r^+}H^2\,\dd\sigma
		&=8\pi\int_0^{\pi/2}\sin\theta\,\dd\theta
		-16\pi rH_0\int_0^{\pi/2}
		\sin^3\theta\cos\theta\,\dd\theta+O(r^2)\\
		&=8\pi-4\pi H_0r+O(r^2),
		\end{align*}
		because the two displayed one-dimensional integrals are $1$ and
		$1/4$, respectively.  This proves~\eqref{eq:Willmore-Fermi}.
	\end{proof}

	The Fermi coordinate spheres above are needed before the boundary
	is known to be totally geodesic.  Once
	$\operatorname{II}\equiv0$, we may instead use genuine geodesic
	hemispheres centred at a boundary point.  The following lemma records
	the two properties that will be used in the rigidity argument.

	\begin{lem}[Free-boundary geodesic hemispheres]
		\label{lem:free-boundary-geodesic-hemisphere}
		Suppose that $\operatorname{II}\equiv0$ on $\partial M$ and
		$\operatorname{Ric}\ge0$.  For every $q\in\partial M$ and all
		sufficiently small $r>0$, there is a smooth half-ball
		$\Omega_r(q)$ centred at $q$ whose interior boundary
		$\Sigma_r(q)$ is a free-boundary geodesic hemisphere and whose
		exterior $M\setminus\Omega_r(q)$ is connected.  Moreover,
		\begin{equation}\label{eq:hemisphere-Willmore-bound}
			\int_{\Sigma_r(q)}H^2\,\dd\mathcal{H}^2\le8\pi.
		\end{equation}
		It also has the second-order expansion
		\begin{equation}\label{eq:hemisphere-Willmore-expansion}
			\int_{\Sigma_r(q)}H^2\,\dd\mathcal{H}^2
			=8\pi-\frac{4\pi}{3}\mathrm{R}(q)r^2+o(r^2).
		\end{equation}
	\end{lem}

	\begin{proof}
		Let $\mathbf{n}_q$ be the inward unit normal to $\partial M$ at
		$q$ and put
		\[
		\mathbb{S}_q^+
		=\{\theta\in T_qM:|\theta|=1,
		\ \langle\theta,\mathbf{n}_q\rangle\ge0\}.
		\]
		For $r$ below the one-sided normal radius of $q$, define directly
		in $M$
		\[
		\Omega_r(q)
		=\exp_q\{s\theta:0\le s\le r,\ \theta\in\mathbb{S}_q^+\},
		\qquad
		\Sigma_r(q)=\exp_q(r\mathbb{S}_q^+).
		\]
		Because $\partial M$ is totally geodesic, a geodesic issuing from
		$q$ with initial velocity in $T_q\partial M$ remains in
		$\partial M$ for small time.  Hence the equator
		$\partial\mathbb{S}_q^+$ is mapped into $\partial M$, and
		$\Omega_r(q)$ is a smooth half-ball with interior boundary
		$\Sigma_r(q)$.  Its exterior is connected: any path entering this
		coordinate half-ball can be rerouted along a slightly larger
		hemisphere in the same normal coordinate chart.  Along the equator,
		the radial vector
		$\partial_r$ is tangent to $\partial M$.  By the Gauss lemma it is
		normal to $\Sigma_r(q)$; thus the normals of $\Sigma_r(q)$ and
		$\partial M$ are orthogonal, which is precisely the free-boundary
		condition.  This construction uses only the one-sided exponential
		map and does not require a metric double.

		We next prove~\eqref{eq:hemisphere-Willmore-bound}.  Along every
		radial geodesic, the mean curvature of the distance spheres
		satisfies the traced Riccati equation
		\[
		\partial_rH=-|h|^2-\operatorname{Ric}(\partial_r,\partial_r)
		\le-\frac12H^2.
		\]
		Since $rH\to2$ as $r\to0^+$, comparison with $2/r$ gives
		$H\le2/r$.  For $r$ sufficiently small the expansion
		$H=2/r+O(r)$ also gives $H>0$, and hence $H^2\le4/r^2$.
		If $J(r,\theta)$ denotes the radial area density for
		$\theta\in\mathbb{S}_q^+$, then
		\[
		\partial_r\log\frac{J(r,\theta)}{r^2}
		=H-\frac2r\le0,
		\qquad
		\lim_{r\to0^+}\frac{J(r,\theta)}{r^2}=1.
		\]
		It follows that $J(r,\theta)\le r^2$.  Since
		$|\mathbb{S}_q^+|=2\pi$, we obtain directly
		\[
		\int_{\Sigma_r(q)}H^2\,\dd\mathcal{H}^2
		\le\frac4{r^2}\int_{\mathbb{S}_q^+}J(r,\theta)\,\dd\theta
		\le8\pi,
		\]
		which proves~\eqref{eq:hemisphere-Willmore-bound}.

		Finally, the standard Jacobi-field expansions, which only use the
		curvature tensor at $q$, give
		\[
		H(r,\theta)=\frac2r-\frac r3
		\operatorname{Ric}_q(\theta,\theta)+o(r),
		\]
		and
		\[
		J(r,\theta)=r^2\left(1-\frac{r^2}{6}
		\operatorname{Ric}_q(\theta,\theta)+o(r^2)\right).
		\]
		Thus $H^2J=4-2r^2\operatorname{Ric}_q(\theta,\theta)+o(r^2)$.
		The function
		$\theta\mapsto\operatorname{Ric}_q(\theta,\theta)$ is even.
		Consequently every antipodal pair contributes once to the
		hemisphere integral, and
		\[
		\int_{\mathbb{S}_q^+}\operatorname{Ric}_q(\theta,\theta)
		\,\dd\theta=\frac{2\pi}{3}\mathrm{R}(q).
		\]
		Integrating the expansion of $H^2J$ over $\mathbb{S}_q^+$ gives
		\eqref{eq:hemisphere-Willmore-expansion}.
	\end{proof}
	
	\subsection{Completion of the proof}
	
	\begin{proof}[Proof of Theorem~\ref{thm:main}]
		Suppose first that $M$ is orientable.
		Assume, by contradiction, that
		$\operatorname{II}\not\equiv0$ on $\partial M$.
		Then there exists $p\in\partial M$ with $H_0(p)>0$.
		
		By Lemma~\ref{lem:Willmore-Fermi}, for $r$ sufficiently small
		the Fermi coordinate ball $B_r(p)$ satisfies
		$\int_{S_r^{+}}H^{2}<8\pi$, and consequently
		$F_p(0)<2\pi$ by~\eqref{eq:Fp-bound}.
		For such a sufficiently small embedded half-ball,
		$M\setminus\overline{B_r(p)}$ is connected.  Set
		$\Omega=\overline{B_r(p)}$ and let $w_p$ be the solution
		of~\eqref{eq:wp}.  By the Hopf boundary-point lemma, $0$ is a
		regular value of $w_p$, so $F_p$ is well defined at $t=0$.
		
		By Lemma~\ref{lem:asymptotic}, there exist $T_0\ge0$ and
		$C_1>0$ such that $F_p(t)\le C_1 e^{-\frac{2}{3-p}t}$ and
		$G_p(t)\le F_p(t)\le C_1 e^{-\frac{2}{3-p}t}$ for a.e.\
		$t\ge T_0$.
		
		From Lemma~\ref{lem:capacity}, $C_p(\Sigma_t)=e^{t}C_p(\Sigma_0)$.
		By H\"older's inequality, writing
		$|\nabla w_p|^{p-1}
		=|\nabla w_p|^{2p/3}\,|\nabla w_p|^{(p-3)/3}$
		with exponents $r=3/p$ and $s=3/(3-p)$,
		for which $1/r+1/s=1$,
		\begin{align*}
			e^{t}C_p(\Sigma_0)
			&= \frac{1}{4\pi}\int_{\Sigma_t}
			\Bigl(\frac{|\nabla w_p|}{3-p}\Bigr)^{p-1}\dd\mathcal{H}^{2} \\
			&\le \frac{1}{4\pi(3-p)^{p-1}}
			\Bigl(\int_{\Sigma_t}|\nabla w_p|^{2}\dd\mathcal{H}^{2}
			\Bigr)^{\!\frac{p}{3}}
			\Bigl(\int_{\Sigma_t}|\nabla w_p|^{-1}\dd\mathcal{H}^{2}
			\Bigr)^{\!\frac{3-p}{3}}.
		\end{align*}
		By Lemma~\ref{lem:asymptotic} and Lemma~\ref{lem:Fp-Gp-relation},
		$G_p(t)=\int_{\Sigma_t}|\nabla w_p|^{2}/(3-p)^{2}
		\le C_1 e^{-\frac{2}{3-p}t}$ for a.e.\ $t\ge T_0$.
		Hence, for a.e.\ $t\ge T_0$,
		\[
		\int_{\Sigma_t}|\nabla w_p|^{-1}\dd\mathcal{H}^{2}
		\ge C_2\,e^{\frac{9-p}{(3-p)^{2}}t},
		\]
		with $C_2>0$ depending on $C_p(\Sigma_0)$, $p$, and $C_1$.
		
		By the coarea formula,
		\[
		\frac{\dd}{\dd t}\operatorname{Vol}\bigl(
		\{w_p\le t\}\setminus\{|\nabla w_p|=0\}\bigr)
		= \int_{\Sigma_t}|\nabla w_p|^{-1}\dd\mathcal{H}^{2}
		\ge C_2\,e^{\frac{9-p}{(3-p)^{2}}t}
		\]
		for a.e.\ $t\ge T_0$.  Set
		$R_t = \sup\{\,d(x,o) : w_p(x)\le t\,\}$.
		Integrating over $[T_0,T_1]$ for $T_1>T_0$, set
		$C_3=C_2(3-p)^2/(9-p)>0$.  The superquadratic volume growth
		condition then gives
		\[
		C_3\bigl(e^{\frac{9-p}{(3-p)^{2}}T_1}
		- e^{\frac{9-p}{(3-p)^{2}}T_0}\bigr)
		\le \operatorname{Vol}(\{w_p\le T_1\})
		\le \operatorname{Vol}(B_{R_{T_1}}(o))
		\le C_{\mathrm{vol}}\,R_{T_1}^{1+\alpha}.
		\]
		
		For every $x\in\{w_p\le T_1\}$, we have
		$u_p(x)=e^{-w_p(x)/(p-1)}\ge e^{-T_1/(p-1)}$.  On the other hand,
		Proposition~\ref{prop:p-nonparabolic} gives, for $d(x,o)$ sufficiently
		large,
		\[
			u_p(x)\le C\,d(x,o)^{-\frac{\alpha+1-p}{p-1}}.
		\]
		Combining these inequalities and raising to the power $p-1$ yields
		\[
			d(x,o)^{\alpha+1-p}\le C e^{T_1}.
		\]
		After enlarging $C$ to cover the remaining bounded range of
		$d(x,o)$, this estimate holds for every
		$x\in\{w_p\le T_1\}$.  Taking the supremum over this entire sublevel
		set gives $R_{T_1}^{\alpha+1-p}\le C e^{T_1}$, hence
		$R_{T_1}^{1+\alpha}\le C\,e^{\frac{1+\alpha}{1+\alpha-p}T_1}$.
		Therefore
		\[
		C_3\bigl(e^{\frac{9-p}{(3-p)^{2}}T_1}
		- e^{\frac{9-p}{(3-p)^{2}}T_0}\bigr)
		\le C' e^{\frac{1+\alpha}{1+\alpha-p}T_1},
		\]
		for some $C'>0$.  For $T_1$ sufficiently large this forces
		$\frac{9-p}{(3-p)^{2}}\le\frac{1+\alpha}{1+\alpha-p}$, i.e.\
		$\alpha\le\frac{4}{5-p}$ (cross-multiplying yields
		$\alpha\,p(5-p)\le4p$).
		Since this holds for every $p\in(1,2)$,
		taking the infimum over $p$ gives
		$\alpha\le\inf_{p\in(1,2)}4/(5-p)=1$, contradicting
		the superquadratic volume growth hypothesis $\alpha>1$.
		Therefore $\operatorname{II}\equiv0$ on $\partial M$, so
		$\partial M$ is totally geodesic.

		We now prove directly that the interior is Ricci-flat.  Choose
		and fix $p\in(1,2)$ sufficiently close to $1$ that
		\begin{equation}\label{eq:choose-p-rigidity}
			\frac4{5-p}<\alpha.
		\end{equation}
		This is possible because $\alpha>1$.  Fix $q\in\partial M$.
		Lemma~\ref{lem:free-boundary-geodesic-hemisphere} provides
		$r_0(q)>0$ such that, for every $r\in(0,r_0(q))$, the half-ball
		$\Omega_r(q)$ and its interior boundary $\Sigma_r(q)$ have all the
		properties stated there.  For each such $r$, let
		$w_{p,r}$ be the corresponding potential and denote its monotone
		quantities by $F_{p,r}$ and $G_{p,r}$.  From
		\eqref{eq:Fp-bound} and
		\eqref{eq:hemisphere-Willmore-bound},
		\begin{equation}\label{eq:Fpr-at-most-2pi}
			F_{p,r}(0)\le\frac14
			\int_{\Sigma_r(q)}H^2\,\dd\mathcal{H}^2\le2\pi.
		\end{equation}
		If $F_{p,r}(0)<2\pi$, the volume-growth argument above, applied
		with this fixed $p$, would give
		$\alpha\le4/(5-p)$, contrary to
		\eqref{eq:choose-p-rigidity}.  Consequently equality holds
		throughout~\eqref{eq:Fpr-at-most-2pi}.  In particular,
		\begin{equation*}
			\int_{\Sigma_r(q)}H^2\,\dd\mathcal{H}^2=8\pi,
			\qquad
			\int_{\Sigma_r(q)}
			\left(\frac H2-\frac{|\nabla w_{p,r}|}{3-p}\right)^2
			\dd\mathcal{H}^2=0,
		\end{equation*}
		where we used the exact identity
		\[
		F_{p,r}(0)=\frac14\int_{\Sigma_r(q)}H^2\,\dd\mathcal{H}^2
		-\int_{\Sigma_r(q)}
		\left(\frac H2-\frac{|\nabla w_{p,r}|}{3-p}\right)^2
		\dd\mathcal{H}^2.
		\]
		Endpoint continuity, Lemma~\ref{lem:Fp-endpoint}, gives
		\begin{equation*}
			F_{p,r}(0+)=F_{p,r}(0)=2\pi.
		\end{equation*}
		Nor can $F_{p,r}$ decrease at a later time: if
		$F_{p,r}(t_*)<2\pi$ for some $t_*>0$, the same shifted argument
		would produce the same contradiction.  Since $F_{p,r}$ is
		nonincreasing and locally absolutely continuous on
		$(0,+\infty)$, we conclude that
		\begin{equation*}
			F_{p,r}(t)\equiv2\pi
			\qquad\text{on }[0,+\infty),
		\end{equation*}
		and therefore
		\begin{equation}\label{eq:Fprime-zero-rigidity}
			F_{p,r}'(t)=0
			\qquad\text{for a.e. }t>0.
		\end{equation}

		It remains to extract the equality information without assuming a
		smooth flow from the initial Dirichlet--Neumann edge.  Put
		\[
			U=M\setminus\Omega_r(q),\qquad
			w=w_{p,r},
			\qquad \mathcal{R}=\{x\in U:|\nabla w|(x)>0\}.
		\]
		The set $U$ is connected by the choice of $\Omega_r(q)$, while
		$\mathcal{R}$ is relatively open because $|\nabla w|$ is continuous.  Moreover,
		item~\ref{item:initial-edge} of
		Lemma~\ref{lem:eps-regularization}, together with the Hopf lemma,
		shows that $|\nabla w|$ is bounded away from zero in a one-sided collar of the
		initial face.  In particular, $\mathcal{R}$ is nonempty.

		We first prove that every dissipation density vanishes throughout
		$\mathcal{R}$.  On this set the equation is uniformly elliptic, so $w$
		is smooth in the interior and has the standard smooth Neumann
		regularity up to the physical boundary away from the initial edge.
		Suppose that one of the
		nonnegative continuous densities in
		\eqref{eq:Fp-derivative} were positive at some interior point of
		$\mathcal{R}$.  In a relatively compact submersion chart around that
		point, continuity would give a smaller coordinate cylinder on which
		this density is bounded below by a positive constant.  The local
		level patches in this cylinder have positive $\mathcal{H}^2$-measure
		for every level in a nonempty open interval.  Since
		$\operatorname{II}=0$, the boundary dissipation vanishes, and
		\eqref{eq:Fp-derivative} would imply
		$F_{p,r}'(t)<0$ for almost every $t$ in that interval.  This
		contradicts~\eqref{eq:Fprime-zero-rigidity}.  The same identities at a
		point of the physical boundary follow by continuity from interior
		points of the same submersion neighborhood.  Consequently, throughout
		$\mathcal{R}$,
		\begin{equation}\label{eq:equality-system-regular-set}
			\nabla^\top|\nabla w|=0,
			\qquad \mathring{h}=0,
			\qquad 2|\nabla w|=(3-p)H,
			\qquad \operatorname{Ric}(\nu,\nu)=0,
		\end{equation}
		where $\nu=\nabla w/|\nabla w|$.

		We next show that no critical point can lie outside the initial
		half-ball.  On $\mathcal{R}$, the equation
		$\Delta_p w=|\nabla w|^p$ gives the level-set identity
		\[
			H=|\nabla w|-(p-1)\frac{\partial_\nu|\nabla w|}{|\nabla w|}.
		\]
		Combining this with $H=2|\nabla w|/(3-p)$ from
		\eqref{eq:equality-system-regular-set}, we obtain
		\[
			(p-1)\frac{\partial_\nu|\nabla w|}{|\nabla w|}
			=|\nabla w|-\frac{2|\nabla w|}{3-p}
			=-\frac{p-1}{3-p}|\nabla w|,
			\qquad
			\partial_\nu|\nabla w|=-\frac{|\nabla w|^2}{3-p}.
		\]
		Since $\nabla^\top|\nabla w|=0$ and $\nabla w=|\nabla w|\nu$, it follows that
		\begin{equation}\label{eq:regular-set-gradient-law}
			\nabla|\nabla w|=-\frac{|\nabla w|}{3-p}\nabla w,
			\qquad
			\nabla\!\left(|\nabla w|e^{w/(3-p)}\right)=0
			\quad\text{on }\mathcal{R}.
		\end{equation}

		Let $V$ be a connected component of $\mathcal{R}$.  The second identity
		in~\eqref{eq:regular-set-gradient-law} shows that
		\[
			|\nabla w|e^{w/(3-p)}=C_V
			\qquad\text{on }V
		\]
		for some $C_V>0$.  We claim that $V$ is closed in $U$.  Indeed, if
		$x_j\in V$ and $x_j\to x\in U$, continuity of $|\nabla w|$ and $w$ gives
		\[
			|\nabla w|(x)e^{w(x)/(3-p)}=C_V>0.
		\]
		Thus $x\in\mathcal{R}$.  A connected relative coordinate neighborhood
		of $x$ contained in $\mathcal{R}$ intersects $V$ for all sufficiently
		large $j$, and hence lies in the same connected component; therefore
		$x\in V$.  Thus $V$ is closed in $U$.  Since manifolds with boundary
		are locally connected, every connected component of the relatively
		open set $\mathcal{R}$ is also open in $U$.  The connectedness of $U$
		and the nonemptiness of $\mathcal{R}$ now imply that $V=U$.  Hence
		\begin{equation*}
			\operatorname{Crit}(w_{p,r})\cap
			(M\setminus\Omega_r(q))=\varnothing.
		\end{equation*}

		Finally, the last identity in
		\eqref{eq:equality-system-regular-set} and the Ricci pinching give at
		every point of $U$
		\[
			0=\operatorname{Ric}(\nu,\nu)
			\ge\rho\mathrm{R}\ge0.
		\]
		Therefore $\mathrm{R}=0$ on $U$.  Since
		$\operatorname{Ric}\ge0$ and its trace is zero, all of its eigenvalues
		vanish.  Consequently,
		\begin{equation}\label{eq:exterior-Ricci-flat}
			\operatorname{Ric}\equiv0
			\qquad\text{on }M\setminus\Omega_r(q).
		\end{equation}
		In dimension three this is equivalent to the vanishing of the full
		Riemann curvature tensor.

		We now globalize~\eqref{eq:exterior-Ricci-flat} by shrinking the
		initial hemisphere about the fixed point $q$.  The preceding argument
		applies separately to every $r\in(0,r_0(q))$; although the potential
		$w_{p,r}$ depends on $r$, it gives
		\[
			\operatorname{Ric}\equiv0
			\qquad\text{on }M\setminus\Omega_r(q)
		\]
		for every such $r$.

		Let $x\in M\setminus\{q\}$.  Since $d(x,q)>0$, choose
		\[
			0<r<\min\{r_0(q),d(x,q)\}.
		\]
		Then $x\notin\Omega_r(q)$, and hence $\operatorname{Ric}_x=0$.
		Thus $\operatorname{Ric}\equiv0$ on $M\setminus\{q\}$.  Since the
		Ricci tensor is continuous up to the boundary, letting $x\to q$ gives
		$\operatorname{Ric}_q=0$.  Consequently,
		\begin{equation}\label{eq:global-flatness}
			\operatorname{Ric}\equiv0
			\qquad\text{on }M.
		\end{equation}
		In dimension three the Riemann tensor is determined by the Ricci
		tensor, so~\eqref{eq:global-flatness} implies that $(M,g)$ is flat.

		It remains only to identify the resulting complete flat manifold
		with boundary.  Let
		\[
			(\widehat{M},\widehat{g})
			=M_+\cup_{\partial M}M_-
		\]
		be the boundary double of $M$.  Since $(M,g)$ is flat and
		$\partial M$ is totally geodesic, the doubled metric $\widehat{g}$ is
		smooth and flat.  Moreover, the double of a complete manifold along
		its boundary is complete.

		Fix $o\in\partial M$.  Changing the base point in the
		superquadratic volume-growth assumption only changes its constants.
		Folding the double onto either copy gives
		\[
			d_{\widehat{M}}(o,x_\pm)=d_M(o,x)
			\qquad\text{for every }x\in M,
		\]
		where $x_\pm$ denotes either lift of $x$ to the double.  Hence
		\[
			\operatorname{Vol}_{\widehat{g}}
			\bigl(B_{\widehat{M}}(o,R)\bigr)
			=2\operatorname{Vol}_{g}\bigl(B_M(o,R)\bigr)
			\asymp R^{1+\alpha},
			\qquad \alpha>1.
		\]
		Thus $(\widehat{M},\widehat{g})$ also has superquadratic volume growth.
		As observed in~\cite[Proof of Theorem~1.3]{BPO}, Euclidean space is
		the only complete flat Riemannian $3$-manifold with superquadratic
		volume growth.  Therefore
		\[
			(\widehat{M},\widehat{g})\cong(\mathbb{R}^3,g_{\mathrm{Euc}}).
		\]

		Let $\iota$ be the isometric involution of $\widehat{M}$ that
		exchanges $M_+$ and $M_-$.  Its fixed-point set is precisely the
		seam $\partial M$.  Under the above Euclidean identification,
		$\iota$ becomes a nonidentity order-two Euclidean isometry whose
		fixed-point set is two-dimensional.  The fixed-point set of a
		Euclidean isometry is an affine subspace; hence this isometry is the
		reflection across an affine plane $P$.  Since $M$ is connected, its
		interior is connected; hence the interiors of $M_+$ and $M_-$ are
		the two components of $\widehat{M}\setminus\partial M$.  They are
		consequently identified with the two components of
		$\mathbb{R}^3\setminus P$.  In particular,
		\[
			(M,\partial M,g)\cong
			(\mathbb{R}^3_+,\partial\mathbb{R}^3_+,g_{\mathrm{Euc}}).
		\]
		
		This proves the theorem for $M$ orientable.  If $M$ is
		nonorientable, its orientable double cover inherits all the
		assumptions and must be $\mathbb{R}^3_+$.  Its nontrivial deck
		transformation would be a fixed-point-free order-two Euclidean
		isometry preserving the half-space.  This is impossible: for any
		$x\in\mathbb{R}^3_+$, the Euclidean midpoint of $x$ and its image is
		still in the convex half-space and is fixed by the involution.
	\end{proof}
	
	\appendix
	\section{Proof of Lemma~\ref{lem:eps-regularization}}
	\label{sec:eps-appendix}
	
	We give the free-boundary part of the proof of
	Lemma~\ref{lem:eps-regularization} in some detail; the interior part
	is standard and is only quoted.  Throughout the appendix
	$u:=u_p^\varepsilon=e^{-w_p^\varepsilon/(p-1)}$,
	$a(\sigma):=(\sigma^2+\varepsilon^2)^{(p-2)/2}$ and
	$A_\varepsilon(\xi):=a(|\xi|)\xi$, so that $u$ solves
	$\Delta_p^\varepsilon u=\operatorname{div}A_\varepsilon(\nabla u)=0$
	with $\partial_y u=0$ on $\partial M$ (in Fermi coordinates, where
	$\partial_y$ is the unit normal).  We denote by $K$ any compact subset
	of $\overline{D\setminus\Omega}
	\setminus\partial_{\operatorname{int}}D$, possibly meeting
	$\partial M$, and by $C(K)$ a constant depending on $K$ but not on
	$\varepsilon$.  Until the subsection devoted to the initial edge
	$E=\partial_{\operatorname{int}}\Omega\cap\partial M$, the set $K$ is
	also understood to stay a positive distance from $E$.
	
	\subsection{Interior and physical-boundary estimates}
	
	For fixed $\varepsilon>0$, $u$ is smooth in the interior and up to the
	initial Dirichlet face or the physical Neumann face away from their
	intersection by classical elliptic regularity.  No assertion is made
	here about the artificial outer Dirichlet face, which is disjoint from
	$K$.  The interior $C^{1,\alpha}$ estimate is
	\cite{DiBenedetto83,Tolksdorf84}, and the corresponding Dirichlet and
	conormal estimates are \cite[Theorems~1 and~2]{Lieberman88}.  Their
	estimates are stable as $\varepsilon\downarrow0$.  Together with the
	energy compactness and strict monotonicity argument described in the
	overview, they give the asserted local $C^1$ convergence.  The initial
	Dirichlet--Neumann edge is treated in the next subsection.

	We now obtain the $W^{1,2}$ estimate for the gradient magnitude without
	a boundary difference-quotient calculation.  In the remainder of this
	subsection, $K$ is a compact subset of the relative open set
	$D\setminus\Omega$; it may meet $\partial M$, but it stays away from
	both Dirichlet faces.  By
	\cite[Corollaries~2.5 and~2.6]{Han20}, the convex manifold
	$(M,d,\operatorname{vol}_g)$ is an $\operatorname{RCD}(0,3)$ space.
	Set
	\[
		\Psi_\varepsilon(s):=(s^2+\varepsilon^2)^{(p-2)/2},
		\qquad
		h_\varepsilon(s):=s\Psi_\varepsilon(s),
		\qquad 0<\varepsilon\le1.
	\]
	For $s>0$ we have
	\[
		\frac{s\Psi_\varepsilon'(s)}{\Psi_\varepsilon(s)}
		=(p-2)\frac{s^2}{s^2+\varepsilon^2}\in[p-2,0],
	\]
	and, because $p-2<0$, for every $s\ge1$,
	\[
		2^{(p-2)/2}s^{p-2}
		\le \Psi_\varepsilon(s)\le s^{p-2}.
	\]
	Thus the ellipticity and $p$-growth constants in
	Schulz--Violo~\cite[Theorem~1.4]{SV25} can be chosen independently of
	$\varepsilon\in(0,1]$.

	The homogeneous conormal condition on $\partial M$ implies that the
	weak equation for $u$ holds against every compactly supported test
	function in the relative open set $D\setminus\Omega$.  Equivalently,
	$u$ is a local weak solution on the metric-measure space $M$ also on
	balls meeting the physical boundary.  The regularized energy bound is
	uniform in $\varepsilon$, and
	\[
		h_\varepsilon(|\nabla u|)
		=(|\nabla u|^2+\varepsilon^2)^{(p-2)/2}|\nabla u|
		\le |\nabla u|^{p-1}.
	\]
	Consequently the averaged flux term on the right-hand sides of
	\cite[Theorem~1.4(i)--(ii)]{SV25} is locally bounded uniformly in
	$\varepsilon$.  Theorem~1.4 therefore gives, after covering $K$ by
	finitely many intrinsic balls,
	\begin{equation}\label{eq:SV-flux-gradient}
		\|h_\varepsilon(|\nabla u|)\|_{W^{1,2}(K)}\le C(K),
		\qquad
		\|\nabla u\|_{L^\infty(K)}\le M_K,
	\end{equation}
	where $C(K)$ and $M_K$ do not depend on $\varepsilon$.

	It remains to pass from the flux magnitude to the gradient magnitude.
	A direct differentiation gives
	\[
		h_\varepsilon'(s)
		=(s^2+\varepsilon^2)^{(p-4)/2}
		  \bigl(\varepsilon^2+(p-1)s^2\bigr)
		\ge(p-1)(s^2+\varepsilon^2)^{(p-2)/2}.
	\]
	For $0\le s\le M_K$ and $0<\varepsilon\le1$, the last expression is
	bounded below by
	\[
		(p-1)(M_K^2+1)^{(p-2)/2}>0.
	\]
	Hence $h_\varepsilon^{-1}$ is Lipschitz on
	$[0,h_\varepsilon(M_K)]$, with a Lipschitz constant independent of
	$\varepsilon$.  The Sobolev chain rule and
	\eqref{eq:SV-flux-gradient} now give
	\begin{equation}\label{eq:SV-gradient-magnitude}
		\bigl\||\nabla u|\bigr\|_{W^{1,2}(K)}\le C(K)
	\end{equation}
	uniformly for $0<\varepsilon\le1$.  This is the direct implication of
	the conclusion
	$\Psi(|\nabla u|)|\nabla u|\in W^{1,2}_{\mathrm{loc}}$ in
	Schulz--Violo's theorem that is used here.
	
	\subsection{The initial Dirichlet--Neumann edge}

	Let $E=\partial_{\operatorname{int}}\Omega\cap\partial M$ and fix
	$z\in E$.  In Fermi coordinates $(x,y)$ centred at $z$, with
	$y\ge0$ and $\partial M=\{y=0\}$, extend the metric and both the
	regularized and exact potentials evenly by
	\[
		\widetilde{g}(x,y)=g(x,|y|),
		\qquad
		\widetilde{u}_p^\varepsilon(x,y)=u_p^\varepsilon(x,|y|),
		\qquad
		\widetilde{u}_p(x,y)=u_p(x,|y|).
	\]
	Here the mixed components of $g$ vanish in Fermi coordinates.  The
	reflected metric is uniformly elliptic and Lipschitz, and the weak
	formulation of the equation, together with the vanishing weak
	conormal flux on $\{y=0\}$, shows that
	$\widetilde{u}_p^\varepsilon$ and $\widetilde{u}_p$ solve their respective
	weak equations across the reflecting face.  This is only a local weak
	reflection of the equations; no smooth metric double is being used.

	The orthogonality of the two boundary faces removes the corner after
	reflection.  Indeed, if the initial face is locally written as
	$x^1=f(x^2,y)$, then orthogonality gives
	$\partial_yf(x^2,0)=0$, so $f(x^2,|y|)$ is a $C^2$ defining graph for
	the reflected initial face.  We have therefore reduced the mixed
	problem near $z$ to a Dirichlet problem on a regular boundary.  In
	local coordinates the reflected regularized flux
	$\widetilde{A}_\varepsilon(x,\xi)$ satisfies, with constants independent
	of $\varepsilon\in(0,1]$,
	\begin{align*}
		c_p(|\xi|^2+\varepsilon^2)^{(p-2)/2}|\eta|^2
		&\le
		\langle D_\xi\widetilde{A}_\varepsilon(x,\xi)\eta,\eta\rangle
		\le C_p(|\xi|^2+\varepsilon^2)^{(p-2)/2}|\eta|^2,\\
		|\widetilde{A}_\varepsilon(x,\xi)
		-\widetilde{A}_\varepsilon(x',\xi)|
		&\le C_p|x-x'|(\varepsilon+|\xi|)^{p-1}.
	\end{align*}
	The constants also absorb the local ellipticity and Lipschitz bounds
	of $\widetilde{g}$.  Since
	$(|\xi|^2+\varepsilon^2)^{(p-2)/2}$ is comparable to
	$(\varepsilon+|\xi|)^{p-2}$, these are the structure conditions of
	\cite[Theorem~1]{Lieberman88} with $m=p-2\in(-1,0)$,
	$K=\varepsilon$, and spatial H\"older exponent $1$.  The exponent and
	estimate in that theorem are independent of $K$, so they are common
	to $\varepsilon\in(0,1]$.  A finite covering of the compact initial
	face gives
	a neighborhood $U_0$ such that
	\begin{equation}\label{eq:edge-uniform-C1}
		\sup_{0<\varepsilon\le1}
		\|u_p^\varepsilon\|_{C^{1,\beta}
		(\overline{U_0\setminus\Omega})}<\infty.
	\end{equation}
	The regularized energies give a uniform $W^{1,p}$ bound.  Strict
	monotonicity of the flux identifies every subsequential weak limit
	with the unique weak $p$-harmonic potential $u_p$ having the same
	mixed boundary data.  Hence \eqref{eq:edge-uniform-C1} and
	Arzel\`a--Ascoli upgrade the convergence of the whole family to
	$u_p^\varepsilon\to u_p$ in
	$C^1(\overline{U_0\setminus\Omega})$.  In particular
	$u_p\in C^{1,\beta}$ up to $E$.

	Since $u_p=1$ on $\partial_{\operatorname{int}}\Omega$ and
	$0<u_p<1$ in the exterior, the Hopf boundary-point lemma applied after
	reflection yields $|\nabla u_p|>0$ on the entire initial face,
	including $E$.  Compactness and the $C^1$ regularity then give a
	neighborhood $U$ of the initial face and $c_0>0$ such that
	$|\nabla u_p|\ge c_0$ on $U\setminus\Omega$.  The $C^1$ convergence
	above permits us to shrink $U$ so that
	$|\nabla u_p^\varepsilon|\ge c_0/2$ there for every sufficiently small
	$\varepsilon$.  The exact and regularized equations are therefore
	uniformly elliptic in this collar: their linearized matrices have
	common positive lower and upper ellipticity bounds, while the reflected
	coefficients are Lipschitz in the spatial variables.  Flattening the
	$C^2$ reflected Dirichlet boundary and applying the boundary
	difference-quotient estimate (the Dirichlet datum is constant) gives
	\begin{equation}\label{eq:edge-uniform-W22}
		\sup_{0<\varepsilon\le\varepsilon_0}
		\|u_p^\varepsilon\|_{W^{2,2}(U\setminus\Omega)}<\infty,
		\qquad u_p\in W^{2,2}(U\setminus\Omega);
	\end{equation}
	equivalently, one controls the
	tangential second derivatives first and uses the equation for the
	normal--normal derivative.  Because $u_p$ is bounded away from zero in
	$U$, the logarithmic change of variables transfers
	\eqref{eq:edge-uniform-C1}--\eqref{eq:edge-uniform-W22} to
	$w_p^\varepsilon$ and $w_p$.  In particular,
	\[
		w_p\in C^{1,\beta}\bigl(\overline{U\setminus\Omega}\bigr)
		\cap W^{2,2}(U\setminus\Omega),
		\qquad |\nabla w_p|\ge c
	\]
	for a possibly smaller $U$ and some $c>0$, while
	$|\nabla w_p^\varepsilon|\ge c/2$ and the stated uniform $W^{2,2}$
	bound hold for small $\varepsilon$.  This proves the edge assertions
	in items~\ref{item:C1} and~\ref{item:initial-edge}.
	
	\subsection{Conclusion}
	
	The comparison principle gives $u_p^\varepsilon\ge c_K>0$ on every
	compact $K\subset D\setminus\Omega$, uniformly in $\varepsilon$.  Since
	\[
		|\nabla w_p^\varepsilon|
		=(p-1)\frac{|\nabla u_p^\varepsilon|}{u_p^\varepsilon},
	\]
	the Sobolev product and chain rules, together with
	\eqref{eq:SV-flux-gradient}--\eqref{eq:SV-gradient-magnitude}, show that
	the family $\{|\nabla w_p^\varepsilon|\}$ is bounded in
	$W^{1,2}(K)$.  Every weakly convergent subsequence has
	$L^2$-limit $|\nabla w_p|$ by the already established
	$C^1_{\mathrm{loc}}$ convergence.  Thus the weak limit is unique and
	the whole family converges weakly to $|\nabla w_p|$ in
	$W^{1,2}_{\mathrm{loc}}$.  This proves item~\ref{item:W12} and the
	$C^{1,\beta}_{\mathrm{loc}}$ assertion in item~\ref{item:reg} of
	Lemma~\ref{lem:eps-regularization}.  Item~\ref{item:levels} follows
	from the Neumann condition as in the overview, and smoothness away from
	$\operatorname{Crit}(w_p)$ follows from uniform ellipticity there.

	\section{\texorpdfstring{$p$}{p}-nonparabolicity and decay of the capacity potential}
\label{sec:pdepart-appendix}

This appendix proves Proposition~\ref{prop:p-nonparabolic}.  The argument
uses only the doubling and Poincar\'e properties of a manifold with
nonnegative Ricci curvature and convex boundary, local Harnack estimates for
the natural Neumann problem, and a capacity estimate for balls centred at
arbitrary distant points.  In particular, the exhaustion domains introduced
below are not required to meet the physical boundary orthogonally.

Let $1<p<\infty$ and let $(M,\partial M,g)$ be a complete noncompact
Riemannian $n$-manifold with $\operatorname{Ric}\ge0$ and convex boundary.
We regard $M$ as a metric-measure space with its intrinsic distance and
Riemannian measure.  The normal $\mathbf{n}$ is the inward unit normal of
$\partial M$.

\subsection*{Analytic input at the physical boundary}

By Han's characterization~\cite[Corollaries~2.5 and~2.6]{Han20},
$(M,d,\mathcal{H}^n)$ is an $\operatorname{RCD}(0,n)$ space.  Bishop--Gromov
comparison and the weak $(1,1)$-Poincar\'e inequality on such spaces give,
for every metric ball and every $s>1$,
\begin{equation}\label{ap:doubling-poincare}
 |B(x,2r)|\le2^n|B(x,r)|,
 \qquad
 \fint_{B(x,r)}|f-f_{B(x,r)}|^s
 \le C_s r^s\fint_{B(x,\lambda_P r)}|\nabla f|^s,
\end{equation}
where $C_s=C(n,s)$ and $\lambda_P\ge1$ are fixed; see
\cite[Theorem~1.2]{Rajala12} for the weak $(1,1)$ estimate and~\cite{HK}
for its $L^s$ consequence.  These are estimates on the metric-measure space
$M$ itself and therefore include balls meeting $\partial M$.

\begin{lem}[Neumann Harnack inequality]\label{ap:local-estimates}
 There is a structural constant $\Lambda_H\ge1$ with the following
 property.  Let $B=B(x,r)$ and let $h\ge0$ be weakly $p$-harmonic in the
 relative ball $B(x,\Lambda_H r)\subset M$, in the sense that
 \begin{equation}\label{ap:local-weak-equation}
  \int_{B(x,\Lambda_H r)}
  |\nabla h|^{p-2}\langle\nabla h,\nabla\varphi\rangle
  \,\dd\mathcal{H}^n=0
 \end{equation}
 for every compactly supported $\varphi\in W^{1,p}$ on that relative ball.
 Then $h$ has a locally H\"older-continuous representative and
 \begin{equation*}
  \sup_{B(x,r)}h\le C_H\inf_{B(x,r)}h,
 \end{equation*}
 where $C_H=C_H(n,p)$.
\end{lem}
\begin{proof}
 Equation~\eqref{ap:local-weak-equation} says precisely that $h$ is a local
 minimizer of the $p$-energy on the relative metric ball.  If the ball meets
 $\partial M$, its test functions need not vanish on the physical boundary;
 this is the variational formulation of the homogeneous Neumann condition.
 Choose any $q\in(1,p)$.  The doubling property and the weak
 $(1,q)$-Poincar\'e estimate supplied by~\eqref{ap:doubling-poincare}
 place the problem in the standard doubling--Poincar\'e theory for
 $p$-energy minimizers.  Local H\"older continuity and the Harnack
 inequality follow from the regularity theory in~\cite{KS01}; enlarging
 the ball by its fixed factor is recorded in $\Lambda_H$.  No reflection
 of the metric, and no
 regularity of an auxiliary outer corner, is involved.
\end{proof}

\subsection*{Global capacity of distant balls}

For the two kinds of sets used below---smooth bounded obstacles and closed
metric balls---we use the following relaxed unnormalized capacity:
\begin{equation}\label{ap:global-capacity}
 \operatorname{cap}_p(K;M)
 =\inf\left\{\int_M|\nabla f|^p\,\dd\mathcal{H}^n:
 \begin{array}{l}
 f\in W^{1,p}(M),\ \operatorname{supp}f\text{ compact},\\
 f\ge1\text{ almost everywhere on }K
 \end{array}\right\}.
\end{equation}
Here $K$ has nonempty interior, and no trace is imposed on the physical
boundary.  We do not use this almost-everywhere convention for
lower-dimensional compact sets.  For a smooth compact
obstacle, the standard Sobolev approximation and truncation theorem shows
that the same infimum is obtained from smooth compactly
supported functions that are at least $1$ on the obstacle.  Conversely, if
such a smooth function $\psi$ is at least $1$ on the obstacle, then, for
$\varepsilon>0$, truncating
$\psi/(1-\varepsilon)$ at the levels $0$ and $1$ produces a Sobolev
competitor equal to $1$ on a neighborhood of the obstacle, with energy at
most $(1-\varepsilon)^{-p}$ times the exterior energy of $\psi$.  Letting
$\varepsilon\downarrow0$ proves the equivalence with the convention in
Section~2.  Since the truncated function is constant on the obstacle, its
energy is supported in the exterior.  Consequently, when $1<p<3$, for
every smooth obstacle $\Omega$ considered in Section~2,
\begin{equation}\label{ap:capacity-normalization}
 C_p(\partial_{\operatorname{int}}\Omega)
 =\frac1{4\pi}\left(\frac{p-1}{3-p}\right)^{p-1}
 \operatorname{cap}_p(\Omega;M).
\end{equation}

\begin{lem}[Capacity of distant balls]\label{ap:ball-capacity}
 Suppose that $|B(o,t)|\ge c_0t^\beta$ for $t\ge t_0$, where $\beta>p$.
 Fix $\kappa\in(0,1)$.  There are constants $c_1>0$ and $R_1$ such that,
 whenever $R=d(o,x)\ge R_1$ and $r=\kappa R$,
 \begin{equation}\label{ap:ball-capacity-bound}
  \operatorname{cap}_p(\overline{B}(x,r);M)\ge c_1r^{\beta-p}.
 \end{equation}
 In particular, every metric ball of positive radius has positive global
 $p$-capacity.
\end{lem}
\begin{proof}
 Fix $x$ as in the statement.  If $t\ge r$, then $R\le\kappa^{-1}t$ and hence
 \[
  B(o,t)\subset B(x,R+t)\subset B(x,(1+\kappa^{-1})t).
 \]
 Iterating the doubling inequality a number of times depending only on
 $\kappa$ gives
 \begin{equation}\label{ap:translated-volume-lower}
  |B(x,t)|\ge c(n,\kappa)|B(o,t)|\ge c_2t^\beta,
  \qquad t\ge r,
 \end{equation}
 after increasing $R_1$ so that $r\ge t_0$.

 Let $f$ be admissible in~\eqref{ap:global-capacity}.  Set
 \[
  B_j=B(x,2^jr),\qquad
  f_j=\fint_{B_j}f,\qquad
  \mathcal{E}(f)=\int_M|\nabla f|^p\,\dd\mathcal{H}^n.
 \]
 Since $f\ge1$ almost everywhere on $B(x,r)$, we have $f_0\ge1$.
 Since $f$ has
 compact support and $|B_j|\to\infty$, H\"older's inequality gives
 $f_j\to0$.

 By doubling, H\"older's inequality, and the Poincar\'e inequality on
 $B_{j+1}$ (with its fixed enlargement if $\lambda_P>1$),
 \begin{align*}
  |f_j-f_{j+1}|
  &\le C\fint_{B_{j+1}}|f-f_{j+1}|\\
  &\le C\,2^jr
  \left(\fint_{B(x,\lambda_P2^{j+1}r)}|\nabla f|^p\right)^{1/p}\\
  &\le C(2^jr)^{1-\beta/p}\mathcal{E}(f)^{1/p},
 \end{align*}
 where the last line follows from
 \eqref{ap:translated-volume-lower}, applied at the enlarged radius.
 Telescoping from $f_0$ to $0$ and using $\beta>p$, we obtain
 \begin{align*}
  1
  &\le |f_0|
   \le\sum_{j=0}^\infty|f_j-f_{j+1}|\\
  &\le C\mathcal{E}(f)^{1/p}
       \sum_{j=0}^\infty(2^jr)^{1-\beta/p}
   \le Cr^{1-\beta/p}\mathcal{E}(f)^{1/p}.
 \end{align*}
 Raising this inequality to the power $p$ and then taking the infimum over
 $f$ proves~\eqref{ap:ball-capacity-bound}.

 For a fixed ball $B(x,r_0)$, the same calculation starts at $r_0$.
 There are only finitely many radii below
 $\max\{t_0,2d(o,x)\}$; the corresponding positive volumes contribute
 finitely many terms.  At all larger radii the translated lower volume
 estimate applies.  The complete series is therefore finite, which gives
 $\operatorname{cap}_p(\overline{B}(x,r_0);M)>0$.
\end{proof}

\subsection*{Exhaustion potentials and decay}

\begin{cor}[Application to Proposition~\ref{prop:p-nonparabolic}]
\label{ap:application}
 Let $n=3$, $p\in(1,2)$, and suppose that
 $|B(o,r)|\ge cr^{1+\alpha}$ for all sufficiently large $r$, where
 $\alpha>1$.  Then $(M,\partial M)$ is $p$-nonparabolic.  For every
 bounded smooth obstacle $\Omega$ as in Section~2, the mixed capacity
 potential $u_p$ exists, tends to zero at infinity, and satisfies
 \begin{equation*}
  u_p(x)\le C|x|^{-\frac{\alpha+1-p}{p-1}}
 \end{equation*}
 for all sufficiently large $|x|=d(o,x)$.
\end{cor}
\begin{proof}
 Put $\beta=1+\alpha$.  Since $\alpha>1$ and $p<2$, we have $\beta>p$.
 Choose a closed ball contained in
 $\operatorname{int}(M)\cap\operatorname{int}(\Omega)$.  This ball has
 positive capacity by Lemma~\ref{ap:ball-capacity}; capacity monotonicity
 and~\eqref{ap:capacity-normalization} then give
 \begin{equation*}
  C_p(\partial_{\operatorname{int}}\Omega)>0.
 \end{equation*}

 To obtain the exhaustion, choose a smooth nonnegative proper exhaustion
 function $\rho$ on $M$.
 By Sard's theorem, one can choose $c_j\to\infty$ that are regular values
 both of $\rho$ and of $\rho|_{\partial M}$.  Fix a compact connected
 neighborhood $K\supset\Omega$ and take the component of
 $\{\rho<c_j\}$ containing $K$.  For $c_1$ sufficiently large these
 components form nested bounded Lipschitz domains $D_j$ exhausting $M$,
 and their smooth artificial boundaries are transverse to $\partial M$.
 Orthogonality at the outer edge is not required.  On
 $D_j\setminus\Omega$, let $v_j$ minimize the $p$-energy among functions
 whose trace is $1$ on $\partial_{\operatorname{int}}\Omega$ and $0$ on
 $\partial_{\operatorname{int}}D_j$; no trace is prescribed on the
 physical boundary.  The direct method gives a unique minimizer, whose
 Euler--Lagrange equation is the mixed weak problem with the natural
 homogeneous Neumann condition on $\partial M$.

 Extend $v_j$ as $1$ on $\Omega$ and as $0$ on $M\setminus D_j$, and set
 \[
  C_j=\int_M|\nabla v_j|^p\,\dd\mathcal{H}^3.
 \]
 The extensions belong to $W^{1,p}(M)$ because the prescribed traces agree
 across both artificial faces.  Truncation and comparison give
 $0\le v_j\le1$ and $v_j\le v_{j+1}$.  Moreover,
 \begin{equation}\label{ap:exhaustion-capacity-limit}
  C_j\downarrow\operatorname{cap}_p(\Omega;M).
 \end{equation}
 Indeed, enlarging the outer domain enlarges the admissible class, so
 $C_{j+1}\le C_j$, and every extended $v_j$ is admissible for the global
 capacity, so $C_j\ge\operatorname{cap}_p(\Omega;M)$.  Conversely, after
 truncation at the levels $0$ and $1$, any compactly supported global
 competitor equals $1$ almost everywhere on $\Omega$, is supported in $D_j$
 for all sufficiently large $j$, and is then admissible for the $j$-th
 problem.
 Taking first the infimum over $j$ and then over global competitors proves
 the reverse inequality in the limit.

 We next establish the energy identity used in the capacity comparison.
 For $0<a<1$, the Sobolev function
 \[
  \varphi=\min\{v_j,a\}-av_j
 \]
 has zero trace on both Dirichlet faces and is an admissible test function
 in the weak equation.  The Sobolev chain rule gives
 \[
  \nabla\varphi=
  \begin{cases}
   (1-a)\nabla v_j,&v_j<a,\\
   -a\nabla v_j,&v_j>a,
  \end{cases}
 \]
 almost everywhere; moreover, $\nabla v_j=0$ almost everywhere on the
 plateau $\{v_j=a\}$.  Therefore
 \[
  0=(1-a)\int_{\{v_j<a\}}|\nabla v_j|^p
    -a\int_{\{v_j>a\}}|\nabla v_j|^p.
 \]
 Adding $a\int_{\{v_j<a\}}|\nabla v_j|^p$ to both sides and using the
 definition of $C_j$ yields
 \begin{equation}\label{ap:truncated-strip-energy}
  \int_{\{v_j<a\}}|\nabla v_j|^p\,\dd\mathcal{H}^3=aC_j.
 \end{equation}

 Fix $x\in M$ with $R=d(o,x)$ sufficiently large and put
 $\kappa=(4\Lambda_H)^{-1}$ and $r=\kappa R$.
 Choose $j$ so large that the relative ball
 $\overline{B}(x,\Lambda_H r)$ is contained in $D_j$ and is disjoint from
 $\Omega$.  Thus this ball does not meet either artificial Dirichlet face.
 On the ball, $v_j$ is a nonnegative local $p$-energy minimizer; if the ball
 meets $\partial M$, its natural boundary condition is precisely the
 Neumann condition in Lemma~\ref{ap:local-estimates}.  Consequently,
 \begin{equation}\label{ap:harnack-at-x}
  v_j(x)\le C_Hm_j,
  \qquad m_j=\inf_{B(x,r)}v_j.
 \end{equation}
 Lemma~\ref{ap:local-estimates} also supplies the continuous representative
 used in this definition.

 If $m_j=0$, estimate~\eqref{ap:harnack-at-x} gives $v_j(x)=0$.  Suppose
 $m_j>0$ and fix $0<a<m_j$.  Continuity and compactness of
 $\overline{B}(x,r)$ show that $v_j>a$ on a neighborhood of this closed
 ball.  Hence the compactly supported Sobolev function
 \[
  f_{j,a}=\min\{v_j/a,1\}
 \]
 is admissible for
 $\operatorname{cap}_p(\overline{B}(x,r);M)$.  By
 \eqref{ap:truncated-strip-energy},
 \begin{align*}
  \operatorname{cap}_p(\overline{B}(x,r);M)
  &\le\int_M|\nabla f_{j,a}|^p\,\dd\mathcal{H}^3\\
  &=a^{-p}\int_{\{v_j<a\}}|\nabla v_j|^p\,\dd\mathcal{H}^3
   =a^{1-p}C_j.
 \end{align*}
 Letting $a\uparrow m_j$ gives
 \begin{equation}\label{ap:minimum-capacity-bound}
  m_j^{p-1}\le
  \frac{C_j}{\operatorname{cap}_p(\overline{B}(x,r);M)}.
 \end{equation}
 Since $C_j\le C_1$ and $r=\kappa R$, Lemma~\ref{ap:ball-capacity},
 \eqref{ap:harnack-at-x}, and~\eqref{ap:minimum-capacity-bound} imply
 \begin{equation}\label{ap:global-exhaustion-decay}
  v_j(x)
  \le C_H\left(
  \frac{C_j}{\operatorname{cap}_p(\overline{B}(x,r);M)}
  \right)^{1/(p-1)}
  \le CR^{-\frac{\beta-p}{p-1}}.
 \end{equation}
 The constant is independent of $x$ and of all sufficiently large $j$.

 We now pass to the global potential.  The monotone sequence $v_j$ is
 bounded between $0$ and $1$, so it has a pointwise limit $u_p$.  On every
 relative ball compactly contained in
 $\overline{M\setminus\Omega}$ and disjoint from the fixed inner Dirichlet
 face, the functions $v_j$ are local $p$-energy minimizers for all large
 $j$.  The local H\"older estimates for minimizers are uniform on smaller
 relative balls.  Arzel\`a--Ascoli and the monotone pointwise convergence
 therefore give local uniform convergence to $u_p$.  Caccioppoli's
 inequality gives weak $W^{1,p}$ compactness, and the standard
 strict-monotonicity argument for $z\mapsto|z|^{p-2}z$ upgrades this to
 strong $W^{1,p}_{\mathrm{loc}}$ convergence.  This is the usual Harnack
 convergence theorem for $p$-energy minimizers in the setting of
 \cite{KS01}.
 Hence the weak equation holds first for test functions supported a positive
 distance from the inner face.  By the standard density theorem for a
 Lipschitz domain with a designated Dirichlet face, functions with zero
 trace on that face can be approximated in $W^{1,p}$ by such test
 functions.  Since
 $|\nabla u_p|^{p-2}\nabla u_p\in L^{p/(p-1)}$, passage through this
 approximation gives
 \begin{equation}\label{ap:global-weak-equation}
  \int_{M\setminus\Omega}|\nabla u_p|^{p-2}
  \langle\nabla u_p,\nabla\varphi\rangle\,\dd\mathcal{H}^3=0
 \end{equation}
 for every compactly supported test function having zero trace on the inner
 Dirichlet face.  Because such test functions are free on $\partial M$,
 \eqref{ap:global-weak-equation} includes the homogeneous Neumann condition.

 It remains to identify the inner trace.  Fix a bounded Lipschitz
 neighborhood of the inner face.  The uniform energy bound and
 $0\le v_j\le1$ give weak $W^{1,p}$ compactness on this neighborhood.
 Pointwise convergence and $0\le v_j\le1$ also imply, by dominated
 convergence on this bounded neighborhood, that $v_j\to u_p$ strongly in
 $L^p$.  Thus every weak $W^{1,p}$ cluster point is the same function
 $u_p$.  The trace operator is bounded and linear, hence weakly continuous,
 and
 every $v_j$ has trace $1$ on
 $\partial_{\operatorname{int}}\Omega$.  It follows that $u_p$ has trace
 $1$ there.  This identifies only the Sobolev trace; the stronger
 pointwise and differentiable regularity at the fixed orthogonal
 Dirichlet--Neumann edge is supplied later by
 Lemma~\ref{lem:eps-regularization} and is not needed in the present
 exhaustion argument.
 Passing to the limit in~\eqref{ap:global-exhaustion-decay} gives
 \[
  u_p(x)\le C|x|^{-\frac{\beta-p}{p-1}}
  =C|x|^{-\frac{\alpha+1-p}{p-1}}.
 \]
 In particular, $u_p(x)\to0$ along every sequence with $|x|\to\infty$.
 To see uniqueness, let $u$ and $v$ be two solutions with these boundary
 data.  Given $\varepsilon>0$, their decay gives a compact set outside
 which $|u-v|\le\varepsilon$.  On a bounded domain containing this compact
 set, the comparison principle applied to $u$ and $v+\varepsilon$, and then
 to $v$ and $u+\varepsilon$, gives $|u-v|\le\varepsilon$.  Letting
 $\varepsilon\downarrow0$ proves $u=v$.

 Finally, weak lower semicontinuity on each compact member of an exhaustion,
 followed by monotone convergence of the domains and
 \eqref{ap:exhaustion-capacity-limit}, gives
 \begin{equation}\label{ap:energy-upper-capacity}
  \int_{M\setminus\Omega}|\nabla u_p|^p\,\dd\mathcal{H}^3
  \le\operatorname{cap}_p(\Omega;M).
 \end{equation}
 Conversely, fix $0<a<1$.  The decay just proved makes
 $(u_p-a)_+/(1-a)$ compactly supported.  It belongs to $W^{1,p}(M)$ after
 being extended as $1$ on $\Omega$, and it is admissible in
 \eqref{ap:global-capacity}.  Therefore
 \[
  \operatorname{cap}_p(\Omega;M)
  \le(1-a)^{-p}
  \int_{\{u_p>a\}}|\nabla u_p|^p\,\dd\mathcal{H}^3.
 \]
 The strong maximum principle gives $u_p>0$ on $M\setminus\Omega$.
 Letting $a\downarrow0$ and using monotone convergence proves the reverse
 inequality in~\eqref{ap:energy-upper-capacity}.  Thus
 \[
  \operatorname{cap}_p(\Omega;M)
  =\int_{M\setminus\Omega}|\nabla u_p|^p\,\dd\mathcal{H}^3,
 \]
 and~\eqref{ap:capacity-normalization} gives exactly the normalized energy
 identity stated in Section~2.
\end{proof}

\section{A non-flat Ricci-pinched manifolds with boundary}
	\label{sec:eps-appendix-c}
    In this appendix, we construct a complete Ricci-pinched $3$-manifold with convex boundary and Euclidean volume growth, but not second-fundamental-form-pinched.
    \begin{example}\label{ex:convex-boundary-counterexample}
		Let $\delta$ be the Euclidean metric on $\mathbb R^3$, and set
		\[
			M=\{(x,y,z):x\ge\sqrt{1+y^2+z^2}\},
			\qquad g=e^{2u}\delta,
			\qquad u=-\frac{1}{8r^4},
		\]
		where $r^2=x^2+y^2+z^2$.  Then $(M,g)$ is complete, has cubic
		volume growth and strictly convex boundary, and satisfies
		$\operatorname{Ric}_g\ge\frac18\mathrm R_g g>0$.
		Indeed, $r\ge x\ge 1$ and $e^{-1/4}\delta\le g\le\delta$ on the closed
		convex domain $M$, so its intrinsic distance is complete.  Moreover,
		$M$ is asymptotic to the solid Euclidean cone
		$\{x\ge\sqrt{y^2+z^2}\}$, and $e^u\to1$ at infinity; hence
		$(M,g)$ has cubic volume growth.

		The conformal change formula
		\[
			\operatorname{Ric}_g=-D^2u+du\otimes du
			-(\Delta_\delta u+|Du|_\delta^2)\delta
		\]
		yields the radial and tangential Ricci eigenvalues with respect to
		$g$,
		\[
			\mu_r=4e^{-2u}r^{-6},
			\qquad \mu_t=e^{-2u}\bigl(r^{-6}-\tfrac14r^{-10}\bigr),
		\]
		with $\mu_t$ of multiplicity two.  Since $r\ge1$,
		$\mu_t\ge\frac34e^{-2u}r^{-6}$ and
		$\mathrm R_g=\mu_r+2\mu_t\le6e^{-2u}r^{-6}$, proving the
		claimed Ricci pinching and strict positivity.

		On $\partial M$, the Euclidean inward unit normal is
		$\mathbf n_\delta=(x,-y,-z)/r$, the Euclidean principal curvatures
		are $r^{-3}$ and $r^{-1}$, and
		$\mathbf n_\delta(u)=1/(2r^7)$.  Thus
		$\operatorname{II}_g=e^u(\operatorname{II}_\delta-
		\mathbf n_\delta(u)\delta_{\partial M})$ gives the principal
		curvatures
		\[
			\kappa_1=e^{-u}\bigl(r^{-3}-\tfrac12r^{-7}\bigr)>0,
			\qquad
			\kappa_2=e^{-u}\bigl(r^{-1}-\tfrac12r^{-7}\bigr)>0.
		\]
		The boundary is therefore strictly convex, but
		$\kappa_1/(\kappa_1+\kappa_2)\to0$ as $r\to\infty$, so it
		satisfies no uniform second-fundamental-form pinching.
	\end{example}

\end{document}